\documentclass[a4paper,reqno]{amsart}

\usepackage[T1]{fontenc}
\usepackage[utf8]{inputenc}
\usepackage{mathtools}
\usepackage{thmtools, thm-restate}
\usepackage{amssymb}
\usepackage{hyperref}
\usepackage[capitalize]{cleveref}
\usepackage{tikz}
\usepackage{bbm}
\usepackage[shortlabels]{enumitem}
\usepackage{stmaryrd} 
\usepackage{pgfplots}
\usepackage[ruled,vlined,linesnumbered]{algorithm2e}

\usepackage{a4wide}
\usepackage[margin=1in]{geometry}
\pgfplotsset{compat=1.15}
\usepgfplotslibrary{fillbetween}

\renewcommand{\le}{\leqslant}
\renewcommand{\ge}{\geqslant}
\renewcommand{\leq}{\leqslant}
\renewcommand{\geq}{\geqslant}

\declaretheorem[numberwithin=section]{theorem}
\declaretheorem[sibling=theorem]{lemma}
\declaretheorem[sibling=theorem]{claim}
\declaretheorem[sibling=theorem]{fact}

\declaretheorem[sibling=theorem]{corollary}
\declaretheorem[sibling=theorem]{proposition}
\declaretheorem[style=definition,numbered=no]{definition}
\declaretheorem[style=remark,sibling=theorem,numbered=no]{remark}

\crefname{claim}{Claim}{Claims}
\crefname{fact}{Fact}{Facts}

\DeclareMathOperator{\Var}{Var}

\renewcommand{\Pr}{\mathbb{P}}
\newcommand{\Qr}{\mathbb{Q}}
\newcommand{\Ex}{\mathbb{E}}

\newcommand{\Prh}{\hat{\Pr}}
\newcommand{\Exh}{\hat{\Ex}}

\newcommand{\1}{\mathbbm{1}}
\newcommand{\Ber}{\mathrm{Ber}}
\newcommand{\Po}{\mathrm{Po}}

\newcommand{\kap}{\text{$k$-AP}}

\DeclarePairedDelimiter{\br}{\llbracket}{\rrbracket}

\DeclarePairedDelimiter{\floor}{\lfloor}{\rfloor}
\newcommand{\eps}{\varepsilon}
\renewcommand{\emptyset}{\varnothing}
\newcommand{\cA}{\mathcal{A}}
\newcommand{\cB}{\mathcal{B}}
\newcommand{\cC}{\mathcal{C}}
\newcommand{\cD}{\mathcal{D}}
\newcommand{\cE}{\mathcal{E}}
\newcommand{\cJ}{\mathcal{J}}
\newcommand{\cF}{\mathcal{F}}
\newcommand{\cG}{\mathcal{G}}
\newcommand{\cH}{\mathcal{H}}

\newcommand{\cL}{\mathcal{L}}

\newcommand{\cP}{\mathcal{P}}

\newcommand{\cR}{\mathcal{R}}
\newcommand{\cS}{\mathcal{S}}
\newcommand{\cT}{\mathcal{T}}

\newcommand{\cV}{\mathcal{V}}
\newcommand{\cX}{\mathcal{X}}

\newcommand{\fC}{\mathfrak{C}}
\newcommand{\fB}{\mathfrak{B}}

\newcommand{\NN}{\mathbb{N}}
\newcommand{\Nats}{\mathbb{N}}
\newcommand{\ZZ}{\mathbb{Z}}
\newcommand{\RR}{\mathbb{R}}

\newcommand{\leH}{\preceq}
\newcommand{\APk}{\mathrm{AP}_k}

\newcommand{\bA}{\mathbf{A}}
\newcommand{\bR}{\mathbf{R}}

\newcommand{\va}{\mathbf{a}}
\newcommand{\vb}{\mathbf{b}}
\newcommand{\vc}{\mathbf{c}}
\newcommand{\vell}{\mathbf{l}}
\newcommand{\vL}{\mathbf{L}}

\newcommand{\DKL}{D_{\mathrm{KL}}}

\newcommand{\cSb}{\cS_{\mathrm{small}}}

\newcommand{\ff}[2]{{#1}^{\underline{#2}}}
\newcommand{\ffAPk}[1]{\APk^{\underline{#1}}}

\title{Upper tails for arithmetic progressions revisited}

\thanks{This research was supported by: the Israel Science Foundation grant 2110/22; the grant 2019679 from the United States--Israel Binational Science Foundation (BSF) and the United States National Science Foundation (NSF); and the ERC Consolidator Grant 101044123 (RandomHypGra).}

\author{Matan Harel}
\address{Department of Mathematics, Northeastern University, Boston, MA, USA}
\email{m.harel@northeastern.edu}
\author{Frank Mousset}
\email{moussetfrank@gmail.com}
\author{Wojciech Samotij}
\address{School of Mathematical Sciences, Tel Aviv University, Tel Aviv 6997801, Israel}
\email{samotij@tauex.tau.ac.il}

\date{\today}

\begin{document}

\begin{abstract}
  Let $X$ be the number of $k$-term arithmetic progressions contained in the
  $p$-biased random subset of the first $N$ positive integers.  We give
  asymptotically sharp estimates on the logarithmic upper-tail probability
  $\log \Pr(X \ge \Ex[X] + t)$ for all $\Omega(N^{-2/k}) \le p \ll 1$ and all $t \gg
  \sqrt{\Var(X)}$, excluding only a few boundary cases. In particular, we show
  that the space of parameters $(p,t)$ is partitioned into three
  phenomenologically distinct regions, where the upper-tail probabilities
  either resemble those of Gaussian or Poisson random variables, or are
  naturally described by the probability of appearance of a small set that
  contains nearly all of the excess $t$ progressions. We employ a variety of
  tools from probability theory, including classical tilting arguments and
  martingale concentration inequalities.  However, the main technical
  innovation is a combinatorial result that establishes a stronger version of
  `entropic stability' for sets with rich arithmetic structure.
\end{abstract}

\maketitle

\section{Introduction}
\label{sec:Intro}

Let $k\ge 3$ and $N$ be positive integers. We write $\APk$ for the set of
$k$-term arithmetic progressions ($k$-APs for short) in the set $\br{N}
\coloneqq \{1, 2, \dotsc, N\}$, that is, $\APk$ is the collection of
$k$-element subsets of $\br N$ of the form $\{a,a+b,a+2b,\dotsc,a+(k-1)b\}$,
where $a$ and $b$ are positive integers.\footnote{
  Clearly, $\APk$ depends on $N$. However, as the meaning of $N$ will be same
  throughout the paper, we will often omit the explicit mention of $N$ in the
  various notations. Similarly, we will omit mentioning the dependence on $k$
  whenever it seems safe to do so.
}
Given $p\in [0,1]$, we may choose a random subset of $\br N$ by including each
number independently with probability $p$. We write $\bR$ for the random set
obtained in this way and let $X$ be the number of elements of $\APk$ that
are contained in $\bR$.

The goal of this work is to calculate the asymptotic behaviour, as $N$ tends to
infinity and $k$ is fixed, of the logarithmic upper-tail probability of $X$, in 
the sparse regime (that is, we assume throughout the paper that $p$ vanishes as
$N$ grows).  To be more precise, our goal is to compute the asymptotic rate of
$\log \Pr (X \ge \Ex[X] + t)$ for all (well-behaved) sequences $t$.  For
notational convenience, we set $\mu \coloneqq \Ex[X]$ and $\sigma^2
\coloneqq \Var(X)$.

There are a few cases which are straightforward. First, if $\mu + t$ is greater
than $|\APk|$, the maximal number of $k$-term arithmetic progressions that can
possibly be contained in $\bR$, then the event $\{X \ge \mu +t\}$ is empty, and the
logarithmic upper-tail probability is negative infinity. If
$\mu$ is bounded, then $X$ is asymptotically Poisson (see, e.g.,
~\cite{BarKocLiu19}), which answers the question when $t$ is also bounded.
Furthermore, a sequence of now-classical works from the 1980s (see, for
example, \cite{BarKarRuc89,Ruc88}) implies that $X$ satisfies a Central Limit
Theorem; i.e., whenever $\mu \to \infty$, then $(X- \mu)/\sigma$
converges weakly to a standard Gaussian random variable, which answers the
question in the case where $t/\sigma$ is bounded. Unfortunately, when $t/\sigma
\to \infty$, this result only tells us that $\Pr(X \ge \mu +t)$ vanishes and
cannot be used directly to deduce any quantitative information on the rate of
convergence. For values of $p$ which vanish sufficiently slowly, it is possible
to prove Berry--Esseen-like bounds on the rate of convergence via Stein's
method (see, e.g., \cite{ross2011fundamentals}); when they are available, such
bounds can be leveraged to prove Gaussian behaviour if $t/ \sigma \to \infty$
very slowly. Such techniques will not be sufficient to prove Gaussian bounds
for a vast majority of the Gaussian regime that will be discussed in this
paper.

The remaining regimes of the upper-tail problem can be divided into three
cases: the case where $t$ is much smaller than $\mu$ (but much larger
than $\sigma$), known as the \emph{moderate-deviation regime}; the case where $t$ is
commensurate with $\mu$, known as the \emph{large-deviation regime}; and the
case where $t$ is much greater than $\mu$, which has received comparatively
little attention, that we will term the \emph{extreme-deviation regime}.

Historically, the
large-deviation regime has been the most studied one. The
main reason for this is that the upper bounds on the logarithmic upper-tail
probability of $X$ that can be proved via classical concentration inequalities
do not match the known lower bounds, even up to constant factors;
see~\cite{janson2002infamous} for a survey of such results.  A breakthrough was
achieved by the work of Chatterjee--Dembo~\cite{chatterjee2016nonlinear}, which
established a large-deviation principle for a wide class of non-linear
functions of independent random variables.  Their result was later extended and
generalised by Eldan~\cite{eldan2018gaussian},
Augeri~\cite{augeri2018nonlinear}, and Austin~\cite{austin2018structure}.
Subsequently, Bhattacharya--Ganguly--Shao--Zhao~\cite{BhaGanShaZha20}
showed that these large-deviation principles apply in the context of $k$-APs
and solved the associated variational problem to obtain asymptotically tight
estimates for the logarithmic upper-tail probability, for a suboptimal range of
the density parameter $p$.  Around the same time, Warnke~\cite{warnke2017upper}
developed a sophisticated moment-based approach in order to prove bounds on the
logarithmic upper-tail probability that were correct only up to a
multiplicative constant factor, but held in the entire large-deviation regime.
Finally, the three authors~\cite{harel2022upper} determined the asymptotic
logarithmic upper-tail probability in the entire large-deviation regime using a
combinatorial approach paired with a conditioned high-moment calculation.

The moderate-deviation regime has also garnered some recent
attention.  In particular, the aforementioned results of both
Bhattacharya--Ganguly--Shao--Zhao~\cite{BhaGanShaZha20} and
Warnke~\cite{warnke2017upper} extend to portions of this regime.  As before,
the results of~\cite{BhaGanShaZha20} determine the exact asymptotics
whereas \cite{warnke2017upper} computes only the order of magnitude.  Both
results hold under strong assumptions on the density $p$; moreover,
\cite{BhaGanShaZha20} further requires the deviation $t$ not to be too
far from the expectation. The recent work of Griffiths, Koch, and
Secco~\cite{griffiths2020deviation} determines exact asymptotics of the
logarithmic upper-tail probability in a substantially larger, but still
incomplete, portion of the moderate-deviation regime (see also~\cite{FizGriSecSer22},
where a similar result is obtained in the setting of $k$-term arithmetic
progressions modulo a prime).

Finally, although the extreme-deviation regime is not explicitly mentioned in
most of the above works, many of the arguments can be extended to
cases where $t$ is much larger than $\mu$, except when $\mu$ is only
polylogarithmic in $N$.

\subsection{Main results}
\label{sec:main-results}

Our main contribution is to determine the asymptotic rate of the logarithmic
upper-tail probability $\log \Pr(X \ge \mu + t)$ for all values of $p$ and $t$,
with the exception of a few liminal cases and the regime $p = \Theta(1)$.
To state the results, we require a few preliminaries.  First, it
is straightforward to verify that, for some positive $C = C(k)$,
\begin{equation}\label{eq:MuSigma2Comp}
  \mu = (1 + o(1)) \cdot \frac{ N^2 p^{k}}{2(k-1)}
  \qquad
  \text{ and }
  \qquad
  \sigma^2 = (1+o(1)) \cdot \frac{N^2 p^k}{2(k-1)}( 1 + C N p^{k-1}).
\end{equation}
We also define the function
\[
  \Po(x) \coloneqq  \int_0^x \log(1+y)\,dy = (1+ x) \log(1+x) - x,
\]
which naturally appears in the rate function of Poisson random variables.
Finally, given a $U \subseteq \br{N}$,  we set $\Ex_{U}[X] \coloneqq \Ex[X \mid
U \subseteq \bR]$, and, for any $t \ge 0$, we define
\begin{equation} \label{eq:PsiDef}
  \Psi(t) = \Psi_{N,p,k}(t) \coloneqq \min{}\{ |U| : \Ex_U[X] \ge \mu + t \},
\end{equation}
with the convention that $\Psi(t) = \infty$ if the set being optimised over is empty.

\begin{definition}
  We say that the sequence $(p,t)$ is in:
  \begin{itemize}
  \item
    the {\em Gaussian regime} if
    \[
      N^{-1/(k-1)}\ll p \ll 1, \quad t  \gg \sigma,
      \quad \text{ and } \quad
      \sqrt{t} \log (1/p) \gg t^2/\sigma^2;
    \]
  \item
    the {\em Poisson regime} if
    \[
      \Omega(N^{-2/k})\leq p \ll N^{-1/(k-1)},\quad t \gg \sigma,
      \quad \text{ and } \quad
      \sqrt{t} \log (1/p) \gg \mu \cdot \Po(t/\mu);
    \]
  \item
    the {\em localised regime} if either
    \[
      N^{-1/(k-1)}< p \ll 1 \quad\text{ and } \quad \sqrt{t} \log (1/p) \ll t^2/\sigma^2,
    \]
    or \[
      \Omega(N^{-2/k}) \leq p \leq N^{-1/(k-1)}
      \quad \text{ and } \quad  \sqrt{t} \log (1/p) \ll
      \mu \cdot \Po(t/\mu).
    \]
  \end{itemize}
\end{definition}

The three regimes are depicted in Figure~\ref{fig:PhaseDiagram}, together with
a fourth regime where $t/\sigma$ is bounded and the Central Limit Theorem applies.

\begin{theorem}\label{thm:MainResult}
  Assume $k\ge 3$ and let $X$ be the number of $k$-term arithmetic
  progressions contained in the random subset of $\br N$ obtained by including
  each number independently with probability $p$.
  \begin{itemize}
  \item If $(p,t)$ is in the Gaussian regime, then
    \[
      - \log \Pr(X \ge \mu + t) = (1+o(1)) \cdot \frac{t^2}{2 \sigma^2}.
    \]
  \item If $(p, t)$ is in the Poisson regime, then
    \[
      - \log \Pr(X  \ge \mu + t) = (1+o(1)) \cdot \mu \cdot \Po(t/\mu).
    \]
  \item If $(p,t)$ is in the localised regime, then
    \[
      - \log \Pr(X \ge \mu + t) = (1+o(1)) \cdot  \Psi(t) \cdot \log(1/p);
    \]
    moreover, if $\mu + t\leq |\APk|$, then $\Psi(t) = (1+o(1))\cdot \sqrt{2(k-1)t}$.
  \end{itemize} 
\end{theorem}

Heuristically, one may think of three different strategies to increase the
number of $k$-term arithmetic progressions by $t$.  First, we may add to the
$p$-biased random set $\bR$ a small but highly structured subset that contains
the $t$ excess arithmetic progressions: this leads to the localised regime.
The other two strategies are more `global', in the sense that the excess
arithmetic progressions are spread out roughly evenly over $\br{N}$. We can do
this either by increasing the probability of the events $\{i \in \bR\}$ in a
roughly uniform fashion (this leads to the Gaussian regime), or by
superimposing $\bR$ and the union of $t$ distinct, arithmetic progressions
chosen uniformly at random (this leads to the Poisson regime). One can provide
a convincing heuristic calculation that associates to each strategy the
respective quantitative bound in \cref{thm:MainResult}; indeed, in each case,
the rate function is precisely the Kullback--Leibler divergence of the random
set obtained by applying the corresponding strategy from the original random
set $\bR$. (Having said that, turning these intuitions into rigorous arguments
requires some work.) It is straightforward to check that the three regimes are
the regions where the respective strategy is the `cheapest', in the sense of
leading to the smallest rate function. In light of this, the main contribution
of \cref{thm:MainResult} is to show that, away from the boundary between the
regimes, one of these three strategies will always dominate the upper-tail
event, up to lower order corrections.

\begin{figure}
  \centering
  \begin{tikzpicture}
    \begin{axis}[
      scale=1.1,
      tickwidth=6pt,
      axis lines=left,
      xmin=-0.22,
      xmax=0.02,
      width=10cm,
      height=8cm,
      xtick={0, -0.111, -0.2},
      xticklabels={$N^0$,$N^{-\frac1{k-1}}$,$N^{-\frac2{k}}$},
      xlabel style={at={(axis description cs:1.0,0.0)},anchor=north},
      xlabel=$p$,
      ymin=0,
      ymax=2.2,
      ytick={0, 0.35, 0.593, 2},
      yticklabels={$N^0$,$N^{\frac{k-2}{2k-2}}$,$N^{\frac{2k-4}{3k-3}}$,$N^2$},
      ylabel style={rotate=-90,at={(axis description cs:-0.03,1.03)},anchor=north},
      ylabel=$t$]
      \addplot [name path=top, forget plot] coordinates { (-0.2,2) (0,2) };
      \addplot [name path=A, forget plot, draw=none] coordinates { (-0.2,0) (0,2) };
      \addplot [name path=B, forget plot, draw=none] coordinates { (-0.2,0) (-0.111,0.593) (0,2) };
      \addplot [forget plot, dashed, thick] coordinates {  (-0.2,0) (-0.111,0.593) (0,2) };
      \addplot [name path=C, forget plot, thick] coordinates { (-0.2,0) (-0.111,0.35) (0,2) };
      \addplot [name path=bottom, forget plot] coordinates { (-0.2,0) (0,0) };
      \addplot [only marks, mark size=1.5pt] coordinates { (-0.2,0) (-0.111,0.35) (-0.111,0.593) (0,2) };
      \addplot [forget plot, dashed, thick] coordinates { (-0.111,0.35) (-0.111,0.593) };
      \addplot [forget plot] coordinates { (0,0) (0,2) };
      \addplot [forget plot] coordinates { (-0.2,0) (-0.2,2) };
      \addplot [green!50,opacity=0.5] fill between [of=top and B];
      \addplot [black!50!green,opacity=0.3] fill between [of=A and B];
      \addplot [red!50,opacity=0.5] fill between [of=B and C, soft clip={domain=-0.2:-0.111}];
      \addplot [blue!50,opacity=0.5] fill between [of=B and C, soft clip={domain=-0.111:0}];
      \addplot [yellow!50,opacity=0.5] fill between [of=C and bottom];
    \end{axis}
  \end{tikzpicture}
  \caption{
    Phase diagram for the upper-tail problem for $k$-term arithmetic
    progressions, with logarithmic axes. The green region north west of the
    two oblique dashed lines represents the localised regime: the darker
    subregion is the moderate-deviation regime, the lighter one the
    extreme-deviation regime, and the boundary between the two is the
    large-deviation regime. The triangular blue region (east of the vertical
    dashed line segment) represents the Gaussian regime, and the triangular
    red region (west of the vertical dashed line segment) represents the
    Poisson regime. The yellow region south east of the two oblique solid
    lines is the region where the Central Limit Theorem holds.
    \cref{thm:MainResult} gives no information on what happens on the dashed
    lines.
  }
  \label{fig:PhaseDiagram}
\end{figure}
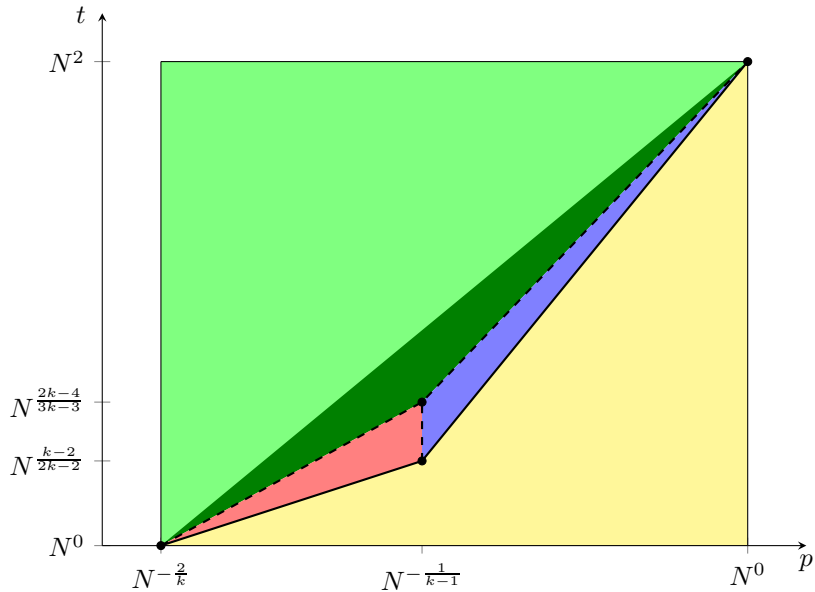

\begin{remark}
  Expanding $\Po(\cdot)$ in Taylor series, one can show that, when $t \ll \mu$,
  \[
    \mu \cdot \Po(t/\mu) = (1+o(1)) \cdot \frac{t^2}{2\mu}.
  \]
  Under the additional assumption that $Np^{k-1} \ll 1$, and thus $\sigma^2 =
  (1+o(1)) \cdot \mu$, the asymptotic rates of the Gaussian and the Poisson
  regimes coincide.  Despite this, there are good reasons to consider the two
  regimes separately.  First, for a narrow range of parameters (when $\mu$ is
  polylogarithmic in $N$), the Poisson regimes includes regions where $t$ is
  commensurate or much larger than $\mu$; when this occurs, the rate function
  in the Poisson regime is significantly smaller than $t^2/(2\sigma^2)$.
  Second, the two regimes are qualitatively very different, since, unlike in the
  Gaussian regime, the rate function in the Poisson regime no longer agrees
  with the naive mean-field prediction, as will be discussed in greater detail
  below. Last but not least, the different phenomenology in the two regimes
  requires vastly different approaches for bounding the tail probabilities from
  both above and below.
\end{remark}

In the localised regime, \cref{thm:MainResult} reduces the upper-tail question
to the solution of a variational problem encoded by $\Psi$.
Following~\cite{harel2022upper}, we consider the family of \emph{$t$-seeds} --
sets that increase the conditional expectation of $X$ by at least $t$:
\begin{equation}
  \label{eq:seeds}
  \cS(t) = \cS_{N,p,k}(t) \coloneqq \{U \subseteq \br{N} : \Ex_U[X]\ge \mu + t\}.
\end{equation}
As we will show in \cref{sec:localised}, the appearance of a $(1 + o(1))
t$-seed implies the upper-tail event with a probability
that is very high compared to the probability of appearance of the seed itself.
Moreover, by picking a particular $t$-seed that realises the minimum
in~\eqref{eq:PsiDef}, one can deduce that the probability of appearance of a
$(1+o(1)) t$-seed in ${\bf R}$ is bounded below by $p^{(1 + o(1)) \cdot
\Psi(t)}$. From this, the lower bound of the localised regime in
\cref{thm:MainResult} follows immediately. The heart of the argument
of~\cite{harel2022upper} that resolved the large-deviation regime was showing
that, for every fixed $\delta > 0$, the probability that $\bR$ contains a
$\delta \mu$-seed $U$ satisfying $|U| = O\big(\Psi(\delta\mu) \log (1/p)\big)$
is $p^{(1+o(1))\Psi(\delta \mu)}$, which is the probability of the appearance
of a smallest such seed. The following theorem, which is the main technical
innovation of this paper, shows that the analogous statement about $t$-seeds
remains true not only for all $t$ but also for a much broader range of sizes of
the seeds.

\begin{theorem}
  \label{thm:localisation}
  Assume $k \ge 3$ and let $p,t,m$ be such that
  \begin{equation} \label{eq:localisation-assumptions-asymptotic}
    t \gg m \cdot \max{} \{1, Np^{k-1} \}
  \quad \text{ and } \quad
  t \gg m^2p^{k-2} \cdot N^{(k-2)(m/t)^{1/(k-1)}}.
  \end{equation} 
  Then
  \[
    \log \Pr\big(U \subseteq \bR \text{ for some $U \in \cS_{N,p,k}(t)$ with $|U| \le
    m$}\big) \le (1 - o(1)) \cdot \Psi(t) \cdot \log p.
  \]
\end{theorem}

\begin{remark}
We claim that the lower-bound assumptions on $t$ are natural. First, ignoring
lower-order terms, every union of $t\leq m/k$ distinct \kap s forms a $t$-seed of
size at most $m$, and so the probability of appearance of such a seed is at least
$\Pr(X\ge t)$. However, since we expect that the planting of a smallest
$(t-\mu)$-seed makes the event $\{X\geq t\}$ significantly more likely, it is
plausible (and, up to lower-order corrections, true) that $\Pr(X\ge t) \geq
p^{\Psi(t-\mu)}$, which is much larger than the upper bound in the theorem,
at least when $t = O(\mu)$. A similar argument applies for all $t=O(m)$, so the
assumption that $t$ is much larger than $m$ is really needed.

Second, observe that every set $U \subseteq \br{N}$ with $m$ elements satisfies
\[
  \Ex_U[X] - \Ex[X] \ge c_k\cdot \left(Nmp^{k-1} + m^2p^{k-2}\right)
\]
for some constant $c_k$ that depends only on $k$; to see this, consider the
$k$-APs intersecting $U$ in either one or two elements. In particular, if $t
\le c_k Nmp^{k-1}$ or $t \le c_k m^2p^{k-2}$, then every $m$-element set is a
$t$-seed. Consequently, at least for $m \le Np$, the probability that $\bR$
contains a $t$-seed with at most $m$ elements is uniformly bounded from below,
contradicting the vanishing upper bound stated by the theorem. Note that the
above argument only justifies a lower bound of the form $t \ge Cm^2p^{k-2}$.
The extra factor of $N^{(k-2)(m/t)^{1/(k-1)}}$ is needed for technical reasons;
however, it is irrelevant once $t/m \gg (\log N)^{k-1}$.
\end{remark}

One may well find it believable that \cref{thm:localisation} plays a direct
role in the proof of the upper bound for the localised regime, where,
following~\cite{harel2022upper}, we use a modified moment argument to show that the
upper-tail event is dominated by the appearance of a `small' seed.  It is
perhaps more surprising that it also plays a crucial role in proving the upper
bound of the Poisson regime. In that context, it allows us to exclude certain
inconvenient terms that arise when calculating the factorial moments of $X$;
these terms correspond to small subsets with rich additive structure.  In fact,
the estimates of factorial moments of $X$ that play the central role in our
treatment of the Poisson regime extend to a portion of the localised regime.
This proves crucial, as there is a small portion of the localised regime (which
we term the \emph{very sparse localised regime}) where the aforementioned
argument based on estimating classical moments of $X$ fails, but can be
salvaged by factorial moment estimates.  (This does not mean, however, that the
very sparse localised regime is phenomenologically distinct from the rest of
the localised regime; see~\cref{sec:localised-regime-UB} for further
discussion.) In contrast, the upper bound for the Gaussian regime is proved by
way of a truncated martingale concentration argument, generalising a classical
inequality of Freedman~\cite{freedman1975tail}. The truncation scheme uses
fairly straightforward moment estimates rather than the more powerful
\cref{thm:localisation}.

\subsection{The naive mean-field approximation}
\label{sec:naive-mean-field}

One way to view \cref{thm:MainResult} is in the context of the naive mean-field
approximation.  For a pair $\Pr$ and $\Qr$ of measures on subsets of $\br{N}$,
with $\Qr$ absolutely continuous with respect to $\Pr$, the {\em
  Kullback--Leibler divergence} of $\Qr$ from $\Pr$ is
defined by
\begin{equation}
  \DKL(\Qr \, \| \, \Pr) \coloneqq \Ex_{\Qr}\left[\log\left(\frac{d\Qr}{d\Pr}(\bR)\right)\right] =
  \sum_{R \subseteq \br{N}} \Qr(\bR=R) \log
  \left(\frac{\Qr(\bR=R)}{\Pr(\bR=R)}\right),
\end{equation}
where $\Ex_{\Qr}$ is the expectation operator associated with the measure
$\Qr$. It is known (cf. \cref{sec:ProbTools}) that the logarithmic probability
of {\em any} event $\cA$ can be obtained by optimising the Kullback--Leibler
divergence over all measures that assign $\cA$ probability one:
\begin{equation}\label{eq:DonskerVaradhan1}
  -\log \Pr(\cA) = \inf_{\substack{\Qr \ll \Pr, \\ \Qr (\cA)=1}} \DKL(\Qr \,\| \,\Pr).
\end{equation} 
The usefulness of such a formulation is limited by the fact that measures that
assign the upper-tail events probability one may be quite difficult to analyse.
The idea of the naive mean-field approximation is to replace the complicated
variational problem in~\eqref{eq:DonskerVaradhan1} by a simpler one, where the
infimum ranges only over product measures (the assumption $\Qr(\cA) = 1$ must
then be relaxed somewhat).  Roughly speaking, the naive mean-field
approximation holds if minimising over this smaller set still
achieves~\eqref{eq:DonskerVaradhan1}, up to lower order corrections. More
precisely, we say the naive mean-field approximation holds for a sequence of
events $\cA_N$, each defined on a measure
space $(\Omega_N,\Pr_N)$, if
\begin{equation}\label{eq:NMF}
  \inf_{\substack{\Qr_N \ll \Pr_N, \\ \lim_{N \to \infty} \Qr_N (\cA_N)=1 \\
  \Qr_N \text{ is a product measure} }} \DKL(\Qr_N \,\| \,\Pr_N) = - (1+o(1))
  \cdot \log \Pr_N(\cA_N).
\end{equation}

The aforementioned large-deviation principles proved
in~\cite{augeri2018nonlinear,chatterjee2016nonlinear,eldan2018gaussian}
establish a version of~\eqref{eq:NMF} when $\cA_N$ are tail events for
non-linear functions of independent random variables that satisfy certain
complexity and smoothness properties.
Bhattacharya--Ganguly--Shao--Zhao~\cite{BhaGanShaZha20} showed that the
number of arithmetic progressions in $\bR$ has the requisite properties when the
density $p$ is sufficiently large. Furthermore, the same work solved the
restricted variational problem of~\eqref{eq:NMF} in the case $\cA_N = \{X \ge
\mu + t\}$ for (nearly) all values of $(p,t)$ with $p$ vanishing and $t \gg
\sigma$. Unsurprisingly, this solution matches the results of
\cref{thm:MainResult} in the entire Gaussian and localised regimes;  \emph{a
posteriori}, \cref{thm:MainResult} thus establishes that the naive mean-field
approximation is valid in those regimes.  In contrast, the naive mean-field
approximation completely fails in the Poisson regime -- the left-hand side
of~\eqref{eq:NMF} is not even of the same order of magnitude as the right-hand
side.

\subsection{Related works}
\label{sec:related}

The study of large- and moderate-deviation regimes of the upper tail of random variables that arise from combinatorial settings has flowered in the last decade.
Besides the aforementioned work of Chatterjee--Dembo~\cite{chatterjee2016nonlinear}, Eldan~\cite{eldan2018gaussian}, Augeri~\cite{augeri2018nonlinear}, and Austin~\cite{austin2018structure}, which are concerned with rather general non-linear functions of independent random variables, there have been numerous works that focus on more specific cases.
The most-studied family of examples are the random variables $X_H$ that count copies of a given graph $H$ in the binomial random graph $G_{n,p}$.
Cook--Dembo~\cite{cook2020large} determined the asymptotics of the logarithmic upper-tail probability of $X_H$ for all $H$ and all $p$ satisfying $n^{-c_H} \ll p \ll 1$ for some positive $c_H$ that depends only on $H$.  More specifically, they established that the naive mean-field approximation holds for $X_H$ in the above range of densities.
Later work of Cook--Dembo--Pham~\cite{cook2024regularity} extended these results to a wider range of densities $p$ and generalised them to the case where $H$ is a uniform hypergraph.
The three authors~\cite{harel2022upper} determined the asymptotics of the logarithmic upper-tail probability of $X_H$ for all regular, non-bipartite $H$ for essentially all densities $p$ (also in the non-mean-field regime); their results were extended to regular, bipartite graphs by Basak--Basu~\cite{basak2023upper}.
In the the moderate-deviation regime, Goldschmidt--Griffths--Scott~\cite{goldschmidt2020moderate} proved asymptotic upper-tail estimates for arbitrary subgraphs for a certain restricted range of densities $p$ and deviations $t$ (using the notation of this paper).
Recently, Alvarado--de Oliviera--Griffiths~\cite{alvarado2023moderate} successfully analysed a far greater (but still sub-optimal) portion of the moderate-deviation regime in the case where $H$ is a triangle.

Finally, there has been some recent progress in the understanding of the typical deviations of the number of $k$-APs in random subsets of $\mathbb{Z}/(N \mathbb{Z})$, the cyclic group of order $N$.  Berkowitz--Sah--Sawhney~\cite{berkowitz2021number} showed that, at least when $p$ is fixed, the standard notion of a local Central Limit Theorem fails for infinitely many $N$, in the sense that the probability that the number of $k$-APs equals a particular integer deviates significantly from the prediction one would get from the Gaussian limit.

\subsection{Organisation}
\label{sec:organisation}

The paper is organised as follows: \cref{sec:ProbTools} includes an overview of
the tilting argument, a classical method for producing lower bounds for rare
events, as well as a proof of the martingale concentration inequality used for
the upper bound of the Gaussian regime. \cref{sec:main-techn-result} is
dedicated to proving \cref{thm:localisation}. The remaining three sections
(\cref{sec:localised,sec:Gaussian,sec:Poisson}) prove \cref{thm:MainResult} for
the localised, Gaussian, and Poisson regimes, respectively; the proof of the
key estimate needed for the very sparse localised regime is postponed
to~\cref{sec:Poisson}, as it is uses methods developed for the Poisson regime.
Finally, \cref{sec:Appendix} proves a bound on the number of connected
hypergraphs with small edge boundary that plays a key role in the analysis of
the Poisson regime (and the very sparse localised regime), and may be of
independent interest.

\section{Probabilistic tools}
\label{sec:ProbTools}

\subsection{The tilting argument}

The tilting argument is a general method to bound the probability of an arbitrary event from below by constructing another measure that makes the event
likely to occur and quantifying its `distance' from the original measure.
Suppose that $\Pr$ and $\Qr$ are two measures on subsets of $\br{N}$. If $\Qr
\ll \Pr$ (that is, if $\Qr$ is absolutely continuous with respect to $\Pr$),
there is a unique (up to a set of measure zero) measurable function
$d\Qr/d\Pr$, called the Radon--Nikodym derivative, such that $\Qr(\cA) = \Ex
\left[d\Qr/d\Pr \cdot  \1_\cA\right]$ for every event $\cA$. In this case, we
define the {\em Kullback--Leibler divergence} of $\Qr$ from $\Pr$ by
\begin{equation}
  \DKL(\Qr \, \| \, \Pr) \coloneqq
  \Ex_{\Qr}\left[\log\frac{d\Qr}{d\Pr}(\bR)\right]= \sum_{R \subseteq \br{N}}
  \Qr(\bR = R) \log \left(\frac{\Qr(\bR = R)}{\Pr(\bR = R)}\right),
\end{equation}
where we use the convention that  $0 \log 0 = 0$.  It is routine to verify that
the Kullback--Leibler divergence between any two measures is nonnegative.  We
will also make use of the following easily verifiable additivity property of
the Kullback--Leibler divergence.

\begin{fact}
  \label{fact:DKL-additivity}
  If $\Pr_1, \dotsc, \Pr_N$ and $\Qr_1, \dotsc, \Qr_N$ are probability measures
  with $\Qr_i \ll \Pr_i$ for all $i \in \br{N}$, then
  \[
    \DKL\big(\Qr_1 \times \dotsb \times \Qr_N \, \| \, \Pr_1 \times \dotsb
    \times \Pr_N \big) = \sum_{i=1}^N \DKL(\Qr_i \,\|\, \Pr_i).
  \]
\end{fact}

It is well known that one can use the notion of Kullback--Leibler divergence to
produce a lower bound for the logarithmic probability of any event $\cA$ under
$\Pr$ by considering a measure $\Qr \ll \Pr$ with $\Qr(\cA) = 1$.

\begin{proposition}\label{prop:DonskerVaradhan}
  Let $\cA$ be an arbitrary event and let $\Pr$ and $\Qr$ be two measures such
  that $\Qr(\cA) = 1$ and $\Qr \ll \Pr$. Then
  \[ \log \Pr(\cA) \ge - \DKL(\Qr \, \| \, \Pr). \]
\end{proposition}
\begin{proof}
  Since our assumptions imply that $\Pr(\cA) > 0$, we may consider the
  conditioned measure $\Pr^* \coloneqq \Pr (\;\cdot \mid \cA)$. Denoting the indicator
  random variable of $\cA$ by $\1_\cA$, we observe that $d\Pr^*/d\Pr = \1_\cA/
  \Pr(\cA)$ and that $\Qr \ll \Pr^*$.  Since the Kullback--Leibler divergence
  is always nonnegative,
  \[
    - \DKL(\Qr \, \| \, \Pr) \le \DKL(\Qr \, \| \, \Pr^*)- \DKL(\Qr \, \| \,
    \Pr) = \Ex_{\Qr}\left[ -\log \frac{d\Pr^*}{d\Pr}(\bR)\right],
  \]
  where the final equality follows because, $\Qr$-almost surely,
  \[
    \frac{d\Qr}{d\Pr^*} \cdot \left(\frac{d\Qr}{d\Pr}\right)^{-1} =
    \left(\frac{d\Pr^*}{d\Pr}\right)^{-1}.
  \]
  Since $\Qr(\cA) = 1$, we find that $d\Pr^*/d\Pr = 1/\Pr(\cA)$ holds
  $\Qr$-almost surely.  Therefore,
  \[
    \Ex_{\Qr}\left[-\log \frac{d\Pr^*}{d\Pr}(\bR) \right] = \log \Pr(\cA),
  \]
  which implies the desired inequality.
\end{proof}

In fact, the proof of \cref{prop:DonskerVaradhan} shows that $-\log\Pr(\cA)$
is precisely equal to the Kullback--Leibler divergence of $\Pr(\;\cdot \mid \cA)$
from $\Pr$.  This allows us to restate the proposition as:
\begin{equation}
  \label{eq:DonskerVaradhan}
  -\log \Pr(\cA) = \inf_{\substack{\Qr \ll \Pr, \\ \Qr (\cA)=1}} \DKL(\Qr \,\| \,\Pr).
\end{equation} 
As mentioned before, \eqref{eq:DonskerVaradhan} is a theoretically useful tool
that is difficult to apply, since the set of measures that assign $\cA$
probability one can be rather unwieldy. Below, we will derive two versions
of this variational principle that are more immediately applicable.  The first,
\cref{lemma:Tilting}, applies to arbitrary measures and will be used for lower
bounds in the localised regime. The second, \cref{prop:lower-bound}, applies
only to measures that assign $\cA$ probability one asymptotically; it will be
used in the Gaussian and Poisson regimes.

\begin{corollary}
  \label{lemma:Tilting}
  For any event $\cA$ and any measures $\Pr$ and $\Qr$ such that $\Qr \ll \Pr$
  and $\Qr(\cA) >0$, 
  \begin{align}
    \log \Pr(\cA) & \ge \log \Qr(\cA) - \Ex_{\Qr}\left[\log
    \frac{d\Qr}{d\Pr}(\bR) \mid \cA\right] \label{eq:TiltingCondExp}
  \end{align}
\end{corollary}
\begin{proof}
  We apply \cref{prop:DonskerVaradhan} to $\Qr(\;\cdot \mid \cA)$. Since 
  \[
    \frac{d \Qr(\;\cdot \mid \cA)}{d\Pr} = \frac{1}{\Qr(\cA)} \cdot  \frac{d \Qr}{d\Pr}
  \]
  holds $\Qr(\;\cdot \mid \cA)$-almost surely, writing $J$ in place of $\log
  (d\Qr/d\Pr)$, we find that
  \[
    \DKL(\Qr(\;\cdot \mid \cA) \, \| \, \Pr) = -\log \Qr(\cA) + \Ex_{\Qr(\cdot
    \mid \cA)}[J(\bR)] = -\log \Qr(\cA) + \Ex_{\Qr}[J(\bR) \mid \cA].
  \]
  The claim now follows from \cref{prop:DonskerVaradhan}.
\end{proof}

\begin{proposition}\label{prop:lower-bound}
  Let $\Pr_N$ and $\Qr_N$ be two sequences of measures satisfying $\Qr_N \ll
  \Pr_N$ for each $N$ and suppose that $\{\cA_N\}$ is a sequence of events such
  that $\limsup_{N \to \infty} \Pr_N(\cA_N) < 1$ and $\lim_{N \to \infty}
  \Qr_N(\cA_N) = 1$. Then 
  \[
    \liminf_{N \to \infty} \frac{\DKL(\Qr_N \, \| \, \Pr_N)}{- \log \Pr(\cA_N)} \ge 1.
  \]
\end{proposition}
\begin{proof}
  We first claim that
  \begin{equation}
    \label{eq:DKL-Bernoulli-lower}
    \DKL(\Qr_N \, \| \, \Pr_N) - \Qr_N(\cA_N)\cdot \log
    \frac{\Qr_N(\cA_N)}{\Pr_N(\cA_N)} - \Qr_N(\cA_N^c) \cdot \log
    \frac{\Qr(\cA_N^c)}{\Pr(\cA_N^c)} \ge 0.
  \end{equation}
  Indeed, for every event $\cE$ with $\Qr_N(\cE) > 0$ (and thus $\Pr_N(\cE) >
  0$), we have
  \[
    \frac{d\Qr_N(\;\cdot \mid \cE)}{d\Pr_N(\;\cdot \mid \cE)} =
    \frac{d\Qr_N}{d\Pr_N} \cdot \frac{\Pr_N(\cE)}{\Qr_N(\cE)}
  \]
  and thus the left-hand side of the above inequality can be seen to equal
  \[
    \Qr_N(\cA_N) \cdot \DKL\big(\Qr_N(\;\cdot \mid \cA_N) \, \| \, \Pr_N(\;\cdot \mid \cA_N)\big) + 
    \Qr_N(\cA_N^c) \cdot \DKL\big(\Qr_N(\;\cdot \mid \cA_N^c) \, \| \, \Pr_N(\;\cdot \mid \cA_N^c)\big),
  \]
  which is clearly nonnegative.  Dividing~\eqref{eq:DKL-Bernoulli-lower}
  through by $-\log \Pr_N(\cA_N)$ and rearranging the terms gives
  \[
    \frac{\DKL(\Qr_N \, \| \, \Pr_N)}{-\log \Pr_N(\cA_N)} \ge \Qr_N(\cA_N) -
    \frac{1}{\log \Pr_N(\cA_N)} \cdot \left(
      \Qr_N(\cA_N) \cdot \log \Qr_N(\cA_N) + \Qr_N(\cA_N^c) \cdot
      \log\frac{\Qr_N(\cA_N^c)}{\Pr_N(\cA_N^c)}
    \right).
  \]
  Finally, our assumptions on the sequences $\Pr_N(\cA_N)$ and $\Qr_N(\cA_N)$
  imply that the first summand in the right-hand side of the above inequality
  tends to one whereas the second summand tends to zero. The desired inequality
  follows by taking the limit inferior of both sides.
\end{proof}

\subsection{A martingale concentration inequality}\label{sec:Martingale}

The main tool for establishing the upper bound in the Gaussian regime is a
martingale concentration inequality, which we formulate in the general context
of hypergraphs. Let $\cH$ be a hypergraph with vertex set $\br{N}$ and let
$\bR$ denote the $p$-biased random subset of $\br{N}$. Let $X$ be the number of
edges of in $\cH[\bR]$, the subhypergraph of $\cH$ that is induced by $\bR$,
and denote the mean and the variance of $X$ by $\mu$ and $\sigma^2$,
respectively. If $\cH$ comprises the $k$-term arithmetic progressions in
$\br{N}$, these notations coincide with the ones used in the rest of paper.
Considering the upper-tail problem for arithmetic progressions in such an
abstract setup of hypergraphs is not a new idea --
both~\cite{griffiths2020deviation,warnke2017upper} follow this route.

Our upper bound for the upper tail of $X$, which could be of independent
interest, is a sum of a Gaussian-like tail bound and three upper-tail
probabilities for various functions of the numbers of edges that the random set
$\bR$ induces in the link hypergraphs of the vertices of $\cH$. For every $i
\in \br{N}$, we let
\begin{equation}
  \label{eq:Li-definition}
  L_i \coloneqq \left|\big\{e \in \cH : e \ni i \text{ and } e \setminus \{i\}
  \subseteq  \bR \big\}\right|
\end{equation}

\begin{proposition}
  \label{prop:GaussianTruncationBound}
  The following holds for all sufficiently small $\eps>0$. Suppose that $\cH$
  is a hypergraph with vertex set $\br{N}$. Let $\bR$ be the $p$-biased random
  subset of $\br{N}$ and let $L_1, \dotsc, L_N$ be the random variables defined
  in~\eqref{eq:Li-definition}. Write $X \coloneqq e(\cH[\bR])$, $\mu \coloneqq
  \Ex[X]$, and $\sigma^2 \coloneqq \Var(X)$. Then for all $t \ge \eps\sigma$,
  we have, letting $\lambda \coloneqq t/\sigma^2$,
  \begin{multline*}
    \Pr(X \ge \mu + t) \le \exp\left(-\frac{(1 - \eps) t^2}{2 \sigma^2}\right)
    + \frac{8 N}{\eps^3 } \cdot \Pr \left(\sum_{i=1}^N L_i^2 > \left(1 +
    \frac{\eps}{10} \right) \cdot \frac{\sigma^2}{p} \right)  \\
    + \frac{8
    N}{\eps^3 } \cdot \Pr \left(\left| \left\{i : L_i >
      \frac{\eps}{\lambda}\right\} \right| \ge \frac{\eps \lambda^2 \sigma^2}{20
  p^{1/2}} \right) + \Pr\left(\exists i \;\; L_i >
\frac{\log(1/p)}{2\lambda}\right).
\end{multline*}
\end{proposition}
\begin{proof}
  Let $Y_i$ be the indicator random variable of the event $\{i\in \bR\}$ and, for
  every $i \in \{0, \dotsc, N\}$, let $\cF_i$ be the $\sigma$-algebra generated
  by $Y_1, \dotsc, Y_i$. The starting point for our considerations is the
  following identity, which holds for all $i \in \br{N}$:
  \begin{equation}
    \label{eq:X-filtration}
    \Ex[X \mid \mathcal{F}_i] - \Ex[X \mid \mathcal{F}_{i-1}] = (Y_i - p) \cdot
    \Ex\left[L_i \mid \cF_{i-1}\right].
  \end{equation}
  Instead of working with the Doob martingale $\big(\Ex[X \mid
  \cF_i]\big)_{i=0}^N$ directly, we will consider a related martingale sequence
  whose differences are truncated versions of~\eqref{eq:X-filtration}.  More
  precisely, for each $i \in \br{N}$, set
  \[
    \hat{L}_i \coloneqq \min{}\{L_i,  \log(1/p)/(2\lambda) \}
  \]
  and define a martingale sequence $(M_i)_{i=0}^N$ by
  \[
    M_0 \coloneqq \Ex[X]
    \qquad \text{and} \qquad
    M_i - M_{i-1} \coloneqq (Y_i - p) \cdot \Ex\big[\hat{L}_i \mid \cF_{i-1}\big] \quad \text{for $i \in \br{N}$}.
  \]
  Since the random variables $X$ and $M_N$ coincide on the event that $\hat{L}_i = L_i$ for all $i$, we find that 
  \begin{equation}
    \label{eq:TruncationBound}
    \Pr(X \ge \mu + t) \le \Pr(M_N - M_0 \ge t) + \Pr\left(\exists i \;\; L_i >
    \frac{\log(1/p)}{2\lambda}\right).
  \end{equation}
  In the remainder of the proof, we will estimate the first probability on the
  right-hand side of~\eqref{eq:TruncationBound}.

  Define the function $\phi \colon \RR \to \RR$ by
  \[
    \phi(0) \coloneqq \frac{1}{2}
    \qquad \text{and} \qquad
    \phi(x) \coloneqq \frac{e^x - x - 1}{x^2} \quad \text{if $x \neq 0$}
  \]
  and observe that $\phi$ is positive and increasing. We also define, for each $i \in \br{N}$,
  \[
    W_i \coloneqq p \cdot \sum_{j=1}^{i} \Ex\big[\phi(\lambda \hat{L}_j) \cdot \hat{L}_j^2 \mid \cF_{j-1}\big].
  \]
  We first show that the upper-tail probability of $M_N$ can be bounded from
  above by the sum of a Gaussian-like tail bound and the probability that $W_N$
  exceeds $\sigma^2/2$ by a macroscopic amount. Our proof is an adaptation of
  the argument used by Freedman~\cite{freedman1975tail} to prove a
  variance-dependent version of the Azuma--Hoeffding inequality; in contrast
  to~\cite{freedman1975tail}, we do not assume an almost-sure bound on $W_N$.

  \begin{claim}
    \label{claim:Freedman}
    For any $\eps >0$, 
    \[
      \Pr(M_N - M_0 \ge t) \le \exp\left(-\frac{(1- \eps)t^2}{2\sigma^2}\right)
      + \Pr\left(W_N > \frac{(1 + \eps) \sigma^2}{2} \right).
    \]
  \end{claim}
  \begin{proof}
    We will show that the sequence $Z_0, \dotsc, Z_N$, defined by
    \[
      Z_i \coloneqq \exp\big(\lambda (M_i - M_0) - \lambda^2 W_i\big)
    \]
    is a supermartingale.  This fact will imply the assertion of the claim.
    Indeed, for every $w \ge 0$,
    \[
      \begin{split}
        \Pr(M_N - M_0 \ge t)
        & = \Pr\left(Z_N \ge e^{\lambda t - \lambda^2 W_N}\right) \le \Pr
        \left(Z_N \ge e^{\lambda t - \lambda^2 w} \right) +  \Pr\left(W_N > w
        \right) \\
        & \le \Ex[Z_N] \cdot e^{\lambda^2 w - \lambda t} +  \Pr\left(W_N > w
        \right),
      \end{split}
    \]
    using Markov's inequality.  If $Z_i$ is in fact a supermartingale, then
    $\Ex[Z_N] \le \Ex[Z_0] = 1$, and the assertion of the claim follows by
    letting $w \coloneqq (1+\eps) \sigma^2/2$ in the above inequality (recall
    that $\lambda = t/\sigma^2$).

    Since $e^{\lambda x} = 1 + \lambda x + \lambda^2x^2 \cdot \phi(\lambda x)$, the definition of $M_i$ yields
    \[
      \begin{split}
        \Ex\left[\exp\big(\lambda (M_i - M_{i-1})\big) \mid \cF_{i-1}\right]
        & = 1 + \lambda^2  \cdot \Ex\left[\phi\big(\lambda (M_i - M_{i-1})\big) \cdot (M_i - M_{i-1})^2 \mid \cF_{i-1}\right] \\
        & \le 1 + \lambda^2 \cdot \Ex \left[\phi\left( \lambda \Ex\big[\hat{L}_i \mid \cF_{i-1}\big]\right) \cdot (Y_i - p)^2 \cdot \Ex\big[\hat{L}_i \mid \cF_{i-1}\big]^2 \mid \cF_{i-1}\right] \\
        & = 1 + \lambda^2 \cdot p(1-p) \cdot \phi\left( \lambda \Ex\big[\hat{L}_i \mid \cF_{i-1}\big]\right) \cdot \Ex\big[\hat{L}_i \mid \cF_{i-1}\big]^2,
      \end{split}
    \]
    where the inequality holds as $\phi$ is increasing, $Y_i - p \le 1$, and
    $\lambda \Ex\big[\hat{L_i} \mid \cF_{i-1}\big] \ge 0$.  Applying Jensen's
    inequality to the convex function $x \mapsto \phi(\lambda x) \cdot x^2 =
    \lambda^{-2} \cdot (e^{\lambda x}-\lambda x-1)$ further gives
    \[
      \begin{split}
        \Ex\left[\exp\big(\lambda (M_i - M_{i-1})\big) \mid \cF_{i-1}\right]
        & \le 1 + \lambda^2 \cdot p(1-p) \cdot \Ex\left[ \phi\big(\lambda \hat{L}_i \big) \cdot \hat{L}_i^2 \mid \cF_{i-1}\right] \\
        & \le \exp \left(\lambda^2 p \cdot \Ex\left[\phi\big(\lambda \hat{L}_i\big) \cdot \hat{L}_i^2 \mid \cF_{i-1}\right]\right) \\
        & = \exp\left(\lambda^2 \big(W_i - W_{i-1}\big)\right).
      \end{split}
    \]
    Rearranging the above inequality gives $\Ex\left[Z_i  \mid \cF_{i-1}\right]
    \le Z_{i-1}$, as claimed, which completes the proof.
  \end{proof}

  While the definition of $W_i$ is convenient in the proof
  \cref{claim:Freedman}, the sequentially conditioned random variables
  appearing in this definition make it difficult to work with this variable
  directly.  Luckily, we may replace the upper-tail probability of $W_N$ by a
  more facile upper-tail probability while incurring only a polynomial loss.
  Define 
  \[
    H_N \coloneqq p \cdot \sum_{i=1}^N \phi(\lambda \hat{L}_i) \cdot \hat{L}_i^2.
  \]

  \begin{claim}
    \label{claim:ExHNWN}
    For any $w \ge 0$, we have $\Ex\big[H_N \mid W_N > w\big] > w$.
  \end{claim}
  \begin{proof}
    We begin by noting that $p \cdot \phi(\lambda \hat{L}_i) \cdot
    \hat{L}_i^2$ is an increasing function of $(Y_1, \dotsc, Y_N)$. Let $G_w
    \coloneqq \{W_N > w\}$ and let $\cG_w$ be the $\sigma$-algebra generated by
    $G_w$.  Harris's inequality~\cite{Har60} implies that, on $G_w$,
    \[
      p \cdot \Ex\left[\phi\big(\lambda \hat{L}_i\big) \cdot \hat{L}_i^2  \mid \cF_{i-1}, \cG_w\right] \ge p \cdot\Ex\left[\phi\big(\lambda \hat{L}_i\big) \cdot \hat{L}_i^2 \mid \cF_{i-1}\right].
    \]
    In particular, we deduce that  
    \[
      \begin{split}
        \Ex[\1_{G_w} \cdot H_N] & = \Ex\left[\1_{G_w}\cdot p \cdot \sum_{i=1}^N \Ex\left[ \phi\big(\lambda \hat{L}_i\big) \cdot \hat{L}_i^2 \mid \cF_{i-1}, \cG_w\right]\right] \\
        & \ge \Ex \left[\1_{G_w} \cdot p \cdot \sum_{j=1}^N  \Ex\left[ \phi\big(\lambda \hat{L}_i\big) \cdot \hat{L}_i^2 \mid \cF_{i-1} \right]\right] = \Ex\left[\1_{G_w} \cdot W_N \right] > w  \cdot \Pr(G_w).
      \end{split}
    \] 
    Dividing through by the probability of $G_w$ completes the proof. 
  \end{proof}

  We now note for future reference that, for every $i \in \br{N}$, since
  $\hat{L}_i \le \log(1/p)/(2\lambda)$ by construction,
  \begin{equation}
    \label{eq:phi-lambda-Li-Li-square}
    \phi\big(\lambda \hat{L}_i\big) \cdot \hat{L}_i^2 \le \frac{e^{\lambda
    \hat{L_i}}}{\lambda^2} \le \frac{1}{p^{1/2} \lambda^2}.
  \end{equation}

  \begin{claim}
    \label{claim:RemovingSigmaAlgebra}
    For all $\eps > 0$ and $t\ge \eps\sigma$,
    \[
      \Pr\left(W_N > \frac{(1 + \eps) \sigma^2}{2}\right) < \frac{8N}{\eps^3 } \cdot \Pr\left(H_N > \frac{(1 + 3\eps/4) \sigma^2}{2}\right).
    \]
  \end{claim}
  \begin{proof}
    Note that, by~\eqref{eq:phi-lambda-Li-Li-square}, we have $H_N \le N / \lambda^2$ almost surely.  In particular, this implies that
    \[
      \Ex\left[H_N \mid W_N > \frac{(1+\eps)\sigma^2}{2} \right] \le \frac{N}{\lambda^2} \cdot \Pr\left( H_N > \frac{(1 + 3\eps/4) \sigma^2}{2} \mid W_N > \frac{(1+ \eps)\sigma^2}{2} \right) + \frac{(1 + 3\eps/4) \sigma^2}{2}.
    \]
    On the other hand, by \cref{claim:ExHNWN},
    \[
      \Ex\left[H_N \mid W_N > \frac{(1+\eps)\sigma^2}{2} \right] > \frac{(1 + \eps) \sigma^2}{2}.
    \]
    Combining these two inequalities, multiplying through by the probability
    that $W_N$ exceeds $(1+\eps)\sigma^2/2$, and recalling that $\lambda^2
    \sigma^2 = (t/\sigma)^2 \ge \eps^2$ gives the assertion of the claim.
  \end{proof}

  Finally, we partition the upper tail of $H_N$.  To this end, observe that
  when $\eps$ is sufficiently small, then for all $i$ such that $\hat{L}_i \le
  \eps / \lambda$, we have $\phi(\lambda \hat{L}_i) \le \phi(\eps) \le
  (1+\eps/2)/2$ .  Using~\eqref{eq:phi-lambda-Li-Li-square} for all remaining
  $i$, we obtain
  \[
    H_N \le \frac{(1+\eps/2)}{2} \cdot p \cdot \sum_{i=1}^N \hat{L}_i^2 +
    \frac{p^{1/2}}{\lambda^2} \cdot \left|\left\{i : \hat{L}_i >
    \frac{\eps}{\lambda}\right\}\right|.
  \]
  Since $(1+\eps/2)(1+\eps/10) + 2\eps/20 \le 1+3\eps/4$ for all sufficiently
  small $\eps > 0$, we may conclude that
  \[
    \Pr\left( H_N > \frac{(1 + 3\eps/4) \sigma^2}{2} \right) \le  \Pr\left(
      \sum_{i=1}^N \hat{L}_i^2  \ge \left(1+\frac{\eps}{10}\right) \cdot
      \frac{\sigma^2}{p} \right) + \Pr\left( \left| \left\{i : \hat{L}_i >
      \frac{\eps}{\lambda} \right\} \right| \ge \frac{\eps \lambda^2
    \sigma^2}{20p^{1/2}} \right).
  \] 

  Combining~\eqref{eq:TruncationBound},
  \cref{claim:Freedman,claim:RemovingSigmaAlgebra}, and the above estimate for
  the upper tail of $H_N$ yields
  \begin{multline*}
    \Pr(X \ge \mu + t) \le
    \exp\left(-\frac{(1 - \eps) t^2}{2 \sigma^2}\right) +
    \frac{8N}{\eps^3 } \cdot \Pr \left(\sum_{i=1}^N \hat{L}_i^2 > \left(1 +
    \frac{\eps}{10} \right) \cdot \frac{\sigma^2}{p} \right)  \\ +
    \frac{8N}{\eps^3} \cdot \Pr \left(\left| \left\{i : \hat{L}_i >
      \frac{\eps}{\lambda}\right\} \right| \ge
      \frac{\eps \lambda^2 \sigma^2}{20
    p^{1/2}} \right) +
    \Pr\left(\exists i \;\; L_i > \frac{\log(1/p)}{2\lambda}\right).
\end{multline*}
Finally, since $\hat{L}_i \le L_i$, we may replace the truncated variables in
both probabilities above with the untruncated versions, thereby only increasing the
right-hand side.
\end{proof}

\section{The probability of small seeds: Proof of~\cref{thm:localisation}}
\label{sec:main-techn-result}

In this section, we prove the main technical result of this paper,
\cref{thm:localisation}. It will be more convenient to state and prove an
equivalent version of this result, where the function $\Psi$ is replaced by
its combinatorial analogue $\Psi^*$, which we now define. In order to do so, we
first define, for all $U \subseteq \br{N}$,
\[
  A_k(U) \coloneqq |\{B \in \APk : B \subseteq U\}|.
\]
With this, for all $t \ge 0$, let
\begin{equation}\label{eq:PsiStarDef}
  \Psi^*(t) = \Psi_{N,p,k}^*(t) \coloneqq \min{} \{|U| : U\subseteq \br N \text{ and } A_k(U) \ge t\},
\end{equation}
cf.~\eqref{eq:PsiDef}. As before we set $\Psi^*(t)=\infty$ when
$t>|\APk|$.  Since every $k$-AP contained in a set $U \subseteq \br{N}$
contributes $1-p^k$ to the difference $\Ex_U[X]-\Ex[X]$, we have
$\Psi(t) \le \Psi^*\big(t/(1-p^k)\big)$ for all $t \ge 0$.  Furthermore, a
straightforward computation shows that $A_k(\br{m}) = (1+o(1)) \cdot
\frac{1}{k-1} \binom{m}{2}$ as $m\to\infty$, which implies that $\Psi^*(t) \le
(1+o(1)) \cdot \sqrt{2(k-1)t}$ whenever $1\ll t\leq |\APk|$. Finally, we will show
that $(1-o(1)) \cdot \sqrt{2(k-1)t}$ is a lower bound on $\Psi(t)$.  Together,
these facts will establish the following proposition.

\begin{proposition}
  \label{prop:Psi-Psi-star}
  Let $k\ge 3$ and assume that $\max{}\{1,N^2p^{2k-2}\}\ll t \leq |\APk|-\mu$
  and $p \ll 1$. Then
  \[
    \Psi(t) = (1+o(1)) \cdot \Psi^*(t) = (1+o(1)) \cdot \sqrt{2(k-1)t}.
  \]
\end{proposition}

We remark that a version of \cref{prop:Psi-Psi-star} was proved by
Bhattacharya--Ganguly--Shao--Zhao~\cite{BhaGanShaZha20}.  However, their
version~\cite[Theorem~2.2]{BhaGanShaZha20}  requires a stronger
lower-bound assumption on $t$, which is in fact necessary for the continuous
relaxation of $\Psi$ which they consider (and which does not always coincide
with the combinatorial notion of $\Psi$ used in this work).  Even though our
proof of~\cref{prop:Psi-Psi-star} essentially repeats the argument
of~\cite[Proposition~4.3]{harel2022upper}, we include it here for completeness.

We now state the aforementioned version of~\cref{thm:localisation}, with $\Psi$
replaced by $\Psi^*$.

\begin{proposition}
  \label{prop:localisation}
  For every positive $\eps$ and every integer $k\ge 3$, there is some $C$ such
  that the following holds. Let $N\in \NN$ and $p\in (0,1/2)$, and define
  $\cSb(t,C)$ to be the set of all $t$-seeds $U\subseteq \br N$ such that
  \begin{equation}
    \label{eq:localisation-assumptions}
    t\ge C|U| \cdot \max\{1, Np^{k-1}\} \qquad \text{and} \qquad t\ge C|U|^2p^{k-2} \cdot N^{(k-2)(|U|/t)^{1/(k-1)}}.
  \end{equation}
  Then, for every $t\geq 0$,
  \[
    \log \Pr\big(U \subseteq \bR  \text{ for some $U \in \cSb(t,C)$}\big)
    \le (1-\eps) \cdot \Psi^*\big((1-\eps)t\big)
    \cdot \log p.
  \]
\end{proposition}


The remainder of this section is organised as follows. In
\cref{sec:derivation-localisation-proof-Psi-Psi-star}, we prove
\cref{prop:Psi-Psi-star} and present the short derivation of
\cref{thm:localisation} from \cref{prop:localisation}.  The remaining two
subsections are devoted to the proof of \cref{prop:localisation}.  The short
\cref{sec:preliminaries-proof-localisation} presents three auxiliary, technical
results needed for the proof, which is presented in the much more substantial
\cref{sec:proof-localisation}.

\subsection{Proof of \cref{prop:Psi-Psi-star} and derivation of \cref{thm:localisation}}
\label{sec:derivation-localisation-proof-Psi-Psi-star}

Given a set $U \subseteq \br{N}$ and an integer $k \ge 3$, it will be
convenient to denote, for every $r \in \br{k}$, the number of $k$-APs that
intersect $U$ in exactly $r$ elements by $A_r^{(k)}(U)$; note that then $A_k(U) =
A_k^{(k)}(U)$. If $X$ denotes the number of $k$-APs in the $p$-biased
random subset
$\bR \subseteq \br{N}$, linearity of expectation allows us to write
\begin{equation}
  \label{eq:ExU-Ex-Ark}
  \Ex_U[X] - \Ex[X] = \sum_{r=1}^k A_r^{(k)}(U) \cdot (p^{k-r} - p^k).
\end{equation}
Since any two numbers lie in at most $\binom{k}{2}$ distinct $k$-APs
(equivalently, the hypergraph $\APk$ of $k$-APs in $\br{N}$ satisfies
$\Delta_2(\APk) \le \binom{k}{2}$), we may bound
\begin{equation}\label{eq:A_k_r_breakdown}
  A_1^{(k)}(U) \le \binom{k}{2}N|U|
  \qquad
  \text{and}
  \qquad
  A_2^{(k)}(U) + \dotsb + A_k^{(k)}(U) \le \binom{k}{2}\binom{|U|}{2}.
\end{equation}
Finally, the proof of \cref{prop:Psi-Psi-star} relies on the following
combinatorial result that appears as~\cite[Theorem~2.4]{BhaGanShaZha20}.
(We note that~\cite{BhaGanShaZha20} considers a slightly different
variational problem, since that work counts progressions with positive and
negative common difference separately; this leads to a difference of $\sqrt{2}$
between the result quoted below and the one that appears
in~\cite{BhaGanShaZha20}.)

\begin{lemma}
  \label{lem:IntervalOptimal}
  For every $U \subseteq \br{N}$ with $m$ elements, $A_k(U) \le A_k(\br{m}) =
  (1 + o(1))\frac{m^2}{2(k-1)}$. In particular, if $1\ll t \leq |\APk|$, then
  $\Psi^*(t) = (1 + o(1))\sqrt{2(k-1) t}$.
\end{lemma}

\begin{proof}[Proof of \cref{prop:Psi-Psi-star}]
  We have already mentioned the bound $\Psi(t) \le \Psi^*\big(t/(1-p^k)\big)$.
  The assumption $t\leq |\APk|-\mu = |\APk|(1-p^k)$ implies that $t/(1-p^k)\leq |\APk|$,
  so \cref{lem:IntervalOptimal} gives
  \[ \Psi(t) \le \Psi^*\big(t/(1-p^k)\big) = (1+o(1))\sqrt{2(k-1)t}
    = (1+o(1))\Psi^*(t),
  \]
  where we used $p\ll 1$.
  In view of this, it remains to show
  that $\Psi(t) \ge (1- o(1)) \cdot \sqrt{2 (k-1) t}$ as long as $t \gg
  \max\{1, N^2p^{2k-2}\}$.  Fix $\eps >0$ and consider an arbitrary set $U
  \subseteq \br{N}$ with $|U| \le (1- \eps) \sqrt{2(k-1)t}$.
  Using~\eqref{eq:ExU-Ex-Ark} and~\eqref{eq:A_k_r_breakdown}, we find that
  \begin{align*}
    \Ex_U[X] - \Ex[X] & \le A_k(U) + \sum_{r=1}^{k-1} A_r^{(k)}(U) \cdot  p^{k-r}
    \\
                      & \le A_k(U) + p \cdot \binom{k}{2} \binom{|U|}{2}  + \binom{k}{2} N |U| p^{k-1}.
  \end{align*}
  The final two terms on the right-hand side are $o(t)$; this follows from the
  upper bound on $|U|$ and the assumption $p \ll 1$ (for the second term) or
  the assumption $t \gg N^2 p^{2k-2}$ (for the third term).
  Furthermore,~\cref{lem:IntervalOptimal} implies that $A_k(U) \le (1 - \eps)
  t$. Thus, $\Ex_U[X] - \Ex[X] < t$ for every set $U$ with at most $(1- \eps)
  \sqrt{2(k-1)t}$ elements, as desired. 
\end{proof}

\begin{proof}[Derivation of \cref{thm:localisation} from \cref{prop:localisation}]
  Note that every $t$-seed with at most $m$
  elements, where $t$ and $m$
  satisfy~\eqref{eq:localisation-assumptions-asymptotic}, belongs to
  $\cSb(t,C)$ for every fixed $C>0$ and all large enough $N$. It thus
  suffices to show that the existence of such a $t$-seed $U$
  implies that the two assumptions
  of~\cref{prop:Psi-Psi-star} hold, so that we may replace $(1 - \eps)
  \Psi^*\big((1-\eps)t\big)$ with $(1-o(1)) \cdot \Psi(t)$.
  To this end, suppose that $U \subseteq \br{N}$ is a $t$-seed with at most $m$ elements.
  Then, firstly, we know that $\mu+t\leq |\APk|$. Secondly, by~\eqref{eq:ExU-Ex-Ark}
  and~\eqref{eq:A_k_r_breakdown},
  \[
    t \le \Ex_U[X] - \Ex[X] \le \sum_{r=1}^k A_r^{(k)}(U) \cdot p^{k-r} \le
    \binom{k}{2} Nm \cdot p^{k-1} + \binom{k}{2}\binom{m}{2}.
  \]
  In particular, this means that $t \le Km \left(Np^{k-1} + m\right)$ for some
  constant $K$ that depends only on $k$. Since we have assumed that $t \gg
  mNp^{k-1}$, we conclude that $t \le 2Km^2$ and thus $t \ge (t/m)^2/(2K) \gg
  N^2p^{2k-2}$. Finally, thanks to the assumption $t \gg m^2p^{k-2}
  N^{(k-2)(m/t)^{1/(k-1)}} \ge m^2p^{k-2}$, we conclude that $p \ll 1$.
\end{proof}

\subsection{Preliminaries for the proof of \cref{prop:localisation}}
\label{sec:preliminaries-proof-localisation}

If $U$ is a finite set and $f\colon \cP(U)\to \RR$ is a function on its power
set, then the partial derivative of $f$ with respect to $u \in U$ is the
function $\partial_u f\colon \cP(U)\to \RR$ defined by $\partial_u
f(U')\coloneqq f(U'\cup \{u\})-f(U'\setminus \{u\})$ for every $U' \subseteq
U$. The following simple lemma, which
generalises~\cite[Lemma~3.8]{harel2022upper}, is a key ingredient in the proof
of \cref{lemma:Ai-core} below. For the readers that are familiar with the
general framework of~\cite{harel2022upper}, we remark that this lemma will
allow us to extract a core from every seed.

\begin{lemma}
  \label{lemma:abstract-core}
  If $U$ is a finite set and $f \colon \cP(U) \to \RR$ is a function, then, for
  every $w\in \RR^{|U|}$, there exists a subset $U^*\subseteq U$ such that
  \begin{enumerate}[label=(\roman*)]
  \item
    \label{item:abstract-core-i}
    $f(U^*) \ge f(U) - \lVert w\rVert_1$ and
  \item
    \label{item:abstract-core-ii}
    $\partial_u f(U^*) \ge w_{|U^*|}$ for all $u\in U^*$.
  \end{enumerate}
\end{lemma}
\begin{proof}
  Let $U =U_0\supsetneq U_1\supsetneq \dotsb \supsetneq U_k = U^*$ be a chain
  of maximal length such that $f(U_{i-1}) - f(U_i) < w_{|U_{i-1}|}$ for all
  $1\le i \le k$. Then $U^*$ satisfies~\ref{item:abstract-core-ii}, since
  otherwise we could obtain a longer chain by setting $U_{k+1} = U_k\setminus
  \{u\}$ for an element $u\in U_k$ with $\partial_u f(U_k)< w_{|U_k|}$. To see
  that $U^*$ also satisfies~\ref{item:abstract-core-i}, note that
  \[
    f(U) - f(U^*) = \sum_{i = 1}^{k} \big(f(U_{i-1})-f(U_i)\big) 
    \le \sum_{i=1}^{k} w_{|U_{i-1}|} \le \lVert w\rVert_1,
  \]
  because the cardinalities $|U_0|, \dotsc, |U_{k-1}|$ are distinct positive integers.
\end{proof}

The second ingredient is a version of Janson's inequality~\cite{Jan90} for the
hypergeometric distribution;  it can be derived from the standard version of
Janson's inequality and the fact that the mean of binomial distribution is also
its median, provided that it is an integer (cf.\ the proof of
\cite[Lemma~3.1]{BalMorSamWar2016}).

\begin{lemma}
  \label{lemma:hyper-Janson}
  Suppose that $(B_\alpha)_{\alpha \in A}$ is a family of subsets of a
  $t$-element set $\Omega$. Let $s \in \{0, \dotsc, t\}$ and let
  \[
    \mu \coloneqq \sum_{\alpha \in A} \left(\frac{s}{t}\right)^{|B_\alpha|} \qquad
    \text{and} \qquad \Delta \coloneqq \sum_{\alpha \sim \beta}
    \left(\frac{s}{t}\right)^{|B_\alpha \cup B_\beta|},
  \]
  where the second sum is over all ordered pairs $(\alpha, \beta) \in A^2$ such
  that $\alpha \neq \beta$ and $B_\alpha \cap B_\beta \neq \emptyset$. Let $S$
  be the uniformly chosen random $s$-element subset of $\Omega$ and let $Z$
  denote the number of $\alpha \in A$ such that $B_\alpha \subseteq S$. Then,
  for every $\eps \in (0,1]$,
  \[
    \Pr\big(Z \le (1-\eps)\mu\big) \le 2 \exp\left(-\frac{\eps^2}{2}
    \cdot \frac{\mu^2}{\mu + \Delta}\right).
  \]
\end{lemma}

In the proof of \cref{prop:localisation}, we will encounter the function $\beta
\colon (0,1] \to \RR_{>0}$ defined by \begin{equation}\label{eq:beta}
  \beta(x) \coloneqq \frac{1}{(2-\log{x})^2}.
\end{equation}
The precise details of this definition have no deeper meaning; what we require
is essentially a function on $(0,1]$ that approaches $0$ sufficiently slowly as
$x\to 0$ while still having the property that $\int_0^1x^{-1}\beta(x)\,\mathrm
dx$ exists. Some relevant properties of $\beta$ are collected in the following
fact.

\begin{fact}\label{fact:beta}
  The following statements hold:
  \begin{enumerate}[(i)]
  \item $\beta$ is increasing;
  \item for $c > 0$, the function $x \mapsto x^{-c} \beta(x)$ is decreasing on $(0, e^{2-2/c}]$ and increasing on $[e^{2-2/c}, e^2)$;
  \item $\int_0^1 x^{-1}\beta(x)\, \mathrm dx = 1/2$.
  \end{enumerate}
\end{fact}
\begin{proof}
  The first item is obvious. 
  A direct computation shows that, for every $c > 0$,
  \[ \big(x^{-c} \beta(x)\big)' = \frac{c\log{x}-2c+2}{x^{c+1} \cdot (2-\log{x})^3}, \]
  which is negative for $x\in (0,e^{2-2/c})$.
  For the last item, we have
  \[ \int_{0}^1 \frac{1}{x \cdot (2-\log x)^2}\, \mathrm dx
   = \left[ \frac{1}{2-\log x}\right]_0^1 = 1/2. \qedhere \]
\end{proof}

\subsection{Proof of \cref{prop:localisation}}
\label{sec:proof-localisation}

As earlier, we write $A_r^{(k)}(U)$ for the number of $k$-term arithmetic
progressions in $\br N$ that intersect $U$ at precisely $r$ elements.
Throughout the proof, we will suppress the dependence of this quantity on $k$
for notational convenience.

%

\begin{definition}\label{def:core}
  A set $U^*\subseteq \br N$ is called a 
  \emph{$(t,\eps,\xi)$-core}, for some $t, \eps, \xi > 0$, if 
  \begin{enumerate}[label=(C\arabic*)]
    \item
      \label{item:core-lower}
      $|U^*| \ge \Psi^*\big((1-\eps)t\big)$ and
    \item
      \label{item:core-derivative}
      for some $r \in \{3, \dotsc, k\}$ and all $u \in U^*$,
      \[
        \partial_u A_r(U^*) \ge \frac{\xi}{|U^*|} \cdot \max\left\{t, \left(\frac{t}{|U^*|^2}\right)^{\frac{1}{k-2}} \cdot \max_{K \subseteq U^*} A_{r-1}(K)\right\}.
      \]
  \end{enumerate}
\end{definition}

We note that every interval in $\br N$ of length $\floor{\sqrt{2(k-1)t}}$ is a
$(t,\eps,\xi)$-core provided that $\xi$ is smaller than some $\xi_0 = \xi_0(k)
> 0$. Indeed, \ref{item:core-lower} follows from~\cref{prop:Psi-Psi-star} and
\ref{item:core-derivative} (for $r=k$) is a straigthforward calculation (the
left-hand side is linear in $\sqrt{t}$ whereas the right-hand side is linear in
$\xi\sqrt{t}$).

Our first lemma states that every seed contains a core. This lemma, together
with the definition of a core, lie at the the very heart of the proof of
Proposition~\ref{prop:localisation}.

\begin{lemma}
  \label{lemma:Ai-core}
  For every $\eps \in (0,1)$, there exist positive
  $\delta=\delta(\eps,k)$ and $C=C(\eps,k)$ such that the following
  holds. Let $U$ be a $t$-seed with $m$ elements, where $t\ge C \cdot \max{\{m^2p^{k-2}, mNp^{k-1}, 1\}}$.
  Then $U$ contains a subset $U^*$ that is a $\big(t,\eps,\delta \beta(|U^*|/m)\big)$-core,
  where $\beta$ is defined as in~\eqref{eq:beta}.
\end{lemma}

Note that $\delta \beta(|U^*|/m) = \delta /\big(2-\log(|U^*|/m)\big)^{2}$ is
not quite a constant -- things would be a little simpler if it were -- but it
is good enough for the purposes of our next lemma, which bounds the number of
cores of a given size.

\begin{lemma}
  \label{lemma:number-of-cores}
  For all $\eps,\delta,\eta \in (0,1/2)$, there is a positive $C=C(\delta,\eps,\eta,k)$ such that the following holds. Let $t$ and $m$ be positive integers
  with $t\ge C \cdot \max{\{m,  mNp^{k-1}, m^2p^{k-2}
  N^{C(m/t)^{1/(k-1)}} \}}$.
  Denote by
  $\cC(s)$ the set of $\big(t,\eps,\delta\beta(s/m)\big)$-cores of size $s$.
  Then
  \[ |\cC(s)| \le\begin{cases}
    (\eta/p)^s & \text{if $s\le m$,}\\
    (1/p)^{\eta s} & \text{if $s\le \min{\{m, \sqrt{4kt}\log(1/p)}\}$}.
  \end{cases} \]
\end{lemma}

Before proving the two lemmas, let us show how they imply the statement
of~\cref{prop:localisation}.
Fix some $\eps\in (0,1/2)$ and let
$\delta=\delta(\eps,k)$ be as in \cref{lemma:Ai-core}; we may clearly assume
that $\delta \le \xi_0$, where $\xi_0$ is the constant introduced below the
definition of a core.  Let $\eta \coloneqq \eps/2 \le 1/e$ and let
$m$ be such that $t\ge C \cdot \max{\{m,  mNp^{k-1}, m^2p^{k-2}
N^{C(m/t)^{1/(k-1)}} \}}$ for a sufficiently large $C=C(\delta,\eps,\eta,k)$.
Denote by $\cC(s)$ the set of $\big(t,\eps,\delta\beta(s/m)\big)$-cores of size
$s$, as in the statement of \cref{lemma:number-of-cores}, and let $s_0 \ge
\Psi^*\big((1-\eps)t\big)$ be the minimal value of $s$ for which $\cC(s)$ is
nonempty. It is enough to show that
\[ \Pr\big(U \subseteq \bR \text{ for some $t$-seed $U$ with $|U| \le m$}\big)
\leq p^{(1-\eps)s_0}. \]
%


By \cref{lemma:Ai-core} and the union bound, the left-hand side of
the above inequality is at most
\[
  \Pr\big(U^* \subseteq \bR \text{ for some $(t,\eps,\delta\beta(|U^*|/m))$-core $U^*$ with $|U^*| \le m$}\big)
  \le \sum_{s=s_0}^m |\cC(s)| \cdot p^s.
\]
Further, by \cref{lemma:number-of-cores},
\[
  \begin{split}
    \sum_{s=s_0}^m |\cC(s)|\cdot p^{s}
    \le \sum_{s=s_0}^{\sqrt{4kt}\log(1/p)} p^{(1-\eta) s}
  +\sum_{s=\sqrt{4kt}\log(1/p)}^m 
  \eta^{s}
  \le \frac{p^{(1-\eps/2)s_0}}{1-p^{1-\eps/2}}
  + \frac{p^{\sqrt{4kt}}}{1-1/e}.
  \end{split}
\]
As mentioned above, any interval of length $\floor{\sqrt{2(k-1)t}}$ is a
$(t,\eps,\xi)$-core, whenever $\xi \le \xi_0$.  In particular, since $\delta
\beta(s/m) \le \delta \beta(1) \le \xi_0$ for all $s \le m$, we have $s_0 \le
\sqrt{2(k-1)t}$.  Consequently, the right-hand side above is at most
$4p^{(1-\eps/2)s_0} \le p^{(1-\eps)s_0}$, as $p \le 1/2$ and $s_0 \ge
\Psi^*\big((1-\eps) t) \ge \sqrt{t} \ge \sqrt{C}$, by \cref{prop:Psi-Psi-star}.

The remaining part of this section is dedicated to proving
\cref{lemma:Ai-core,lemma:number-of-cores}.
\begin{proof}[Proof of \cref{lemma:Ai-core}]
  Let $\eta=\eta(\eps,k) > 0$ be sufficiently small and let $C=C(\eps,\eta,k) >
  0$ be sufficiently large. Define, for every $r\in \{2, \dotsc, k\}$, the
  function $a_r\colon \RR\to\RR$ given by
  \[ a_r(x)\coloneqq
  (1-\eta) \left(\frac{x^2}{\eta t}\right)^{\frac{k-r}{k-2}}. \]
  We say that a subset $U'\subseteq U$ is \emph{$r$-dense} if
  $A_r(U') \ge a_r(|U'|) \cdot t$.  Observe that, as long as $\eta < 4/(k^2+4)$, no set $U' \subseteq U$ can be $2$-dense. Indeed, by~\eqref{eq:A_k_r_breakdown},
  \[
    A_2(U') \le \binom{k}{2} \binom{|U'|}{2} \le \frac{k^2|U'|^2}{4} <
    \frac{(1-\eta) |U'|^2}{\eta} = a_2(|U'|) \cdot t.
  \]

  \begin{claim}
    The $t$-seed $U$ is $r$-dense for some $r \ge 3$.
  \end{claim}
  \begin{proof}
    By the definition of a $t$-seed and~\eqref{eq:ExU-Ex-Ark}, for every
    sequence $\lambda \in \RR^k$ with $\lambda_1 + \dotsb + \lambda_k \le 1$,
    \[
      \sum_{r=1}^k \lambda_r t \le t \le \Ex_U[X] - \Ex[X] \le \sum_{r=1}^k A_r(U) \cdot p^{k-r},
    \]
    which implies that there is some $r \in \br k$ such that $A_r(U) \cdot
    p^{k-r}\ge \lambda_r t$.  With this in mind, we define $\lambda_r$ as
    follows:
    \[
      \lambda_r \coloneqq
      \begin{cases}
        (1-\eta)\eta^{k-1} & \text{if } r = 1, \\
        a_r(|U|) \cdot p^{k-r} &  \text{if }r \ge 2.\\
      \end{cases}
    \]
    Our assumption $t\ge C|U|^2p^{k-2} \ge \eta^{1-k}|U|^2p^{k-2}$ (which holds
    for all large enough $C$) implies that $\lambda_r \le (1-\eta)\eta^{k-r}$
    for all $r\in \br k$, so indeed
    \[
      \lambda_1+\dotsb+\lambda_k
      \le 
      (1-\eta)\sum_{r=1}^k \eta^{k-r}
      < (1-\eta)\sum_{r=0}^\infty \eta^{r}
      = 1. \]
    Consequently, $A_r(U) \cdot p^{k-r} \ge \lambda_r t$ for some $r \in \br
    k$.  Note that we may rule out the case $r=1$, thanks to the bound $A_1(U)
    \le \binom{k}{2}|U|N$, see~\eqref{eq:A_k_r_breakdown}, and the assumption
    that $t\ge C|U|Np^{k-1}$ for a sufficiently large $C=C(\eps,\eta,k)$.  We
    may therefore conclude that $A_r(U) \ge a_r(|U|) \cdot t$, i.e., that $U$
    is $r$-dense, for some $r \in \{2, \dotsc, k\}$;  since we have shown above
    that no subset of $U$ can be $2$-dense, we must have $r \ge 3$.
  \end{proof}

  The claim allows us to define $\tilde U\subseteq U$ as a smallest nonempty
  subset of $U$ that is $r$-dense for some $r \ge 3$.  In a slight abuse of
  notation, let $r$ be the smallest index such that $\tilde U$ is $r$-dense.
  Then the minimality of $\tilde U$ and $r$ implies that no subset of $\tilde
  U$ is $(r-1)$-dense.

  Now let $w\in \RR^{|\tilde U|}$ be defined by
  $w_\ell \coloneqq \eta A_r(\tilde U) \cdot  \ell^{-1} \beta(\ell/|\tilde U|)$. Since $x^{-1}\beta(x)$ is decreasing on $(0,1]$
  and $\int_0^1x^{-1}\beta(x)\,\mathrm dx \le 1$, we have
  \[ \lVert w\rVert_1 = \eta A_r(\tilde U) \cdot  \frac{1}{|\tilde U|} \sum_{\ell= 1}^{|\tilde U|} \frac{\beta(\ell/|\tilde U|)}{\ell / |\tilde U|}
  \le \eta A_r(\tilde U) \int_0^1 
  x^{-1}\beta(x)\,\mathrm dx \le \eta A_r(\tilde U). \]
  Therefore, we may apply \cref{lemma:abstract-core} to obtain a subset $U^*\subseteq \tilde U$ such that 
  \begin{equation}\label{eq:AiUs-lower}
    A_r(U^*)\ge A_r(\tilde U)-\lVert w\rVert_1 \ge  (1-\eta) \cdot A_r(\tilde U) \ge 
    (1-\eta)\cdot a_r(|\tilde U|) \cdot t
    \ge (1-\eta) \cdot a_r(|U^*|) \cdot t,
  \end{equation}
  and, for every $u \in U^*$,
  \begin{equation}
    \label{eq:AiUs-mindeg}
    \partial_u A_r(U^*) \ge w_{|U^*|} = 
    \eta A_r(|\tilde U|) \cdot \frac{\beta(|U^*|/|\tilde U|)}{|U^*|}
    \ge 
    \frac{\eta \cdot \beta(|U^*|/|U|)}{{|U^*|}}
    \cdot
    a_r(|U^*|) \cdot t,
  \end{equation}
  where the last inequality uses
  $A_r(\tilde U)\ge a_r(|\tilde U|) \cdot t\ge a_r(|U^*|) \cdot t$ and the fact that $\beta$ is increasing.

  \begin{claim}
    The set $U^*$ is a $\big(t,\eps,\eta^2(1-\eta)^2\beta(|U^*/|U|)\big)$-core.
  \end{claim}
  \begin{proof}
    We first prove that
    \begin{equation}\label{eq:seed-or-large}
      A_k(U^*) \ge (1-2\eta)t
      \qquad\text{or}\qquad
      |U^*| \ge \sqrt{4kt},
    \end{equation}
    which implies~\ref{item:core-lower}, in the first case by the definition of $\Psi^*$ 
    (provided that $\eta < \eps/2$) and in the second case by~\cref{prop:Psi-Psi-star}.
    Suppose that $r \in \{3, \dotsc, k\}$ attains~\eqref{eq:AiUs-lower}. First,
    if $r = k$, then~\eqref{eq:AiUs-lower} immediately gives
  \[
    A_k(U^*) \ge \left(1-\eta\right) \cdot a_k(|U^*|) \cdot t = (1-\eta)^2t
    \ge  (1-2\eta)t,
  \]
  as desired.  Otherwise, if $3 \le r < k$, then~\eqref{eq:AiUs-lower} gives
  \[ A_r(U^*) \ge (1-\eta) \cdot a_r(|U^*|) \cdot t = (1-\eta)^2
    \left(\frac{|U^*|^2}{\eta t}\right)^{\frac{k-r}{k-2}}t
  \]
  which we combine with the bound $A_r(U^*) \le \binom{k}{2}
  \binom{|U^*|}{2} \le k^2|U^*|^2$, see~\eqref{eq:A_k_r_breakdown}, to obtain
  \[ |U^*|^{2r-4}k^{2k-4} \ge (1-\eta)^{2k-4}\eta^{r-k}t^{r-2}. \]
  Since the exponent of $\eta$ is negative, the required inequality $|U^*|\ge
  \sqrt{4kt}$ follows for sufficiently small $\eta$.

  Finally, we prove that~\ref{item:core-derivative} holds with $\xi \coloneqq
  \eta^2\beta(|U^*|/|U|)$.  By~\eqref{eq:AiUs-mindeg}, it is enough
  to show that
  \[
    a_r(|U^*|) \ge \eta \qquad \text{and} \qquad a_r(|U^*|) \cdot t
    \ge \eta \left(\frac{t}{|U^*|^2}\right)^{\frac{1}{k-2}} \cdot \max_{K \subseteq
    U^*} A_{r-1}(K).
  \]
  First, since $|U^*|^2 \ge \Psi^*\big((1-\eps)t\big)^2 \ge t$,
  by~\ref{item:core-lower} and \cref{prop:Psi-Psi-star}, we have 
  \[
    a_r(|U^*|) = (1-\eta)  \left(\frac{|U^*|^2}{\eta t}\right)^{\frac{k-r}{k-2}} \ge \eta.
  \] 
  Second, the minimality of $\tilde U$ and $r$ implies that every $K \subseteq
  U^* \subseteq \tilde U$ is not $(r-1)$-dense and therefore
  \[
    A_{r-1}(K) < a_{r-1}(|K|) \cdot t \le a_{r-1}(|U^*|) \cdot t
    = \left(\frac{|U^*|^2}{\eta t}\right)^{\frac{1}{k-2}} \cdot a_r(|U^*|) \cdot t
    \le \frac1{\eta}\left(\frac{|U^*|^2}{t}\right)^{\frac{1}{k-2}} \cdot a_r(|U^*|) \cdot t.
  \]
  This completes the proof of the claim.
\end{proof}
The assertion of the lemma now follows with $\delta \coloneqq \eta^2(1-\eta)^2$.
\end{proof}

\begin{proof}[Proof of \cref{lemma:number-of-cores}]
  Fix some $r\in \br k$, $U\subseteq \br N$, and $u\in \br N$.  It will be
  convenient to introduce a quantity that is closely related to the discrete
  derivative $\partial_u A_r(U)$ but has a slightly simpler combinatorial
  interpretation.  Define $A_r(U,u)$ as the number of \kap{}s in $\br N$ that
  intersect $U\cup \{u\}$ in exactly $r$ and $U\setminus \{u\}$ in exactly
  $r-1$ elements.  We will first show that, for every $u \in U$,
  \begin{equation}
    \label{eq:partial-AiUu}
    \partial_uA_r(U) = A_r(U, u) - A_{r+1}(U, u) \le A_r(U, u) = A_r(U \setminus \{u\}, u).
  \end{equation}
  In order to see this, we count the \kap{}s that intersect both $U$ and $U
  \setminus \{u\}$ in exactly $r$ elements.  We can express their number in two
  ways: first, by $A_r(U) - A_r(U,u)$ and second, by $A_r(U \setminus \{u\}) -
  A_{r+1}(U,u)$.  This implies the first equality in~\eqref{eq:partial-AiUu};
  the remainder of~\eqref{eq:partial-AiUu} is straightforward.

  Since we can clearly assume that $\cC(s)\neq\emptyset$,
  property~\ref{item:core-lower} and~\cref{prop:Psi-Psi-star} imply that $s \ge
  \Psi^*\big((1-\eps)t\big) \ge \sqrt{t}$.  For the sake of brevity, we let
  \begin{equation}
    \label{eq:xi-def}
    \xi \coloneqq \delta \cdot \beta(s/m).
  \end{equation}
  Suppose that $U^*\in \cC(s)$. By~\ref{item:core-derivative}
  and~\eqref{eq:partial-AiUu}, there is some $r\ge 3$ such that
  \begin{equation}\label{eq:core-derivative-again}
    s\cdot A_r(U^*,u)\ge
    s\cdot \partial_u A_r(U^*) \ge 
    \xi \cdot \max{\left\{ t, \left(\frac{t}{s^2}\right)^{\frac{1}{k-2}} \cdot \max_{K \subseteq U^*}A_{r-1}(K)\right\}}\quad\text{for all $u\in U^*$}.
  \end{equation}
  For every $r\ge 3$, let $\cC_r(s)$ be the subset of $\cC(s)$ containing those $U^*$ that satisfy~\eqref{eq:core-derivative-again}.
  Since $|\cC(s)| \le |\cC_3(s)|+ \dotsb + |\cC_k(s)|$, it is enough to bound each $|\cC_r(s)|$ individually.

  For the remainder of the proof, fix some $r\ge 3$. We will say that an
  element $u \in \br{N}$ is \emph{rich} with respect to a subset $K \subseteq
  \br{N} \setminus \{u\}$ if
  \[
    s\cdot A_r(K,u) \geq \frac{\xi}{2} \cdot \left(\frac{|K|}{s}\right)^{r-1}
    \left(\frac{t}{s^2}\right)^{\frac{1}{k-2}}\cdot A_{r-1}(K).
  \]
  Let $\cR(K)$ denote the set of all elements of $\br{N}\setminus K$ that are
  rich with respect to $K$ and observe that
  \[
    \begin{split}
      \frac{\xi |\cR(K)|}{2} \cdot \left(\frac{|K|}{s}\right)^{r-1}
      \left(\frac{t}{s^2}\right)^{\frac{1}{k-2}} \cdot A_{r-1}(K) & \le \sum_{u
      \in \cR(K)} s \cdot A_r(K,u) \\
      & \le \sum_{u \in \br{N} \setminus K} s \cdot A_r(K,u) = (k-r+1) \cdot
      s\cdot A_{r-1}(K).
  \end{split}
  \]
  This implies that, whenever $K$ is nonempty,
  \begin{equation}\label{eq:rich-bound}
    |\cR(K)| \le \frac{2ks}{\xi} \cdot \left(\frac{s}{|K|}\right)^{r-1}\left(\frac{s^2}{t}\right)^{\frac1{k-2}}.
  \end{equation}
  It is also easy to see that $|R(\emptyset)| = N$.

  Let us briefly explain how the notion of rich elements can help us bound the
  number of cores $U^*\in \cC_r(s)$.   If we order the elements of $U^*$ at
  random as $u_1,\dotsc,u_s$, then, on average, a
  $\left(\frac{d-1}{s}\right)^{r-1}$-fraction of the $\kap$s counted by
  $A_r(U^*,u_d)$ is also counted by $A_r(\{u_1,\dotsc,u_{d-1}\},u_d)$.
  In particular, \eqref{eq:core-derivative-again} suggests that, in a typical
  random ordering, $u_d$ will be rich with respect to $\{u_1,\dotsc,u_{d-1}\}$
  for most indices $d$. On the other hand, due to~\eqref{eq:rich-bound}, the
  fact that $u_d$ is rich means that it comes from a somewhat small set that
  depends only on $\{u_1, \dotsc, u_{d-1}\}$.  This translates into an upper
  bound on the number of ways to choose the whole set $U^*$ element-by-element.

  We now turn to the implementation of this idea.  Given an arbitrary sequence
  $u_1, \dotsc, u_s$ of distinct elements of $\br{N}$, define the set of
  \emph{poor} indices
  \[
    \cP(u_1, \dotsc, u_s) \coloneqq
    \big\{d \in \br s : \text{$u_d$ is \emph{not} rich with respect to $\{u_1, \dotsc, u_{d-1}\}$}\big\}.
  \]
  First, we show that, for an average ordering $u_1, \dotsc, u_s$ of the
  elements of a core $U^*$, the set of poor indices is small. Second, we give
  an upper bound on the total number of $s$-element sequences for which the
  poor indices belong to a given set.

  \begin{claim}
    \label{claim:Ex-cX}
    Let $U^* \in \cC_r(s)$ and let $u_1, \dotsc, u_s$ be a uniformly chosen
    random ordering of the elements of $U^*$. Then
    \[
      \Ex\big[|\cP(u_1,\dotsc,u_s)|\big] \le 2s\left(\frac{9k^3s}{\xi t}\right)^{1/(k-1)}.
    \]
  \end{claim}
  \begin{proof}
    For every integer $d \in \br s$ and all $u \in U^*$, define
    \begin{equation}
      \label{eq:mudu-lower}
      \mu_d(u) \coloneqq \left(\frac{d-1}{s-1}\right)^{r-1} \cdot A_r(U^*, u)
      \ge \frac{\xi}s\cdot \left(\frac{d-1}{s}\right)^{r-1} \cdot
      \max{\left\{t, \left(\frac{t}{s^2}\right)^{\frac{1}{k-2}} \cdot \max_{K
      \subseteq U^*} A_{r-1}(K)\right\}},
    \end{equation}
    where the inequality follows from~\eqref{eq:core-derivative-again}.
    Note that $d \in \cP(u_1,\dotsc,u_s)$ implies \[ A_r\big(\{u_1,
    \dotsc, u_{d-1}\}, u_d\big) \le \mu_d(u_d)/2. \] Note also that $1\notin
    \cP(u_1,\dotsc,u_s)$. Consequently,
    \begin{equation}
      \label{eq:ExcX-upper}
      \Ex\big[|\cP(u_1,\dotsc,u_s)|\big] \le \sum_{d=2}^s
      \Pr\big(A_r\big(\{u_1, \dotsc, u_{d-1}\}, u_d\big) \le \mu_d(u_d)/2\big).
    \end{equation}

    Fix some $d \in \br s$ and note that, conditioned on $u_d$, the set $\{u_1,
    \dotsc, u_{d-1}\}$ is a uniformly random $(d-1)$-element subset of $U^*
    \setminus \{u_d\}$.  Thus, we may use Janson's inequality
    (\cref{lemma:hyper-Janson}) to get an upper bound on the probabilities in
    the above sum.  For $u \in U^*$, let $\cJ_u$ be the multiset defined by
    \[
      \cJ_u \coloneqq \big\{B \cap (U^* \setminus \{u\}) : \text{$B \in \APk$,
      $u \in B$, and $|B \cap U^*| = r$}\big\},
    \]
    where the multiplicity of each element is equal to the number of
    \kap{}s $B$ containing $u$ giving rise to the same $(r-1)$-element set $B
    \cap (U^* \setminus \{u\})$. Observe that, for every $u \in U^*$, we have
    $|\cJ_u| = A_r(U^*, u)$ and $|\{J\in \cJ_u : J \subseteq K\}| \le A_r(K,
    u)$ for all $K \subseteq U^* \setminus \{u\}$. Gearing towards an
    application of \cref{lemma:hyper-Janson}, note first that
    \[
      \sum_{J \in \cJ_u} \left(\frac{d-1}{s-1}\right)^{|J|} = |\cJ_u| \cdot
      \left(\frac{d-1}{s-1}\right)^{r-1} = \mu_d(u).
    \]
    Further, writing $J \sim J'$ to mean that $J \neq J'$ and $J\cap J' \neq
    \emptyset$ (which also includes the case where $J$ and $J'$ are the same
    $(r-1)$-element subset of two different \kap{}s),
    \[
      \Delta_d(u) \coloneqq \sum_{\substack{J, J' \in \cJ_u \\ J \sim J'}}
      \left(\frac{d-1}{s-1}\right)^{|J \cup J'|} \le \mu_d(u) \cdot k^3,
    \]
    where the inequality holds because, for every $J\in \cJ_u$, there are at
    most $k^3$ progressions of length $k$ that contain $u_d$ and some element
    of $J$.  \cref{lemma:hyper-Janson} implies that
    \[
      \begin{split}
        \Pr\left(A_r\big(\{u_1, \dotsc, u_{d-1}\}, u_d\big) \le \mu_d(u_d)/2 \mid
          u_d\right) &\le 2\exp\left(-\frac{\mu_d(u_d)^2}{8\big(\mu_d(u_d) +
      \Delta_d(u_d)\big)}\right) \\
      & \le
      2\exp\left(-\frac{\mu_d(u_d)}{9k^3}\right).
    \end{split}
    \]
    Observe moreover that~\eqref{eq:mudu-lower} implies $\mu_d(u_d) \ge
    \frac{\xi}{s} \cdot \left(\frac{d-1}{s}\right)^{k-1}  \cdot t$, so taking
    expectations of the above expression allows us to conclude that
    \[
      \Pr\left(A_r\big(\{u_1, \dotsc, u_{d-1}\}, u_d\big) \le \mu_d(u_d)/2\right)
      \le 2 \cdot \Ex\left[\exp\left(-\frac{\mu_d(u_d)}{9k^3}\right)\right]
      \le 2 \cdot \exp\left(-\frac{(d-1)^{k-1}\xi t}{9k^3s^k}\right).
    \]

    Substituting this inequality into~\eqref{eq:ExcX-upper} and using
    $h\sum_{d\ge 1} f(d \cdot h) \le \int_0^\infty f(x)\,\mathrm dx$ for the
    decreasing function $f(x) = \exp(-x^{k-1})$ and $h = \left(\frac{\xi
    t}{9k^3s^k}\right)^{1/(k-1)}$, we obtain
    \[
      \Ex[|\cP(u_1,\dotsc,u_s)|]\le
      2 \sum_{d=1}^{s-1} \exp\left( - \frac{\xi t}{9k^3s} \cdot
      \frac{d^{k-1}}{s^{k-1}} \right) \le 2s\left(\frac{9k^3s}{\xi
    t}\right)^{1/(k-1)} \int_0^\infty \exp\left(-x^{k-1}\right)\, \mathrm dx.
    \]
    Finally, since the gamma function is convex on $\RR_{>0}$ and $\Gamma(1)=\Gamma(2)=1$, we have
    \[ \int_0^\infty
    \exp(-x^{k-1})\, \mathrm dx = 
    \frac1{k-1}\int_0^\infty
    y^{\frac1{k-1}-1} e^{-y}\,\mathrm dy
    =
    \frac1{k-1}\Gamma\left(\frac1{k-1}\right) =
    \Gamma\left(\frac{k}{k-1}\right)\le 1,
    \]
    which completes the proof of the claim.
  \end{proof}

  To state our second claim, it will be convenient to define, for any $P \subseteq \br{s}$,
  \[
    \cX_P \coloneqq \big\{(u_1, \dotsc, u_s) \in \ff{\br{N}}{s} : \cP(u_1, \dotsc, u_s) \subseteq P\big\}.
  \]
  
  \begin{claim}
    \label{claim:cXP-upper}
    For every $P \subseteq \br{s}$,
    \[
      |\cX_P| \le (2ke^r)^s\cdot N^{|P|+1}\cdot
      \left(\frac{s^2}{t}\right)^{\frac{s}{k-2}}\cdot \xi^{-s} \cdot s!.
  \]
  \end{claim}
  \begin{proof}
    We can choose the elements of every sequence $(u_1, \dotsc, u_s) \in \cX_P$
    one-by-one as follows:  Suppose that $u_1, \dotsc, u_{d-1}$ have already
    been chosen.  If $d \in P$, then we have at most $N$ choices for $u_d$.
    Otherwise, if $d \notin P$, then $u_d$ must be chosen from the set
    $\cR\big(\{u_1, \dotsc, u_{d-1}\}\big)$ of elements that are rich with
    respect the set of previously chosen elements.
    By~\eqref{eq:rich-bound}, this set has at most $f(d-1)$ elements, where
    \[
      f(x) \coloneqq
      \begin{cases}
        N & \text{if $x=0$}\\
        \frac{2ks}{\xi} \cdot
        \left(\frac{s}{x}\right)^{r-1}\left(\frac{s^2}{t}\right)^{\frac1{k-2}} & \text{otherwise.}
      \end{cases}
    \]
    Since $s \ge \sqrt{t}$ implies that $f(d)\ge 1$ for all $d\in \br{s}$,
    \[
      |\cX_P| \le N^{|P|}\cdot \prod_{d\notin P}f(d-1) \le N^{|P|+1}\cdot
      \prod_{d=1}^sf(d) = N^{|P|+1} \left(\frac{2ks^r}{\xi}\right)^s \cdot
      \left(\frac{s^2}{t}\right)^{\frac{s}{k-2}} \cdot \prod_{d=1}^s
      \frac{1}{d^{r-1}}.
    \]
    Using the inequality $s! \ge (s/e)^s$, we moreover have
    \[
      \prod_{d=1}^s \frac{1}{d^{r-1}} = (s!)^{1-r} \le s!\cdot(e/s)^{rs},
    \]
    which implies the claimed bound.
  \end{proof}

  We now use the two claims to prove \cref{lemma:number-of-cores}, where we
  can assume that $s\leq m$. Let $\tau \coloneqq 4s\big(9k^3s/(\xi
  t)\big)^{1/(k-1)}$. \cref{claim:Ex-cX} and Markov's inequality imply that,
  for every $U^* \in \cC_r(s)$, at least $s!/2$ orderings of the elements of
  $U^*$ belong to $\cX_P$ for some $P \subseteq \br{s}$ with at most $\tau$
  elements. Therefore,
  \[
    |\cC_r(s)| \le \frac{2}{s!} \cdot \sum_{\substack{P \subseteq \br{s}\\ |P|
    \le \tau}} |\cX_P| \le \frac{2^{s+1}}{s!} \cdot \max_{\substack{P \subseteq
    \br{s} \\ |P| \le \tau}} |\cX_P|.
  \]
  and thus, by \cref{claim:cXP-upper},
  \[
    |\cC_r(s)| \le 2^{s+1} \cdot (2ke^r)^s\cdot N^{\tau+1}\cdot
    \left(\frac{s^2}{t}\right)^{\frac{s}{k-2}}\cdot \xi^{-s}.
  \]
  Recall the definition of $\xi$ given in~\eqref{eq:xi-def}.  Using the fact
  that $x \mapsto x / \beta(x)$ is increasing on $(0,1]$, by \cref{fact:beta},
  we have $s/\beta(s/m) \le m / \beta(1) = m/4$ and thus
  \[
    \tau = 4s  \left(\frac{9k^3s}{\delta \beta(s/m) t}\right)^{1/(k-1)} \le
    4s\left(\frac{9k^3m}{4\delta t}\right)^{1/(k-1)}.
  \]
  Moreover, our assumptions imply that $\tau$ is sufficiently large, so we can
  assume that $\tau+1\leq 4^{1/(k-1)}\tau$. Set $h\coloneqq
  10k^3e^k/\delta$. Using~\eqref{eq:xi-def} once more, we may write
  \begin{equation}
    \label{eq:core-count}
    \begin{split}
      |\cC_r(s)|^{1/s} & \le \frac{8ke^k}{\delta} \cdot
      N^{4\left(\frac{9k^3m}{\delta t}\right)^{1/(k-1)}} \cdot
      \left(\frac{s^2}{t}\right)^{\frac{1}{k-2}}\cdot \frac1{\beta(s/m)} \\
      & \le h \cdot N^{h(m/t)^{1/(k-1)}}\cdot \left(\frac{s^2}{t}\right)^{\frac{1}{k-2}}
      \cdot \frac1{\beta(s/m)}.
    \end{split}
  \end{equation}
  
  To prove the first assertion of the lemma, it suffices to show that
  $|\cC_r(s)|\le \frac1k (\eta/p)^{s}$. Since $(0,1] \ni x\mapsto
  x^{2/(k-2)}/\beta(x)$ achieves its maximum at some $w_k$ that depends only on
  $k$ (see \cref{fact:beta}), we have
  \[
    \frac{s^{2/(k-2)}}{\beta(s/m)} \le m^{2/(k-2)} \cdot \frac{w_k^{2/(k-2)}}{\beta(w_k)}.
  \]
  Using~\eqref{eq:core-count} and our assumption
  $t \ge C m^2p^{k-2} N^{C(m/t)^{1/(k-1)}}$,
  we therefore obtain
  \[
    |\cC_r(s)|^{1/s} \le h \cdot N^{h(m/t)^{1/(k-1)}}\cdot \left(\frac{m^2}{t}\right)^{\frac1{k-2}} \cdot
    \frac{w_k^{2/(k-2)}}{\beta(w_k)} \le \frac{\eta}{k^{1/s}p},
  \]
  provided that $C$ is sufficiently large.

  Finally, for the second assertion of the lemma, assume that $s\le \sqrt{4kt}\log(1/p)$. We then have
  \[
    \left(\frac{s^2}{t}\right)^{\frac{1}{k-2}} \le \left(4k\log(1/p)^2\right)^{\frac1{k-2}}
  \]
  and, as $x\mapsto 1/\beta(x)$ is decreasing and $s^2 \ge t \ge m^2p^{k-2}$,
  \[
    \frac1{\beta(s/m)} \le \frac1{\beta(p^{(k-2)/2})} =\left(\frac{k-2}{2}\log(1/p) + 2\right)^2.
  \]
  To bound $N^{h(m/t)^{1/(k-1)}}$, we distinguish two cases. If $p< N^{-1/(2k-2)}$, then
   the assumption $t\ge Cm$ implies
  \[ N^{h(m/t)^{1/(k-1)}} \le N^{hC^{-1/(k-1)}}
  \le \left(\frac1p\right)^{\eta/2}
  \]
  for all sufficiently large $C$. On the other hand, if $p\ge N^{-1/(2k-2)}$,
  then we have $t\ge CNmp^{k-1}\ge mN^{1/2}$.
  Moreover, the inequalities $CNmp^{k-1} \leq t\leq s^2 \leq Nm$ imply that $p\leq C^{-\frac1{k-1}}$,
  so if $C$ is large enough, then
  \[ N^{h(m/t)^{1/(k-1)}} \le N^{hN^{-1/(2k-2)}}
  \le 2
  \le \left(\frac1p\right)^{\eta/2}
  \]
  also in this case. 
  Substituting the above inequalities into~\eqref{eq:core-count}, we obtain
  \[
    |\cC_r(s)|^{1/s}
    \le h\cdot 
    \left(\frac1{p}\right)^{\eta/2}
    \left(4k\log(1/p\right)^2)^{\frac1{k-2}}
    \cdot 
    \left(\frac{k-2}{2}\log(1/p) + 2\right)^2
    \le \frac{1}{k^{1/s}} \cdot \left(\frac1{p}\right)^{\eta}
  \]
  provided $C$ is large enough.
  This completes the proof of the lemma, and with it,~\cref{prop:localisation}.
\end{proof}

\section{The localised regime}
\label{sec:localised}

In this section, we prove~\cref{thm:MainResult} in the localised regime.  We
will assume throughout this section that $(p,t)$ is in the localised regime,
that is,
\begin{equation}
  \label{eq:LocalizedRegimeRep}
  \begin{split}
    N^{-1/(k-1)} < p \ll 1 \quad &\text{ and } \quad \sqrt{t} \log (1/p) \ll t^2/\sigma^2, \quad \text{ or }\\
     \Omega(N^{-2/k}) \leq p \leq N^{-1/(k-1)}  \quad& \text{ and } \quad
     \sqrt{t} \log (1/p) \ll \mu \cdot \Po(t/\mu).
\end{split}
\end{equation}
Since~\cref{thm:MainResult} holds vacuously when $\mu + t > |\APk|$, we can
assume without loss of generality that $t\leq |\APk| \leq N^2$.
Moreover, it is straightforward to show that these assumptions and $k\geq 3$ also imply
that $t \gg \max{}\{1, N^2p^{2k-2}\}$; indeed, $\mu \cdot \Po(t/\mu) \le t^2/\mu$
and $\sigma^2 = \Omega\big(N^2p^k + N^3p^{2k-1}\big)$.

We will prove the lower and the upper bounds on the upper-tail probability of
$X$ separately.  The proof of the lower bound is a fairly straightforward
application of \cref{lemma:Tilting}.  The proof of the upper bound involves a
conditioned moment argument, adapted from \cite{harel2022upper}, that crucially
relies on \cref{thm:localisation} (or rather on its alternate version,
\cref{prop:localisation}).  To complicate things further, this approach breaks
down in a small sliver of the localised regime, where we have to resort to a
much more delicate estimate of conditioned factorial moments of $X$.

\subsection{Proof of the lower bound in the localised regime}

Since $t \gg \max\{1, N^2p^{2k-2}\}$ in the entire localised regime, the
following proposition, together with \cref{prop:Psi-Psi-star}, implies the
required lower bound on the tail probability.

\begin{proposition}\label{prop:LocalisedTheoremSection}
  Suppose that the sequence $(p,t)$ satisfies~\eqref{eq:LocalizedRegimeRep}.
  Then, for every $\eps >0$ and integer $k\geq 3$, and all large enough $N$,
  \[
    \log \Pr \left(X \ge \mu + t \right) \ge -  (1+ \eps) \cdot \Psi^*\big((1 + \eps) t\big) \cdot \log(1/p).
  \]
\end{proposition}

\begin{proof}
  Fix an $\eps >0$ and let $U \subseteq \br{N}$ be a smallest set satisfying
  $A_k(U) \ge (1+ \eps) t$, so that $|U| = \Psi^*\big((1+ \eps) t\big)$. Define
  $\Qr \coloneqq \Pr(\cdot \mid U \subseteq \bR)$; more explicitly,
  \[
    \Qr(R) =
    \begin{cases}
      p^{|R| - |U|} (1-p)^{N - |R|} & \text{if } U \subseteq R, \\
      0& \text{otherwise}.
    \end{cases}
  \]
  We apply~\cref{lemma:Tilting} to the event $\{X \ge \mu + t\}$, and conclude that
  \[
    \log \Pr ( X \ge \mu + t) \ge \log \Qr ( X \ge \mu + t) - \Ex_{\Qr}\left[\log\frac{d\Qr}{d\Pr}(\bR) \mid X \ge \mu + t\right].
  \]
  Since $\log (d\Qr/d\Pr)$ is identically equal to $|U| \log(1/p)$ on the support of $\Qr$, we have
  \[
    \log \Pr ( X \ge \mu + t) \ge \log \Qr ( X \ge \mu + t) - |U| \log (1/p) =  \log \Qr ( X \ge \mu + t) - \Psi^*\big((1+ \eps) t\big) \log (1/p).   
  \]
  As $\psi^*\big((1+\eps)t\big) \ge \sqrt{t}$, by \cref{prop:Psi-Psi-star}, in order to complete the proof, it suffices to show that 
  \begin{equation}
    \label{eq:ProbAfterTilting}
    \log \Qr ( X \ge \mu + t) > -\eps  \cdot  \sqrt{t} \log (1/p).
  \end{equation}

  To this end, note that, by~\eqref{eq:ExU-Ex-Ark} and since $p\to 0$, we have
  \[
    \Ex_{\Qr}[X] = \Ex_U[X] \ge \mu + A_k(U) \cdot (1-p^k) \ge  \mu + (1+\eps/2) t
  \]
  for all $N$ large enough.  Since $X \le |\APk| \le N^2$ always, we may conclude that
  \[
    \mu + (1 + \eps/2) t \le \Ex_{\Qr}[X] \le (\mu + t)  + N^2 \cdot \Qr (X \ge \mu + t),
  \]
  which implies that
  \[
    \log \Qr (X \ge \mu  +t ) \ge \log \left(\frac{\eps t}{2N^2} \right) \ge -2\log N.
  \]
  Finally, we prove that $2\log N \le \eps\sqrt{t}\log(1/p)$ by distinguishing two cases. 
  If $p\le N^{-1/k}$, i.e., if $\log(1/p) \ge (\log N)/k$, this follows from $t \gg 1$.
  Otherwise, if $p > N^{-1/k}$, the assumption that $(p,t)$ is in the localised
  regime implies that $\sqrt t\log(1/p) \gg \sqrt t \gg \sigma^{2/3}$, which easily implies the claim
  since $\sigma^2 = \Omega(N^2p^k) = \Omega(N)$.
\end{proof}

\subsection{Proof of the upper bound in the localised regime}
\label{sec:localised-regime-UB}

In view of \cref{prop:localisation}, to prove the upper bound in the localised
regime, it suffices to show that the upper-tail event is dominated by the
appearance of a small $(1 - \eps)t$-seed, that is, an element of
$\cSb((1-\eps)t, C)$ for a suitably large $C=C(k,\eps)$, where $\cSb$ is
defined as in the statement of \cref{prop:localisation}.


\begin{proposition}
  \label{prop:LocUBSeed}
  Suppose that the sequence $(p,t)$ satisfies~\eqref{eq:LocalizedRegimeRep}.
  Then, for all $C,\eps > 0$, every integer $k\geq 3$, and all sufficiently
  large $N$,
  \[
    \Pr\left(X \ge \mu + t\right) \le (1 + \eps)\cdot  \Pr \left(U \subseteq
      \bR \text{ for some } U \in \cSb\big((1- \eps)t,C\big)\right).
  \]
\end{proposition}

In order to prove \cref{prop:LocUBSeed}, we will bound the probability of the
upper-tail event occurring {\em without} the appearance of a small
$(1-\eps)t$-seed and then compare that bound with the lower bound on the upper
tail probability that we proved in the previous subsection.  The following
lemma will provide a suitable estimate on the probability of the upper-tail
event occurring without the appearance of a small $(1-\eps)t$-seed in all but a
tiny sliver of the localised regime.  Even though it is implicitly proved
in~\cite{harel2022upper}, we include its proof (an adaptation of the elegant
argument from~\cite{janson2004upper}) here for completeness. Given $u,m > 0$,
let $Z(u,m)$ be the indicator random variable of the event that $\bR$ does not
include a $u$-seed of size at most $m$.

\begin{lemma}
  \label{lemma:JOR}
  For every $u \ge 0$ and every positive integer $m$,
  \[
    \Ex[ X^m \cdot Z(u, km)] \le (\mu +u)^m.
  \]
\end{lemma}

\begin{proof}
  Recall that $\cS(u)$ denotes the set of all $u$-seeds.
  For any $S \subseteq \br{N}$, let $Z_{S}$ be the indicator of the event that
  $\bR \cap S$ does not contain any $U \in \cS(u)$ with $|U| \le  km$ as a
  subset, so that $Z_{\br{N}} = Z(u,km)$; note that $Z_S \le Z_{S'}$ if $S' \subseteq S \subseteq \br{N}$.
  For any $B \subseteq \br{N}$, let $Y_B$ be
  the indicator of the event that $B \subseteq \bR$. Since we can write $X
  = \sum_{B \in \APk} Y_B$, it is straightforward to see that, for every
  nonnegative integer $\ell$,
  \[
    X^\ell \cdot Z = \sum_{B_1, \dotsc, B_\ell \in \APk} Y_{B_1 \cup \dotsb
    \cup B_{\ell}} \cdot Z \le \sum_{B_1, \dotsc, B_\ell \in \APk} Y_{B_1 \cup
    \dotsb \cup B_{\ell}} \cdot Z_{B_1 \cup \dotsb \cup B_{\ell-1}} \eqqcolon
    M_\ell.
  \]
  The required inequality will clearly follow if we show that $\Ex[M_\ell]
  \le (\mu+u)^\ell$ for all $\ell \in \br{m}$.  We prove this
  stronger estimate using induction on $\ell$.  The base case $\ell=1$ holds
  vacuously, as $M_1 = X$.  Suppose now that $\ell \ge 2$ and that
  $\Ex[M_{\ell-1}] \le (\mu+u)^{\ell-1}$.  Fix some $B_1, \dotsc,
  B_{\ell-1} \in \APk$,  let $U \coloneqq B_1 \cup \dotsb \cup B_{\ell-1}$,
  and let $\cA_U$ denote the event that $Y_U \cdot Z_{U} =1$.  Since $|U| \le k\ell \le  km$,
  $\cA_U$ holds if and only if $U \subseteq \bR$ and $U \notin \cS(u)$.
  Therefore, if $\cA_U$ has nonzero probability, we have $\Ex[X \mid
  \cA_U] = \Ex[X \mid Y_U=1] \le \mu+ u$.  This means that
  \[
    \sum_{B_\ell \in \APk}  \Ex[Y_{U \cup B_{\ell}} \cdot Z_{U}] = \Ex[Y_{U}
    \cdot Z_{U}] \cdot \Ex[X \mid \cA_U]  \le \Ex[Y_{U} \cdot Z_{B_1 \cup \dotsb
    \cup B_{\ell-2}}] \cdot (\mu + u).
  \]
  Summing this inequality over all $B_1, \dotsc, B_{\ell-1} \in \APk$ yields
  $\Ex[M_\ell] \le \Ex[M_{\ell-1}] \cdot (\mu + u) \le (\mu+u)^\ell$.
\end{proof}

In the following, let
$m_{\max}(u,C)$ be the largest~$m$ that satisfies
\begin{equation}
  \label{eq:LocAssumptions}
  u \ge  Cm \cdot \max\{1, Np^{k-1}\}   \qquad \text{and} \qquad u \ge  Cm^2p^{k-2} \cdot N^{(k-2)(m/u)^{1/(k-1)}},
\end{equation}
so that a $u$-seed $U$ belongs to $\cSb(u)$ precisely when $|U|\leq
m_{\max}(u,C)$. Moreover, let $Z_u\coloneqq Z(u,m_{\max}(u,C))$ be the
indicator random variable for the event that $\bR$ does not contain a small
$u$-seed.

As in~\cite{harel2022upper}, the above lemma permits us to derive an upper bound on the
probability of the upper-tail event occurring without the appearance of a small
$(1-\eps)t$-seed by applying Markov's inequality to the variable $X^{m_{\max}((1-\eps)t,C)/k} \cdot
Z_{(1-\eps)t}$. However, this bound is only useful when it is smaller 
than the available lower bound on the upper-tail. Unfortunately, there
is a small portion of the localised regime where this is not the case.  Specifically, this
happens when $\mu \ll t \le O\big(\big(\log(1/p)\big)^2\big)$;  we shall
informally say that this case falls into the \emph{very sparse localised
regime}. In this regime, instead of bounding the classical moments of $X$ on
the event that $\bR$ does not contain a small $(1-\eps)t$-seed, we will need to bound
factorial moments of $X$.  More precisely, given an integer $r$, let $\ff{X}{r}
\coloneqq X(X-1) \dotsb (X-r+1)$ denote the $r$-th falling factorial of $X$.
The following estimate is a special case of the more general
\cref{prop:factorial-moment-no-seed}, which is proved
in~\cref{sec:upper-bound-Poisson}.

\begin{corollary}
  \label{cor:JORVSLocalised}
  For every $k\geq 3$, there exists a constant $K$ such that the following
  holds for all $C,\eps > 0$. Suppose that $\Omega(N^{-2/k}) \le p \ll N^{-1/(k-1)}$
  and that $t$ is a positive integer satisfying $t \le
  \big(\log(1/p)\big)^3$. Then, for all large enough $N$ and all $u \le t$,
  \[
    \Ex\left[\ff{X}{t} \cdot Z_u\right] \le \mu^t \cdot
    \exp \big((u + \eps t/2)\log(1 + Kt/\mu)\big).
  \]
\end{corollary}

\begin{proof}[Proof of~\cref{prop:LocUBSeed}]
  Fix an $\eps >0$, let $u \coloneqq (1-\eps)t$, and define
  \[
    m \coloneqq \left\lceil\frac{2k \cdot (\mu + t)\cdot \log(1/p)}{\eps \sqrt{t}}\right\rceil
  \]

  We will first consider the case where $km\leq m_{\max}(u,C)$.
  Let $Z \coloneqq Z(u, km)$ and note that,
  as $km\leq m_{\max}(u,C)$, we have $Z\geq Z_u$. Since $Z$ is an
  indicator random variable, Markov's inequality and~\cref{lemma:JOR} then give
  \[
    \Pr\big(X \ge \mu + t\big) \le \Pr \big(X \cdot Z \ge \mu + t\big)  + \Pr(Z=0)
    \le \frac{\Ex[X^m \cdot Z]}{(\mu + t)^m} + \Pr(Z=0)
    \le \left(\frac{\mu +u}{\mu + t}\right)^m + \Pr(Z_u=0).
  \]
  Since
  \[
    \Pr(Z_u = 0) = \Pr\left(\text{$U \subseteq \bR$ for some $U \in \cSb\big((1-\eps)t, C\big)$}\right),
  \]
  the proposition will follow from
  \[
    \left(\frac{\mu + u}{\mu + t}\right)^m \le \frac{\eps}{1+\eps} \cdot \Pr\left(X \ge \mu + t\right).
  \]
  To this end, observe that, by the definitions of $u$ and $m$,
  \[
    \left(\frac{\mu + u}{\mu + t}\right)^{m} = \left(1 - \frac{\eps
      t}{\mu +t}\right)^{m} \le \exp\left(- \frac{ \eps t m }{\mu +
    t}\right) \le \exp\left(- 2 k \sqrt{t} \cdot\log (1/p)\right).
  \]
  The assumptions of~\cref{prop:Psi-Psi-star} hold throughout the localised
 regime, and so, thanks
 to~\cref{prop:LocalisedTheoremSection,prop:Psi-Psi-star}, we may conclude that 
  \[
    \log \Pr\left(X \ge \mu + t\right) \ge  -  (1+ \eps) \cdot \Psi^*\big((1 + \eps) t\big) \cdot \log(1/p) \ge - (1 +3 \eps) \cdot \sqrt{2 (k-1) t} \cdot \log(1/p). 
  \]
  Finally, since $t \gg 1$ in the entire localised regime, we obtain
  \[
    \Pr\left(X \ge \mu + t\right) \ge \frac{1+\eps}{\eps}  \cdot  \exp\left(- 2 k \sqrt{t} \cdot\log (1/p)\right) \ge \frac{1+\eps}{\eps} \cdot \left(\frac{\mu+u}{\mu+t}\right)^m
  \]
  for all large enough $N$. This gives the assertion of the proposition in the 
  case where $km \le m_{\max}(u, C)$.  The following claim states that
  this inequality holds unless $\mu \ll t = O\big(\big(\log(1/p)\big)^2\big)$.

  \begin{claim}
    We have $km \le m_{\max}(u,C)$ unless $\mu \ll t \le
    C'\big(\log(1/p)\big)^2$ for some constant $C' = C'(k,C)$.
  \end{claim}
  \begin{proof}
    Observe first that the inequality $km \le m_{\max}(u,C)$ is equivalent to
    \begin{equation}
      \label{eq:u-km-condition}
      u \ge Ckm \cdot \max\{1, Np^{k-1}\}
      \qquad
      \text{and}
      \qquad
      u \ge Ck^2m^2p^{k-2} \cdot N^{(k-2)(km/u)^{1/(k-1)}}.
    \end{equation}
    Observe further that, if $C$ is sufficiently large, then the first
    inequality in~\eqref{eq:u-km-condition} implies that
    \[
      N^{(k-2)(km/u)^{1/(k-1)}} \le N^{(k-2)(C \max\{1,Np^{k-1}\})^{-1/(k-1)}} \le p^{-1/4}
    \]
    for all sufficiently large $N$, where the last inequality can be verified
    by considering separately the cases $Np^{k-1}\leq N^\delta$ and $Np^{k-1}> N^\delta$ for
    some small enough $\delta>0$.
    In particular~\eqref{eq:u-km-condition} follows from
    \begin{equation}
      \label{eq:u-km-condition-simpler}
      u \ge Ckm \cdot \max\{1, Np^{k-1}\}
      \qquad
      \text{and}
      \qquad
      u \ge Ck^2m^2p^{k-9/4}.
    \end{equation}
    Further, by the definition of $m$,
    \[
      \frac{u}{m} \ge \frac{t}{2m} \ge \frac{\eps}{6k\log(1/p)} \cdot \frac{t^{3/2}}{\mu+t}
      \qquad
      \text{and}
      \qquad
      \frac{u}{m^2} \ge \frac{t}{2m^2} \ge \left(\frac{\eps}{3k\log(1/p)} \cdot \frac{t}{\mu+t}\right)^2
    \]
    and, using also $p^{(4k-9)/8}\log(1/p) \ll p^{(k-2)/3}$ for $k\geq 3$, it is thus enough to show that
    \begin{equation}
      \label{eq:u-km-condition-simplest}
      \frac{t^{3/2}}{\mu + t} \ge C'\max\{1, Np^{k-1}\}\log(1/p)
      \qquad
      \text{and}
      \qquad
      \frac{t}{\mu+t} \ge 
      p^{(k-2)/3}
    \end{equation}
    for some $C' = C'(C,k,\eps)$, unless $\mu \ll t \le 2C'\big(\log(1/p)\big)^2$.

    Assume first that $Np^{k-1} > 1$ and thus $\sigma^2 = \Theta(N^3p^{2k-1})$
    and $\sqrt{t} \log(1/p) \ll t^2/\sigma^2$, by the definition of the
    localised regime.  In particular,
    we may assume that, for every $K > 0$ and
    all sufficiently large $N$, we have $t \ge t_K$, where
    \[
      t_K \coloneqq \big(K\sigma^2\log(1/p)\big)^{2/3} =
      \Theta\big(K^{2/3}N^2p^{(4k-2)/3}\log(1/p)^{2/3}\big) \ll \mu;
    \]
    the last inequality holds as $k\geq 3$ and $p\ll 1$.
    Since the function $(0, \infty) \ni x \mapsto \frac{x^{3/2}}{\mu + x}$ is
    increasing, we have
    \[
      \frac{t^{3/2}}{\mu+t} \ge \frac{t_K^{3/2}}{\mu + t_K} \ge \frac{K\sigma^2\log(1/p)}{2\mu} \ge Kc_kNp^{k-1}\log(1/p),
    \]
    for some $c_k>0$. Choosing $K$ appropriately, we thus obtain the first inequality in
    \eqref{eq:u-km-condition-simplest}.
    Since the function $(0, \infty) \in x \mapsto \frac{x}{\mu+x}$ is also increasing,
    \[
      \frac{t}{\mu+t} \ge \frac{t_K}{\mu+t_K} \ge
      \frac{\big(K\sigma^2\log(1/p)\big)^{2/3}}{2\mu} \ge
      \frac{K^{2/3}\sigma^{4/3}}{2\mu} \ge K^{2/3} c_k' \cdot
      \frac{N^2p^{(4k-2)/3}}{N^2p^k} \ge p^{(k-2)/3},
    \]
    for some $c_k'>0$, giving the second inequality
    in~\eqref{eq:u-km-condition-simplest}.

    Assume now that $Np^{k-1} \le 1$ and thus $\sqrt{t} \log(1/p) \ll \mu \cdot
    \Po(t/\mu)$.  If $t = O(\mu)$, then the estimate $\mu \cdot \Po(t/\mu) \le
    t^2/\mu$ immediately implies the first inequality
    in~\eqref{eq:u-km-condition-simplest}; further, since the same estimate
    implies that $t^{3/2} \gg \mu$, we obtain, using $Np^{k-1} \le 1$, that
    \[
      \frac{t}{\mu+t} \ge \frac{t}{O(\mu)} \gg \mu^{-1/3} \ge \big(N^2p^k\big)^{-1/3} \ge p^{(k-2)/3},
    \]
    which gives the second inequality in~\eqref{eq:u-km-condition-simplest}.
    We may therefore assume that $t \gg \mu$.  In this case,
    \[
      \frac{t^{3/2}}{\mu+t} \ge \frac{\sqrt{t}}{2}
      \qquad
      \text{and}
      \qquad
      \frac{t}{\mu+t} \ge \frac{1}{2} \ge p^{(k-2)/3}.
    \]
    In particular, both inequalities in~\eqref{eq:u-km-condition-simplest} hold
    unless $\mu \ll t \le \big(2C'\log(1/p)\big)^2$.
  \end{proof}

  Assume now that $\mu \ll t \le \big(\log(1/p)\big)^3$. Since $X \ge \mu+t$
  implies $\ff{X}{t} \ge \ff{(\mu+t)}{t}$, and since $Z_u$ is an indicator random
  variable, we have
  \[
    \Pr(X \ge \mu + t) \le \Pr\big(\ff{X}{t} \cdot Z_u \ge \ff{(\mu+t)}{t} \big) + \Pr(Z_u=0).
  \]
  Since $\ff{X}{t} \cdot Z_u \ge 0$, as $X$ is integer-valued, we may apply
  Markov's inequality and \cref{cor:JORVSLocalised} to conclude that
  \[
    \Pr\big(\ff{X}{t} \cdot Z_u \ge \ff{(\mu+t)}{t} \big) \le \frac{\Ex[\ff{X}{t} \cdot Z_u]}{\ff{(\mu+t)}{t}} \le \frac{\mu^t}{\ff{(\mu+t)}{t}} \cdot \exp \big((u + \eps t/2)\log(1 + Kt/\mu)\big).
  \]
  Further,
  \[
    \log\frac{\mu^t}{\ff{(\mu+t)}{t}} \le - \int_0^t \log\frac{\mu+x}{\mu}\, dx = - \mu \cdot \int_0^{t/\mu} \log(1+y) \, dy = - \mu \cdot \Po\left(\frac{t}{\mu}\right),
  \]
  whereas the assumption $t \gg \mu$ gives
  \[
    t \log\left(1+\frac{Kt}{\mu}\right) = (1+o(1)) \cdot \left((\mu+t) \log\left(1+\frac{t}{\mu}\right) - t\right) = (1+o(1)) \cdot \mu \cdot \Po\left(\frac{t}{\mu}\right).
  \]
  Consequently,
  \[
    \Pr(X \ge \mu + t) \le \exp\big(- \eps\mu/4 \cdot \Po(t/\mu)\big) + \Pr(Z_u=0).
  \]
  Finally, since $\mu \cdot \Po(t/\mu) \gg \sqrt{t} \log(1/p)$ in the localised regime,
  \[
    \exp\big(- \eps\mu/4 \cdot \Po(t/\mu)\big) \le \exp\left(- 2 k \sqrt{t} \cdot\log (1/p)\right) \le \frac{\eps}{1+\eps} \cdot \Pr\left(X \geq \mu + t\right),
  \]
  which implies the assertion of the proposition.
\end{proof}

\section{The Gaussian regime}
\label{sec:Gaussian}

In this section, we prove~\cref{thm:MainResult} in the Gaussian regime, i.e.,
for all sequences $(p,t)$ satisfying
\[
  Np^{k-1} \gg 1, \quad t \gg \sigma, \quad \text{ and } \quad \sqrt{t}
  \log (1/p) \gg t^2/\sigma^2.
\]
The lower and the upper bounds on the upper-tail probability of $X$ will be
proved separately. In \cref{sec:lower-bound-Gaussian}, we prove the lower bound
by applying~\cref{prop:lower-bound} to a product measure $\Qr$ that is a small
perturbation of the $p$-biased product measure $\Pr$. In
\cref{sec:upper-bound-Gaussian}, we prove the matching upper bound by applying
our martingale concentration inequality (\cref{prop:GaussianTruncationBound})
to the hypergraph $\APk$ and bounding the three error terms using combinatorial
arguments.

Throughout the section, it will be convenient to work with an expression that
closely approximates the variance of $X$ and involves only the numbers of
$k$-APs that contain a given $i \in \br{N}$. For any $i \in \br{N}$, denote by
$A_1(i)$ the number of $k$-term arithmetic progressions that contain $i$
(recalling the notation $A_r^{(k)}(U)$ used in Section~\ref{sec:main-techn-result},
we can see that this definition corresponds to $A_1^{(k)}(\{i\})$). Define
\[
  \cV \coloneqq \sum_{i \in \br{N}} A_1(i)^2 \cdot p^{2k-1}.
\]

\begin{lemma}
  \label{lemma:cV}
  For every $\eps>0$, there exists a $C$ such that, whevener $CN^{-1/(k-1)} \le p \le 1/C$,
  \[
    (1-\eps) \cdot \cV \le \Var(X) \le (1+\eps) \cdot \cV.
  \]
\end{lemma}
\begin{proof}
  Since
  \[
    \Var(X) - \cV = \sum_{B,B' \in \APk}
    \left(p^{|B \cup B'|} - p^{2k} - |B \cap B'| \cdot p^{2k-1}\right),
  \]
  we have
  \[
    |\Var(X) - \cV| \le \left|\big\{(B, B') \in \APk^2 : |B\cap B'| = 1\big\}\right| \cdot p^{2k}
    + \left|\big\{(B, B') \in \APk^2 : |B\cap B'| \ge 2\big\}\right| \cdot p^k.
  \]
  It is easy to see that the first term in the right-hand side is at most $p
  \cdot \cV$.  Moreover, as every pair of elements of $\br{N}$ is contained in
  at most $\binom{k}{2}$ arithmetic progressions of length $k$, the second term
  is at most $|\APk| \cdot \binom{k}{2}^2 \cdot p^k$.
  Finally, since $|\APk| = \Theta(N^2)$ and $\cV = \Theta(N^3p^{2k-1})$, the
  assertion of the lemma now follows from the assumptions on $p$
  provided that $C$ is sufficiently large as a function of $\eps$.
\end{proof}

\subsection{Proof of the lower bound in the Gaussian regime}
\label{sec:lower-bound-Gaussian}

We call a measure $\Qr$ supported on subsets of $\br{N}$ a \emph{$p$-bounded
product measure} if it is a product measure with $\Qr(i \in \mathbf{R}) \in
[p,2p]$ for every $i \in \br{N}$.  We first show that the variance of $X$ under
an arbitrary $p$-bounded product measure is not much larger than the
variance $\sigma^2$ taken with respect to the $p$-biased product measure $\Pr$:

\begin{lemma}\label{lemma:TiltedVariance}
  Assume that $p \le 1/2$ and that $\Qr$ is a $p$-bounded product measure.
  Then, for all sufficiently large $N$,
  \[
    \Var_{\Qr}(X) \le  2^{2k} \cdot \sigma^2.
  \]
\end{lemma}
\begin{proof}
  On the one hand,
  \[
    \begin{split}
      \Var_{\Qr}(X)
      & = \sum_{\substack{B,B' \in \APk \\ B \cap B' \neq \emptyset}}
      \big(\Qr(B \cup B' \subseteq \bR) - \Qr( B \subseteq \bR)\cdot  \Qr( B'
      \subseteq \bR)\big) \\
      & \le  \sum_{\substack{B,B' \in \APk \\ B \cap B' \neq \emptyset}}
      2^{|B\cup B'|} \cdot \Pr( B \cup B' \subseteq \bR)
      \le 2^{2k-1} \cdot \sum_{\substack{B,B' \in \APk \\ B \cap B' \neq
      \emptyset}} \Pr(B \cup B' \subseteq \bR).
    \end{split}
  \]
  On the other hand,
  \[
    \sigma^2 = \Var(X)=  \sum_{\substack{B,B' \in \APk \\ B \cap B' \neq
    \emptyset}} \big( \Pr(B \cup B' \subseteq \bR) - \Pr( B \subseteq \bR)
    \cdot \Pr( B' \subseteq \bR) \big) \ge \frac{1}{2}  \sum_{\substack{B,B'
    \in \APk \\ B \cap B' \neq \emptyset}} \Pr( B \cup B' \subseteq \bR), 
  \]
  where the inequality follows from $p \le 1/2$. 
\end{proof}


We restate the lower bound of~\cref{thm:MainResult} in the Gaussian regime in a
non-asymptotic form and prove it by applying \cref{prop:lower-bound} to a
well-chosen $p$-bounded product measure.  Note that we state and prove the
lower bound on the upper-tail probability of $X$ in a larger region of the
parameter space that includes also a part of the localised regime (where this
Gaussian-like lower estimate is no longer optimal).

\begin{proposition}\label{prop:GaussianLowerBound}
  For every $\eps >0$ and integer $k\geq 3$, there exists $C>0$ such that the following holds. If
  \[
    C\cdot N^{-1/(k-1)} \le p \le 1/C
    \qquad \text{and} \qquad
    C \cdot \sigma \le t \le \mu/C,
  \]
  then
  \[
    \log \Pr\left(X \ge \mu +t \right) \ge - (1 + \eps)\cdot  \frac{t^2}{2 \sigma^2}.
  \]
\end{proposition}
\begin{proof}
  For each $i \in \br{N}$, set
  \[
    q_i \coloneqq \frac{(1+\eps)tp^k}{\cV} \cdot A_1(i)
  \]
  and let $\Qr$ be the product measure given by $\Qr(i \in \bR) = p+q_i$ for all $i$.  Since $A_1(i) \le kN$ for all $i$ and $\sigma^2 = \Theta(\mu \cdot Np^{k-1})$, \cref{lemma:cV} implies that
  \[
    \max_{i \in \br{N}} \frac{q_i}{p} \le \frac{(1+\eps)ktNp^{k-1}}{\cV} \le \frac{(1+\eps)^2ktNp^{k-1}}{\sigma^2 } \le \frac{(1+\eps)^2 k \mu \cdot Np^{k-1}}{C \sigma^2 } \le  \frac{c_k}{C}
  \]
  for some $c_k$ that depends only on $k$;  in particular, $\Qr$ is
  $p$-bounded, provided that $C$ is large.  Crucially,
  \[
    \begin{split}
      \Ex_{\Qr}[X] - \mu & = \sum_{B \in \APk} \big(\Qr(B \subseteq \bR) - \Pr(B \subseteq \bR) \big) = \sum_{B \in \APk} \left(\prod_{i \in B} (p+q_i) - p^k\right) \\
      & \ge \sum_{B \in \APk} \sum_{i\in B}q_ip^{k-1} = \sum_{i \in \br{N}} A_1(i) q_i p^{k-1} = \frac{(1 + \eps)t}{\cV} \cdot \sum_{i \in \br{N}} A_1(i)^2 p^{2k-1} = (1 + \eps) t
  \end{split}
  \]
  and consequently, by Chebyshev's inequality and \cref{lemma:TiltedVariance}, 
  \[
    \Qr( X < \mu + t) \le \Qr( X < \Ex_{\Qr}[X] -\eps t) \le
    \frac{\Var_{\Qr}(X)}{\eps^2 t^2} \le \frac{2^{2k} \sigma^2 }{\eps^2 t^2}
    \le \frac{2^{2k}}{\eps^2 C^2}.
  \]
  In particular, if $C$ is sufficiently large as a function of $\eps$, \cref{prop:lower-bound} implies that
  \[
    \Pr(X \ge \mu + t) \ge - (1 + \eps) \cdot \DKL(\Qr \, \| \, \Pr).
  \]
  Since both $\Pr$ and $\Qr$ are products of Bernoulli distributions, \cref{fact:DKL-additivity} implies that
  \[
    \DKL(\Qr \, \| \, \Pr) = \sum_{i=1}^N \DKL\big(\Ber(p+q_i) \, \| \, \Ber(p)\big) = \sum_{i=1}^N j_p(p + q_i),
  \]
  where $j_p(x) \coloneqq x \log(x/p) + (1-x) \log( (1-x)/(1-p))$. It is
  straightforward to verify that $j_p(p) = j_p'(p) =0$ and that $j_p''(x) =
  1/x(1-x)$. Finally, by expanding $j_p$ in Taylor series with Lagrange
  remainder and using the assumption $p \le 1/C$, we can conclude that, whenever
  $C$ is sufficiently large as a function of $\eps$, 
  \[
    \DKL(\Qr \, \| \, \Pr) = \sum_{i \in \br{N}} j_p(p + q_i) \le \sum_{i \in
    \br{N}} \frac{q_i^2}{2p(1 - p)} = \frac{(1 + \eps)^2 t^2}{2(1-p)\cV} \le
    (1+\eps)^4 \cdot \frac{t^2}{2\sigma^2},
  \]
  completing the proof. 
\end{proof}

\subsection{Proof of the upper bound in the Gaussian regime}
\label{sec:upper-bound-Gaussian}

Define, for every $i \in \br{N}$,
\[
  \cB_k'(i) \coloneqq \big\{B \setminus \{i\} : B \in \APk \text{ and } i \in B\big\},
\]
and let $L_i \coloneqq |\{B \in \cB_k'(i): B \subseteq \bR\}|$; in other words,
$\cB_k'(i)$ is the link of $i$ in the hypergraph $\APk$, cf.\ the definition of
$L_i$ given in~\eqref{eq:Li-definition}. We will prove the claimed
Gaussian-like upper bound for the upper tail of $X$ by applying
\cref{prop:GaussianTruncationBound} and showing that, for all sequences $(p,t)$
in the Gaussian regime, the three error terms involving upper-tail estimates on
the respective functionals of $L_i$ are dominated by the Gaussian term.
In particular, it suffices to establish the following three estimates:

\begin{proposition}\label{prop:threeitems}
  For all $\eps >0$ and $k\geq 3$, there exists $C$ such that the following holds. If
  \[
    C\cdot N^{-1/(k-1)} \le p \le 1/C
    \qquad \text{and} \qquad
    C \cdot\sigma \le t \le  \frac{(\sigma^2 \log(1/p))^{2/3}}{C},
  \]
  then
  \begin{align}
    \label{eq:OneLarge}
    \Pr \left( \exists i \;\;  L_i > \frac{ \sigma^2 \log(1/p)}{2t} \right) & < \exp \left(-\frac{t^2}{\sigma^2} \right),\\
    \label{eq:FewMedium}
    \Pr \left(\left| \left\{i : L_i > \frac{\sigma^2}{t}\right\} \right| \ge \frac{\eps t^2 }{\sigma^2 p^{1/2}} \right) & < N^{-1} \cdot \exp \left(-\frac{t^2}{\sigma^2} \right), \\
    \label{eq:SquareDeviation}
    \Pr\left(\sum_{i=1}^N L_i^2 > \left(1 + \eps \right) \cdot \frac{\sigma^2}{p} \right) & < N^{-1} \cdot \exp \left(-\frac{t^2}{\sigma^2} \right).
  \end{align}
\end{proposition}
Since the asymptotic inequality $\sqrt{t} \log(1/p) \gg t^2/\sigma^2$ clearly
implies that $t \le \big(\sigma^2 \log(1/p)\big)^{2/3}/C$ for all $C$ and all
sufficiently large $N$, \cref{prop:GaussianTruncationBound,prop:threeitems}
yield the claimed Gaussian-like upper bound on the upper-tail probability of
$X$ in the entire Gaussian regime.

The remainder of this section is dedicated to proving the three estimates
of~\cref{prop:threeitems}, sequentially.  We will prove \eqref{eq:OneLarge}
(resp.\ \eqref{eq:FewMedium}) by showing that the corresponding upper-tail
events imply the appearance of a large collection of pairwise-disjoint elements
of $\cB_k'(i)$ (resp. `nearly' pairwise-disjoint elements of
$\bigcup\{\cB_k'(i) : L_i > \sigma^2/t\}$) in $\bR$ and then bounding the
probability of the latter event using a straightforward first-moment
calculation.  The proof of~\eqref{eq:SquareDeviation} will be another
adaptation of the elegant argument from~\cite{janson2004upper}, which played a
key role in our proof of the upper bound in the localised regime.  We begin
with a helpful combinatorial fact that will help us find large collections of
pairwise-disjoint elements of $\cB_k'(i)$.

\begin{fact}
  \label{fact:link-pairwise-disjoint}
  For each $i \in \br{N}$, every collection $\cB \subseteq \cB_k'(i)$ contains
  a subcollection $\cD \subseteq \cB$ with $|\cD| \ge |\cB|/k^3$ whose elements
  are pairwise disjoint.
\end{fact}
\begin{proof}
  Since any two numbers are contained in at most $\binom{k}{2}$ \kap{}s, it
  follows that every $B \in \cB_k'(i)$ intersects at most
  $(k-1)\binom{k}{2}\leq k^3$ sets from $\cB_k'(i)$, including $B$ itself.
  In particular, given an arbitrary collection $\cB\subseteq \cB_k'(i)$, we may
  construct a family $\cD \subseteq \cB$ of pairwise-disjoint sets greedily by
  adding sets $B \in \cB$ one by one, each time removing the (at most $k^3$)
  sets that intersect $B$ from $\cB$.  The resulting
  collection $\cD$ satisfies $|\cD| \ge |\cB|/k^3$.
\end{proof}

Let us also note that the assumptions of \cref{prop:threeitems} imply that
\begin{equation}
  \label{eq:u-threeitems-lower}
  \frac{\sigma^2}{t} = \frac{(\sigma^2\log(1/p))^{2/3}\cdot \sigma^{2/3}}{t \cdot (\log(1/p))^{2/3}}
  \geq \frac{C\sigma^{2/3}}{(\log(1/p))^{2/3}}\geq
  \frac{\sqrt{C}Np^{2k/3-1/3}}{(\log(1/p))^{2/3}}
  \geq \sqrt{C}Np^{k-1}\cdot p^{-1/4},
\end{equation}
whenever $C$ is chosen to be sufficiently large as a function of $k$.

\begin{proof}[Proof of \eqref{eq:OneLarge} in \cref{prop:threeitems}]
  Set $u \coloneqq \sigma^2 \log(1/p)/(2t)$. By
  \cref{fact:link-pairwise-disjoint}, the inequality $L_i \ge u$ implies that
  there is a collection $\cD \subseteq \cB_k'(i)$ of pairwise-disjoint subsets
  of $\bR$ with cardinality at least $b \coloneqq \lceil u/k^3 \rceil$. Let
  $Z_i$ denote the number of such collections, so that
  \[
    \Pr(L_i \ge u) \le \Pr(Z_i = 1) \le \Ex[Z_i].
  \]
  Since $|\cB_k'(i)| \le kN$ and the union of all sets in each collection
  counted by $Z_i$ has $(k-1)b$ elements,
  \[
    \Ex[Z_i] \le \binom{kN}{b} \cdot p^{(k-1)b} \le
    \left(\frac{ekNp^{k-1}}{b}\right)^b \le
    \left(\frac{ek^4Np^{k-1}}{u}\right)^{u/k^3} \le p^{u/(4k^3)},
  \]
  where the penultimate inequality holds because $b \ge u/k^3 \ge kNp^{k-1}$
  and, when $a > 0$, the function $x \mapsto (ea/x)^x$ is decreasing on $[a,
  \infty)$ and the ultimate inequality holds when $C$ is sufficiently large as
  a function of $k$, by \eqref{eq:u-threeitems-lower}.
  By the union bound over all $i \in \br{N}$,
  \[
    \Pr\left(\exists i \;\; L_i \ge \frac{ \sigma^2 \log(1/p)}{2t}\right) \le N
    p^{u/(4k^3)}
    \le N \cdot \exp\left(\frac{-C\sigma^2 (\log(1/p))^2}{4k^3t}\right)
    \le \exp\left(\frac{-C\sigma^2 (\log(1/p))^2}{5k^3t}\right),
  \]
  where the last inequality holds as
  $\sigma^2/t \geq N^{1/(4k-4)}$
  by~\eqref{eq:u-threeitems-lower} and our
  lower-bound assumption on $p$. Finally, the upper-bound assumption on $t$ implies that
  \[\sigma^2(\log(1/p))^2/t \geq \sqrt{t}\log(1/p)\geq \frac{t^2}{\sigma^2},\]
  which results in the claimed bound.
\end{proof}

\begin{proof}[Proof of \eqref{eq:FewMedium} in \cref{prop:threeitems}]
  Set
  \[
    u \coloneqq \frac{ \sigma^2}{t}
    \qquad \text{ and } \qquad
    I \coloneqq \{i \in \br{N} : L_i > u\}.
  \]

  We first claim that, when $|I| > (m-1)^2 \cdot u$ for some positive integer
  $m$, then there exist distinct $i_1, \dotsc, i_m \in \br{N}$ and collections
  $\cD_1 \subseteq \cB_k'(i_1), \dotsc, \cD_m \subseteq \cB_k'(i_m)$, each
  comprising $b \coloneqq \lceil u / (4k^3) \rceil$ pairwise-disjoint subsets
  of $\bR$, such that, for every $j \in \br{m}$, letting
  \[
    D_{<j} \coloneqq \bigcup_{j' < j} \bigcup \cD_{j'},
  \]
  either
  \begin{enumerate}[label=(\roman*)]
  \item
    \label{item:cBj-disjoint}
    each $D \in \cD_j$ is disjoint from $D_{<j}$ or
  \item
    \label{item:cBj-one-element}
    each $D \in \cD_j$ intersects $D_{< j}$ in exactly one element.
  \end{enumerate}
  To see this, suppose that $1 \le j \le m $ and that sequences $i_1, \dotsc,
  i_{j-1}$ and $\cD_1, \dotsc, \cD_{j-1}$ with the above properties have
  already been defined.  Let $D_{<j}$ be the set defined above and define
  the sets
  \[
    \cX \coloneqq \big\{B \setminus \{i\} : i \in B \in \APk \text{ such that } |B \cap D_{<j}| \ge 2\big\}
  \]
  and
  \[
    I^{\cX} \coloneqq \big\{i \in I : |\cB_k'(i) \cap \cX| \ge u/2\big\}.
  \]
  The key observation is that, for every $i \in I \setminus I^{\cX}$, the
  family $\cB_k'(i)$ contains at least $u/2$ subsets of $\bR$ that intersect
  $D_{<j}$ in at most one element; in particular, either at least $u/4$ of them
  are disjoint from $D_{<j}$ or at least $u/4$ of them intersect it in exactly
  one element.  Consequently, Fact~\ref{fact:link-pairwise-disjoint} allows us
  to find a suitable $\cD_j$ for each $i_j \in I \setminus I^{\cX}$.  Thus, it
  suffices to show that $I \setminus I^{\cX}$ is not empty.  To this end, note
  that $|\cX| \le \binom{|D_{<j}|}{2} \cdot \binom{k}{2} \cdot k \le (m-1)^2b^2
  k^3$ and that $|\cX| \ge |I^{\cX}| \cdot u/6$, as each $B \subseteq \br{N}$
  belongs to $\cB_k'(i)$ for at most three different $i \in \br{N}$ (at most
  two different $i \in \br{N}$ if $k \ge 4$).  Thus,
  \[
    |I^{\cX}| \le \frac{6 (m-1)^2 b^2 k^3}{u} \le (m-1)^2 \cdot u < |I|.
  \]

  Now, define
  \[
    m \coloneqq \left\lceil \frac{\sqrt{\eps} t^{3/2}}{\sigma^2 p^{1/4}}\right\rceil
  \]
  and let $Z$ denote the number of pairs of sequences $(i_1, \dotsc, i_m)$ and
  $(\cD_1, \dotsc, \cD_m)$ as above, so that
  \[
    \Pr \left(\left| \left\{i : L_i > \frac{\sigma^2}{t}\right\} \right| \ge
    \frac{\eps t^2 }{\sigma^2 p^{1/2}} \right) \le \Pr\big(|I| > (m-1)^2
  \cdot u\big) \le \Pr(Z \ge 1) \le \Ex[Z].
  \]
  In order to estimate the expectation of $Z$, for any $J \subseteq \br{m}$,
  let $Z_J$ denote the expected number of such pairs of sequences for
  which~\ref{item:cBj-one-element} above holds for $j \in J$
  and~\ref{item:cBj-disjoint} above holds for $j \notin J$.  There are at most
  $N^m$ choices for $(i_1, \dotsc, i_m)$ and at most $\binom{kN}{b}$ further
  choices for each $\cD_j$, as each $\cD_j$ is a $b$-element subset of the set
  $\cB'_k(i_j)$, which has size at most $kN$. If $j \in J$, however, the number
  of choices for $\cD_j$ after $\cD_1, \dotsc, \cD_{j-1}$ have already been
  chosen is at most
  \[
  \binom{|D_{<j}|+b-1}{b} \cdot \binom{k}{2}^b \le \binom{mkb}{b} \cdot k^{2b} \le \big(emk^3\big)^b.
  \]
  Therefore,
  \[
    \begin{split}
      \Ex[Z_J] & \le N^m \cdot \left(\binom{kN}{b} \cdot p^{(k-1)b}\right)^{m-|J|} \cdot \left(\big(emk^3\big)^b \cdot p^{(k-2)b}\right)^{|J|} \\
      & \le N^m \cdot \left(\frac{ekNp^{k-1}}{b}\right)^{(m-|J|)b} \cdot \left(emk^3p^{k-2}\right)^{|J|b} \\
      & = N^m \cdot \left(\frac{ekNp^{k-1}}{b}\right)^{mb} \cdot \left(\frac{mk^2b}{Np}\right)^{|J|b}.
    \end{split}
  \]
  Since
  \[
    mk^2b \le mu = \frac{m\sigma^2}{t} \le \frac{\sqrt{t}}{p^{1/4}} \le \frac{\sigma^{2/3}
    \big(\log(1/p)\big)^{1/3}}{ p^{1/4}} \le  Np \cdot p^{(2k-4)/3 - 1/4}
    \big(\log(1/p)\big)^{1/3} \le Np,
  \]
  we may conclude that
  \[
    \begin{split}
    \Ex[Z] = \sum_{J \subseteq \br{m}} \Ex[Z_J] \le (2N)^m \cdot
    \left(\frac{ekNp^{k-1}}{b}\right)^{mb}
     & \le (2N)^m \cdot
     \left(\frac{4ek^4Np^{k-1}}{u}\right)^{mu/(4k^3)} \\
     & \le (2N)^m \cdot p^{mu/(16k^3)},
    \end{split}
  \]
  where the last inequality follows from~\eqref{eq:u-threeitems-lower} and the
  penultimate inequality holds because $b \ge u/(4k^3)$ and, when $a > 0$, the
  function $x \mapsto (ea/x)^x$ is decreasing on $[a, \infty)$.

  We conclude that
  \[
    \Pr \left(\left| \left\{i : L_i > \frac{\sigma^2}{t}\right\} \right| \ge
    \frac{\eps t^2 }{\sigma^2 p^{1/2}} \right)
    \le N^{-m} \cdot \left((2N^2)^{1/u} \cdot p^{1/(16k^3)}\right)^{mu} \le N^{-m} \cdot \exp ( - mu ),
  \]
  where the final inequality follows since $u \geq N^{1/(4k-4)}$, again
  by~\eqref{eq:u-threeitems-lower} and our lower-bound assumption on $p$.
  Finally, since $m \ge 1$, \eqref{eq:FewMedium} follows as
  \[
    mu  \ge  \frac{\sqrt{\eps t}}{2 p^{1/4}} \ge \sqrt{t} \log(1/p) \ge t^2/\sigma^2.\qedhere
  \]
\end{proof}

The following key lemma will allow us to deduce the claimed upper bound on the
upper-tail probability of $L_1^2 + \dotsb + L_N^2$.  As we remarked above, its
proof is another adaptation of the elegant argument of~\cite{janson2004upper}.

\begin{lemma}
  \label{lemma:V-high-moments}
  Let $\beta > 0$ and $k\geq 3$, and assume that $Np^{k-1} > 1$. Then there
  exists an $\alpha > 0$ depending only on $\beta$ and $k$ such that, for every
  nonnegative integer $\ell \le \alpha N p^{(2k-2)/3}$, letting $V \coloneqq
  L_1^2 + \dotsb + L_N^2$, we have
  \[
    \Ex[V^\ell] \le (1+\beta)^\ell \cdot \Ex[V]^{\ell}.
  \]
\end{lemma}

\begin{proof}
  We prove the claimed estimate by induction on $\ell$.  The inequality holds
  vacuously when $\ell \le 1$, so we may assume that $\ell \ge 2$.  Define the
  multiset\footnote{The map $\bigcup_{i=1}^N \big(\cB_k'(i)\big)^2 \ni (B, B')
  \mapsto B \cup B' \in \cV_k$ is generally not injective.  For example, $(\{1,
  2, 3\} \cup \{3, 4, 5\}) \setminus \{3\} = (\{1, 3, 5\} \cup \{2, 3, 4\})
  \setminus \{3\}$. }
  \[
    \cV_k \coloneqq \bigcup_{i=1}^N \big\{B \cup B' : B, B' \in \cB_k'(i)\big\} = \big\{(B \cup B') \setminus \{i\} : B, B' \in \APk, i \in B \cap B'\big\}
  \]
  so that $V$ counts (with multiplicities) sets in $\cV_k$ that are contained in $\bR$.  We have
  \[
    \begin{split}
      \Ex[V^\ell] & = \sum_{P_1, \dotsc, P_\ell \in \cV_k} \Pr\big(P_1 \cup
      \dotsb \cup P_\ell \subseteq \bR\big) \\
      & = \sum_{P_1, \dotsc, P_{\ell-1} \in \cV_k} \Pr\big(P_1 \cup \dotsb \cup
      P_{\ell-1} \subseteq \bR\big) \cdot \sum_{P_\ell \in \cV_k}
      \Pr\big(P_\ell \subseteq \bR \mid P_1 \cup \dotsb \cup P_{\ell-1}
      \subseteq \bR\big).
    \end{split}
  \]
  Note that the first sum above equals $\Ex[V^{\ell-1}]$ and the second sum is
  $\Ex_{P_1 \cup \dotsb \cup P_{\ell-1}}[V]$.  As each $P \in \cV_k$ has at
  most $2k-2$ elements, we may use our inductive assumption to conclude that
  \[
    \Ex[V^\ell] \le (1+\beta)^{\ell-1} \cdot \Ex[V]^{\ell-1} \cdot \max\big\{\Ex_U[V] : |U| \le 2k\ell\big\}.
  \]
  In particular, it suffices to show that the maximum above is at most $(1+\beta) \cdot \Ex[V]$, provided that $\ell \le \alpha N p^{(2k-2)/3}$ for some small $\alpha = \alpha(\beta, k)$.  
  \begin{claim}
    There is a constant $C$ depending only on $k$ such that, for every $U \subseteq \br{N}$,
    \[
      \Ex_U[V] - \Ex[V] \le C \cdot \left(|U|^3 + |U|^2Np^{k-1} + |U|N^2p^{2k-3}\right).
    \]
  \end{claim}
  \begin{proof}
    Observe that
    \[
      \Ex_U[V] - \Ex[V] = \sum_{P \in \cV_k} \big(p^{|P \setminus U|} - p^{|P|}\big) \le \sum_{\substack{P \in \cV_k \\ P \cap U \neq \emptyset}} p^{|P \setminus U|}.
    \]
    For every $s \ge 0$, let $v_s(U)$ denote the number of $P \in \cV_k$ with $|P \setminus U| = s$.  Since every pair of elements of $\br{N}$ lies in at most $\binom{k}{2}$ arithmetic progressions of length $k$, one can check that, for some constant $C$ that depends only on $k$,
    \[
      v_s(U) \le
      \begin{cases}
        C|U|^3 & \text{if $s \le k-3$,} \\
        C|U|^3 + C|U|N & \text{if $s \le k-2$,} \\
        C|U|^2N & \text{if $s \le 2k-4$,} \\
        C|U|N^2 & \text{if $s \le 2k-3$.} \\
      \end{cases}
    \]
    (The odd term $C|U|N$ corresponds to `degenerate' $P$ of the form $(B \cup B) \setminus \{i\}$, where $B \in \APk$ and $i \in B$, that intersect $U$ in exactly one element.)
    We may thus conclude that
    \[
      \Ex_U[V] - \Ex[V] \le \sum_{s = 0}^{2k-3} v_s(U) \cdot p^s \le kC|U|^3 + C|U|Np^{k-2} + kC|U|^2Np^{k-1} + C|U|N^2p^{2k-3}.
    \]
    The claim follows, since $Np^{k-2} \le N^2 p^{2k-3}$. 
  \end{proof}
  Finally, the claim implies that, for every $U \subseteq \br{N}$ with $|U| \le 2k\alpha Np^{(2k-2)/3}$,
  \[
    \Ex_U[V] - \Ex[V] \le (2k)^3 \alpha \cdot N^3p^{2k-2} \cdot C \cdot \left(1 + p^{(k-1)/3} + p^{(2k-5)/3}\right) \le \beta \cdot \Ex[V],
  \]
  where the last inequality holds for all sufficiently small $\alpha$, as $\Ex[V] \ge c_k N^3 p^{2k-2}$ for some positive $c_k$ that depends only on $k$.
\end{proof}

\begin{proof}[Proof of \eqref{eq:SquareDeviation} in \cref{prop:threeitems}]
  Observe first that, for every $i \in \br{N}$,
  \[
    \begin{split}
      \Ex[L_i^2] & = \sum_{B, B' \in \cB'_k(i)} p^{|B \cup B'|} \le A_1(i)^2 \cdot p^{2k-2} + \sum_{\substack{B, B' \in \cB'_k(i)\\ B \cap B' \neq \emptyset}} p^{k-1} \\
      & \le A_1(i)^2 \cdot p^{2k-2} + A_1(i) \cdot (k-1)\binom{k}{2}p^{k-1} \le A_1(i)^2 \cdot p^{2k-2} \cdot \left(1 + \frac{k^4}{Np^{k-1}}\right),
    \end{split}
  \]
  where the finaly inequality hods as $A_1(i) \ge  N/k$ for every $i$.  Since
  $Np^{k-1} \ge C$, summing the above estimate over all $i \in \br{N}$, and
  recalling the definition of $\cV$, we obtain
  \[
    \Ex[V] = \sum_{i=1}^N \Ex[L_i^2] \le \frac{(1+\eta)  \cdot \cV}{p}
  \]
  for every constant $\eta>0$ and all sufficiently large values of $C$.  By~\cref{lemma:cV}, we may
  conclude that $\sigma^2 /p \ge (1 - 2 \eta) \Ex[V]$.  Let $\ell \coloneqq
  \lfloor \alpha Np^{(2k-2)/3}\rfloor$. Let $\beta>0$ be such that $1+\beta = (1+\eps/2)^{1/2}$.
  Then by \cref{lemma:V-high-moments} and Markov's inequality, 
  \[
    \Pr\left(V > \left(1 + \eps \right) \cdot \frac{\sigma^2 }{p} \right) \le \Pr\left(V > \left(1 + \eps/2 \right) \cdot \Ex[V] \right)
    \le \frac{\Ex[V^{\ell}]}{(1 + \eps/2)^{\ell} \cdot \Ex[V]^{\ell}}
    \le 
    \left(\frac1{1+\eps/2}\right)^{\alpha  Np^{(2k-2)/3}/2}.
  \]
  Finally, as $p\geq CN^{-1/(k-1)}$, we have
  \[
  Np^{(2k-2)/3} \geq C N^{1/3} \geq C\log N\]
  and, as $\sigma^2\log(1/p) \geq (Ct)^{3/2}$ and $p\leq 1/C$,
  \[
    Np^{(2k-2)/3} \ge (\sigma^2/p)^{1/3} \ge (\sigma^2\log(1/p))^{4/3}/\sigma^2
    \ge Ct^2/\sigma^2.
  \]
  From these, the claim follows.
\end{proof}

\section{The Poisson regime}
\label{sec:Poisson}

In this section, we prove~\cref{thm:MainResult} in the Poisson regime, i.e.,
for all sequences $(p,t)$ satisfying
\[
  \Omega(N^{-2/k}) \leq p \ll N^{-1/(k-1)},\quad t \gg \sigma,
  \quad \text{ and } \quad
  \sqrt{t} \log (1/p) \gg \mu \cdot \Po(t/\mu);
\]
Note that these assumptions ensure that $\sigma = (1+o(1))Np^{k/2}$ is bounded away
from zero, so, in particular, we have $t\gg 1$.

As usual, the lower and the upper bounds on the upper-tail probability of $X$
will be proved separately. In \cref{sec:lower-bound-Poisson}, we prove the
lower bound by applying~\cref{prop:lower-bound} to a probability measure
induced by the union of the random set $\bR$ and (approximately) $t$ random
$k$-APs. In \cref{sec:upper-bound-Poisson}, we prove the matching upper bound
by showing that the $t$-th factorial moment of $X$, restricted to the event
that $\bR$ does not contain a small $\eps t$-seed, is approximately $\mu^t$,
the $t$-th factorial moment of a genuine Poisson random variable with mean
$\mu$.

We recall the definition of the Poisson rate function $\Po \coloneqq [0,\infty) \to [0, \infty)$:
\begin{equation}
  \label{eq:Po-definition}
  \Po(x) \coloneqq  \int_0^x \log(1+y)\,dy = (1+ x) \log(1+x) - x.
\end{equation}
One property of $\Po$ that we will use several times throughout this section is that
\begin{equation}
  \label{eq:Po-xlog1px}
  x \log (1 +x) \le 2\mathrm{Po}(x),
\end{equation}
which follows easily from the defintion of $\Po$ and the inequality $\log(1+y) \ge (y/x) \cdot \log(1+x)$, which is valid for all $y \in [0,x]$.

\subsection{Proof of the lower bound in the Poisson regime}
\label{sec:lower-bound-Poisson}

In this section, we will prove the lower bound for the Poisson regime:
\begin{proposition}
  \label{prop:PoissonLowerBound}
  Suppose that the sequence $(p,t)$ satisfies
  \[
    \Omega(N^{-2/k}) \le p \ll N^{-1/(k-1)}, \quad t  \gg \sigma,
    \quad \text{ and } \quad
    \sqrt{t} \log (1/p) \gg \mu \cdot \Po(t/\mu).
  \]
  Then, for any $\eps >0$ and $N$ large enough 
  \[
    \log \Pr(X > \mu +t ) > -(1 + \eps)\cdot  \mu \cdot \Po(t /\mu).
  \]
\end{proposition}
We prove this proposition by
applying the tilting argument to a measure $\Prh$ which is not close to any
product measure.  Instead, the random set $\bR$ will be the union of $\br{N}_p$
and approximately $t$ random $k$-term arithmetic progressions. To be more
precise, for any positive integer $u$, let $\ffAPk{u}$ be the set containing
all sequences of $u$ distinct elements of $\APk$. We define $\Pr_u$ to be
the measure corresponding to an independent sample from $\br{N}_p$ and a
uniformly chosen element from $\ffAPk{u}$; that is, for any $S \subseteq \br{N}$
and $(\bA_1, \dotsc, \bA_u) \in \ffAPk{u}$,
\[
\Pr_u(S, (\bA_1, \dotsc, \bA_u)) = \frac{p^{|S|} (1-p)^{N - |S|}}{|\ffAPk{u}|}.
\]
We now let $\Prh_u$ be the marginal of $\Pr_u$ onto the union $S \cup \bA_1
\cup \dotsb \cup \bA_u$. A straightforward computation shows that, for any $R
\subseteq \br{N}$,
\begin{equation}
  \label{eq:PoissonSprinkling}
  \Prh_{u}(\bR = R) =  \sum_{\substack{ (A_1, \dotsc, A_{u}) \in \ffAPk{u}   \\ A_1 \cup \dotsb \cup A_{u} \subseteq R}} \frac{1}{|\ffAPk{u}| } \cdot p^{|R \setminus (A_1 \cup \dotsb \cup A_u)|} \cdot (1-p)^{N - |R|}.
\end{equation}

As before, our argument has two parts. First, we will show that the upper-tail
event $\{X \ge \mu + t\}$ is very likely to occur under the measure $\Prh_u$,
for an appropriately chosen $u \approx t$, and thus the logarithmic upper-tail
probability can be bounded from below by $-(1+o(1)) \cdot \DKL(\Prh_u \,\|\,\Pr)$.
Second, we will show that $\DKL(\Prh_u \, \| \, \Pr)$ is close to $\mu
\cdot \Po(t/\mu)$, which will be a fairly simple consequence of the fact that
$\Exh_u[X] \le \mu + t + o(t)$, where $\Exh_u$ is the expectation
operator associated with $\Prh_u$.

\begin{lemma}
  \label{lemma:PoissonTiltWorks}
  For every $\eps >0$ and $k\ge 3$, there exists $C>0$ such that the following holds for all $p
  \in (N^{-2/k}/C, N^{-1/(k-1)}/C)$ and all $t \in (C^k
  \sqrt{\mu},N^{1-1/(k-1)}/C)$.  Letting \begin{equation}
    \label{eq:u-PoissonTilt}
    u \coloneqq \left\lceil t + C \sqrt{\mu} \right\rceil,
  \end{equation}
  we have
  \[
    \Prh_{u}( X \ge \mu + t ) \ge 1 - \eps
    \qquad
    \text{and}
    \qquad
    \Exh_u[X] \le \mu + (1+\eps)u.
  \]
\end{lemma}
\begin{proof} Let $Z$ count the progressions among $\bA_1, \dotsc, \bA_u$ that
  are contained in $S$ and let $Y$ count the progressions contained in $S \cup
  \bA_1 \cup \dotsb \cup \bA_u$ that are neither fully contained in $S$ nor
  belong to the set $\{\bA_1, \dotsc, \bA_u\}$.  It is straightforward to see
  that
  \begin{equation}
    \label{eq:XR-Poisson-bounds}
    A_k(S) + u - Z \le X \le A_k(S) + u + Y.
  \end{equation}

  Recall that, under $\Prh$, the random set $\bR$ is the union of $S$ and the $\bA_i$. The first inequality in~\eqref{eq:XR-Poisson-bounds} and the union bound yield
  \[
    \Prh_u (X \le \mu + t) \le \Pr_u (A_k(S) + u - Z  \le \mu + t) \le \Pr_u\big(  A_k(S) \le \mu - (u - t)/2\big)  + \Pr_u\big( Z \ge (u-t)/2\big).
  \]
  Since the marginal of $S$ under $\Pr_u$ is $\Pr$, the variance of
  $A_k(S)$ under $\Pr_u$ is equal to $\sigma^2$, the
  variance of $X$ under $\Pr$. Thus, we may use Chebyshev's
  inequality to conclude that
  \[
    \Pr_u\big(A_k(S) \le \mu - (t - u)/2\big) \le \frac{4
    \sigma^2}{(t-u)^2} \le \frac{8 \mu}{C^2 \mu} ,
  \]
  where the final bound follows from the estimate $\sigma^2 \le (1 + \eps) \mu$,
  which, in turn, is a consequence of our bounds on $p$. Meanwhile, since
  $\Ex_u[Z] = u p^k $, where $\Ex_u$ is the expectation operator associated with
  $\Pr_u$, a straightforward application of Markov's inequality yields
  \[
    \Pr_u( Z \ge (u-t)/2) \le \frac{2\Ex_u[Z]}{ C \sqrt{\mu}} \le
    \frac{2up^k}{C\sqrt{N^2p^k/(2k)}} \le \frac{8ktp^{k/2}}{CN} \le
    \frac{8k}{C},
  \]
  where we use the inequalities $u \le 2t \le N$, which follow from the
  definition of $u$ and the assumption on $t$.  Choosing $C$ large enough
  yields the first assertion of the lemma.

  In light of the second inequality in~\eqref{eq:XR-Poisson-bounds} and the
  fact that $\Ex_u[A_k(S)] = \Ex[X] = \mu$, the second assertion of the lemma
  follows from the inequality $\Ex_u[Y] \le \eps u$, which we will establish in
  the remainder of the proof.  To this end, note that $Y$ can be written as a
  sum, over all $B \in \APk$, of the indicators of the event $\cE_B$ that $B
  \subseteq S \cup \bA_1 \cup \dotsb \cup \bA_u$, but $B$ is neither fully
  contained in $S$ nor is it equal to any of the $\bA_i$.  Since $|\APk| \le
  \binom{N}{2}$, it will be sufficient to prove that $\Pr_u(\cE_B) \le \eps
  u/N^2$ for every $B \in \APk$.  For the remainder of the proof, we will fix
  an arbitrary progression $B \in \APk$ and write $\cE$ in place of $\cE_B$.

  For each $j \in \br{u}$, let $\cB_j$ and $\cD_j$ be the events that $B \cap
  \bA_j \neq \emptyset$ and $|B \cap \bA_j| \ge 2$, respectively.  The crucial
  observation is that, in order for $\cE$ to occur, at least one of the
  following must occur:
  \begin{itemize}
  \item
    $\cD_j$ occurs for some $j$ and either $B \cap S \neq \emptyset$ or
    $\cB_\ell$ occurs for some $\ell \neq j$; or
  \item
    there is an $r \in \{1, \dotsc, k\}$ such that $|B \cap S| = k-r$ and
    $\cB_j$ occurs for $r$ distinct indices $j$.
  \end{itemize}
  We shall denote the first event by $\cF$ and the second by $\cT$.  As the
  above observation implies that $\Pr_u(\cE) \le \Pr_u(\cF) + \Pr_u(\cT)$, it
  sufficies to bound the probabilities of these two events.

  Recall that $|\APk| = (1+o(1)) \cdot N^2/(2(k-1))$ and that there are at most
  $kN$ progressions in $\APk$ that contain any given element of $\br{N}$ and at
  most $\binom{k}{2}$ such progressions that contain any given pair of elements
  of $\br{N}$.
  As $|B\cap S|$ follows a binomial distribution $\text{Bin}(k,p)$ and is
  independent of all the events $\{\cB_j\}_j$ and $\{\cD_j\}_j$, we find that
  \[
    \begin{split}
      \Pr_u(\cF) & \le \sum_{j=1}^u \Pr_u\big(\cD_j \cap  \{B \cap S \neq \emptyset\}\big) + \sum_{\substack{j, \ell \in \br{u} \\ j  \neq \ell}} \Pr_u \big(\cD_j \cap \cB_\ell \big) \\
      & \le u \cdot \frac{\binom{k}{2}^2}{|\APk|} \cdot kp + u^2 \cdot \frac{\binom{k}{2}^2 \cdot k^2N}{|\APk|(|\APk|-1)} \le \frac{k^6up}{N^2} + \frac{2k^8u^2}{N^3},
    \end{split}
  \]
  provided that $C$ is sufficiently large (and thus $N$ is sufficiently large).
  The inequalities $u \le 2t \le 2N/C$ and $p < 1/C$ for a large constant $C$
  imply that $\Pr_u(\cF) \le \frac{\eps u}{2N^2}$.
  Similarly,
  \[
    \begin{split}
      \Pr_u(\cT) & \le \sum_{r=1}^k \Pr_u(|B \cap S| = k-r) \cdot
      \Pr_u(\text{$\cB_j$ occurs for $r$ distinct $j \in \br{u}$}) \\
      & \le \sum_{r=1}^k (kp)^{k-r}
      \binom{u}{r}\frac{\ff{(k^2N)}{r}}{\ff{|\APk|}{r}} \le (kp)^k \cdot
      \sum_{r=1}^k \left(\frac{2k^2u}{Np}\right)^r.
    \end{split}
  \]
  In order to bound the right-hand side of the above inequality, we consider
  two cases.  If $2k^2u \le Np$, then
  \[
    \Pr_u(\cT) \le p^k \cdot k^{k+1} \cdot \frac{2k^2u}{Np} = 2k^{k+3}Np^{k-1} \cdot
    \frac{u}{N^2} \le \frac{\eps u}{2N^2},
  \]
  by our assumption that $p \le N^{-1/(k-1)}/C$.  Otherwise, if $2k^2u > Np$, then
  \[
    \Pr_u(\cT) \le p^k \cdot k^{k+1} \cdot \left(\frac{2k^2u}{Np}\right)^k =
    \frac{2^kk^{3k+1}u^{k-1}}{N^{k-2}} \cdot \frac{u}{N^2} \le \frac{\eps
    u}{2N^2},
  \]
  by our assumption that $u \le 2t \le 2N^{1-1/(k-1)}/C$.
\end{proof}

\begin{proof}[Proof of~\cref{prop:PoissonLowerBound}] 
By \cref{prop:lower-bound} and \cref{lemma:PoissonTiltWorks}, it is sufficient to show that,
for every fixed $\eps>0$,
\begin{equation}
  \label{eq:PoissonLowerGoal}
  \DKL(\Prh_u \, \| \, \Pr) \le (1+5\eps) \mu \cdot \Po(u/\mu)
  \qquad
  \text{and}
  \qquad
  \Po(u/\mu) \le (1+ \eps) \cdot \mathrm{Po}(t/\mu),
\end{equation}
where $u$ is the integer defined in~\eqref{eq:u-PoissonTilt}.  To this end,
recall~\eqref{eq:PoissonSprinkling} and note that, for every $R \subseteq
\br{N}$,
\begin{equation}
  \label{eq:PoissonSprinklingRN}
  \frac{\Prh_{u}(\bR = R)}{\Pr(\bR = R)} = \frac{1}{|\ffAPk{u}|}
  \sum_{\substack{ (A_1, \dotsc, A_u) \in \ffAPk{u}  \\ A_1 \cup \dots \cup
  A_{u} \subseteq R}} p^{-|A_1 \cup \dotsb \cup A_u|}
  \le \frac{\ff{A_k(R)}{u}}{|\ffAPk{u}| \cdot p^{ku}}.
\end{equation}
Consequently,
\[
  \begin{split}
    \DKL(\Prh_u \, \| \, \Pr)
    & = \sum_{R \subseteq \br{N}} \Prh_u(\bR = R) \log \frac{\Prh_u(\bR =
    R)}{\Pr(\bR = R)} \le \Exh_u[\log \ff{X}{u}] - \log
    \left(|\ffAPk{u}| \cdot p^{ku} \right).
\end{split}
\]

A straightforward calculation shows that  
\[
  |\ffAPk{u}| \cdot p^{ku} = \mu^u \cdot \frac{\ff{|\APk|}{u}}{|\APk|^u} \ge \mu^u \cdot \left(1-\frac{u}{|\APk|}\right)^u.
\]
Since $|\APk|  \gg \mu$, we can conclude that
\[
  - \log \left(|\ffAPk{u}| \cdot p^{ku} \right)
  \le -u\log \mu + \frac{\eps}{2} u \log\left(1 + \frac{u}{\mu}\right)
  \le -u\log \mu +  \eps \mu \cdot \Po\left(\frac{u}{\mu}\right),
\]
where the last inequality holds by~\eqref{eq:Po-xlog1px}.

Further, since the function $[u, \infty) \ni x \mapsto \log \ff{x}{u}$ is
concave, Jensen's inequality and \cref{lemma:PoissonTiltWorks} imply that
\[
  \begin{split}
    \Exh_u[\log \ff{X}{u}] & \le \log \ff{(\Exh_u[X])}{u} \le
    \log  \ff{(\mu  + (1+\eps)u)}{u} = \sum_{i=1}^u \log(\mu+\eps u + i) \\
    & \le \int_{\mu}^{\mu + u} \log(x+\eps u+1) \, dx =
    \int_{\mu}^{\mu+u} \log x \, dx + \int_{\mu}^{\mu+u} \log\left(1+\frac{\eps
    u+1}{x}\right) \, dx.
  \end{split}
\]
It follows from~\eqref{eq:Po-definition} that $\mu \cdot \Po(u/\mu) =
\int_{\mu}^{\mu+u} \log x \, dx - u \log \mu$; so we conclude that
\begin{equation}
  \label{eq:DKL-Prhu-Pr-final}
  \DKL(\Prh_u \, \| \, \Pr) \le
  (1+\eps)\mu\cdot\Po\left(\frac{u}{\mu}\right)
  + \int_{\mu}^{\mu + u} \log\left(1+\frac{\eps u + 1}{x}\right) \, dx.
\end{equation}

Further, as $\log(1+y) \le y$ for every $y > -1$ and $u \ge t \ge 1/\eps$,
\[
  \int_{\mu}^{\mu+u} \log \left(1 + \frac{\eps u + 1}{x}\right) \, dx \le \int_{\mu}^{\mu+u} \frac{2\eps u}{x} \, dx = 2\eps u \cdot \log\left(1+\frac{u}{\mu}\right) \le 4\eps\mu \cdot \Po\left(\frac{u}{\mu}\right),
\]
where the last inequality is again~\eqref{eq:Po-xlog1px}. Substituting this
estimate into~\eqref{eq:DKL-Prhu-Pr-final} yields the first inequality
in~\eqref{eq:PoissonLowerGoal}. Finally,
\[
  \Po\left(\frac{u}{\mu}\right) - \Po\left(\frac{t}{\mu}\right) =
  \int_{t/\mu}^{u/\mu} \log(1+x) \, dx \le \frac{u-t}{\mu} \cdot
  \log\left(1 + \frac{u}{\mu}\right) \le 2\left(1-\frac{t}{u}\right) \cdot
  \Po\left(\frac{u}{\mu}\right),
\]
where the last inequality follows from~\eqref{eq:Po-xlog1px}. Our assumption
that $t \gg \sigma = \Omega(\sqrt{\mu})$ implies that $t/u\to 1$, giving the
second inequality in~\eqref{eq:PoissonLowerGoal} and completing the proof.
\end{proof}

\subsection{Proof of the upper bound in the Poisson regime}
\label{sec:upper-bound-Poisson}

In this section, we prove an upper bound on the upper-tail probability in the
Poisson regime. Our starting point here is the realisation that it would be
enough to establish the following estimate for all $\eps > 0$:
\begin{equation}
  \label{eq:PoissonUpperGoal}
  \Ex[\ff{X}{t}] \le \mu^t \cdot \exp\big(\eps t \log(1 + t/\mu)\big).
\end{equation}
Indeed, if~\eqref{eq:PoissonUpperGoal} were true, then a simple application of Markov's inequality would yield
\[
  \Pr(X \ge \mu + t) = \Pr(\ff{X}{t} \ge \ff{(\mu+t)}{t}) \le \frac{\Ex[\ff{X}{t}]}{\ff{(\mu+t)}{t}} \le \frac{\mu^t}{\ff{(\mu+t)}{t}} \cdot \exp\big(\eps t \log(1 + t/\mu)\big),
\]
which gives the desired estimate, as $t \log(1+t/\mu) \le 2\mu \cdot \Po(t/\mu)$, see~\eqref{eq:Po-xlog1px}, whereas
\[
  \log\left(\frac{\mu^t}{\ff{(\mu+t)}{t}}\right) \le - \int_0^t \log\left(\frac{\mu+x}{\mu}\right)\, dx = - \mu \cdot \int_0^{t/\mu} \log(1+y) \, dy = - \mu \cdot \Po\left(\frac{t}{\mu}\right).
\]

Unfortunately, the estimate~\eqref{eq:PoissonUpperGoal} is not correct in the
entirety of the Poisson regime. To remedy this, we will define a family of
indicator random variables $\{Z_u\}_{u\in \RR}$ such that, for all sequences
$(p,t)$ falling into the Poisson regime, we have
\[ \Ex[\ff{X}{t} \cdot Z_u ] \le \mu^t \cdot \exp\big( (u + \eps t) \log(1 + Kt/\mu) \big) \]
and
\[ \Pr(Z_u = 0) \le  \exp\left(- (1 - \eps) \Psi^*\big((1 - \eps) u\big) \log(1/p) \right), \]
where $K=K(k)$ is some constant and $\Psi^*$ was defined in \eqref{eq:PsiStarDef}.

We now set $u \coloneqq \eps t$. Then $\Psi^*((1-\eps)u) \geq \sqrt{(1-\eps)u}$
holds by \cref{prop:Psi-Psi-star}; the assumptions of that proposition apply since we
assume $t\gg \sigma =\Omega(Np^{k/2}) \geq \max{}\{1, N^2p^{2k-2}\}$. Therefore,
for some positive $c = c(\eps)$,
\[
  \begin{split}
    \Pr(X \ge \mu + t)
    & \le \Pr(X \cdot Z_{\eps t} \ge \mu +t) + \Pr(Z_{\eps t} =0)\\
    &\le \frac{\Ex[\ff{X}{t} \cdot Z_{\eps t}]}{\ff{(\mu+t)}{t}} + \exp\big( -
    \Psi^*(\eps t/2) \log(1/p) \big) \\
    & \le \frac{\mu^t}{\ff{(\mu+t)}{t}} \cdot \exp\big( 2\eps t \log(1 + Kt/\mu) \big) +
    \exp\left(- c\sqrt{t}  \log(1/p) \right) \\
    & \le \exp\big(- \mu \cdot \Po(t/\mu) + 2K\eps t \log(1 + t/\mu) \big) + \exp\left(- c\sqrt{t}  \log(1/p) \right).
  \end{split}
\]
which gives the desired bound as $t \log (1+t/\mu) \le 2\mu \cdot \Po(t/\mu)$, by~\eqref{eq:Po-xlog1px}, and $\sqrt{t} \log(1/p) \gg \mu \cdot \Po(t/\mu)$ across the Poisson regime.

We now define $Z_u = Z_u(C)$ to be the indicator of the
event that $\bR$ does not contain any set in $\cSb(u,C)$, where $\cSb$ is defined
as in the statement of \cref{prop:localisation}. The estimate
\[
  \Pr(Z_u = 0) \le  \exp\left(- (1 - \eps) \Psi^*\big((1 - \eps) u\big) \log(1/p) \right)
\]
is then immediate from~\cref{prop:localisation}, provided $C$ is
chosen to be sufficiently large.

It remains to prove the following proposition, which is the main
business of this section.

\begin{proposition}
  \label{prop:factorial-moment-no-seed}
  For every $k\geq 3$, there exists a constant $K$ such that the following
  holds for all $C,\eps > 0$.  If $\Omega(N^{-2/k}) \leq p \ll N^{-1/(k-1)}$ and $t$ is
  a positive integer satisfying $t \le \big(\mu \log(1/p)\big)^{2/3} +
  \big(\log(1/p)\big)^3$, then, for all sufficiently large $N$ and all $u \le
  t$,
  \[
    \Ex[\ff{X}{t} \cdot Z_u(C)] \le \mu^t \cdot \exp\big( (u+\eps t) \log(1 + Kt/\mu) \big).
  \]
\end{proposition}

We remark that the upper-bound assumption on $t$ in the proposition indeed holds
in the entire Poisson regime, due to our assumption $\sqrt{t}\log(1/p)\gg \mu
\Po(t/\mu)$. To see this, it helps to distinguish between the cases $t\leq
\mu/2$ and $t>\mu/2$. In the first case,
the inequality $\log(1+x)\geq x-x^2/2$ rather easily implies that
$\Po(t/\mu)\geq t^2/4\mu^2$, and so our assumption results in $t\ll (\mu
\log(1/p))^{2/3}$. In the second case, we can use~\eqref{eq:Po-xlog1px} to get
$\mu\Po(t/\mu)\geq \frac12 t(1+\log(t/\mu)) > t/100$ and similarly obtain $t
\ll (\log(1/p))^2$.

We now continue with the proof of~\cref{prop:factorial-moment-no-seed}. Since
$\ff{X}{t}$ is a sum of indicators of the events $\{A_1 \cup \dotsb \cup A_t
\subseteq \bR\}$ over all ordered sequences $(A_1, \dotsc, A_t)$ of $t$
arithmetic progressions of length $k$, and such an event precludes the event
$\{Z_u = 1\}$ when the union $A_1 \cup \dotsb \cup A_t$ contains a set from
$\cSb(u)$, we have 
\begin{equation}
  \label{eq:Ex-ffXt-Z}
  \Ex[\ff{X}{t} \cdot Z_u] \le \sum_{\substack{(A_1, \dotsc, A_t) \in
  \ffAPk{t} \\ \cP(A_1 \cup \dotsb \cup A_t) \cap \cSb(u) = \emptyset}}
  \Pr\big(A_1 \cup \dotsb \cup A_t \subseteq \bR\big).
\end{equation}

In order to estimate the right-hand side of~\eqref{eq:Ex-ffXt-Z}, we will
analyse the overlap structure of progressions in each sequence in
$\ffAPk{t}$. It is convenient to introduce the notation
\[ \psi(U) \coloneqq \Ex_U[X] - \Ex[X] = \sum_{r=1}^k A_r^{(k)}(U) \cdot (p^{k-r} - p^k), \]
where, as before, $A_r^{(k)}(U)$ is the number of \kap{}s in $\br N$ that intersect $U$ in
precisely $r$ elements. The following superadditivity property will turn out to
be useful.

\begin{lemma}
  \label{lemma:union-seeds}
  Suppose that $U_1, \dotsc, U_j$ are pairwise disjoint subsets of $\br{N}$. Then
  \[ \psi(U_1 \cup \dotsb \cup U_j) \ge \psi(U_1) + \dotsb + \psi(U_j). \]
\end{lemma}
\begin{proof}
  It is enough to prove the claim for two disjoint sets $U_1$ and $U_2$.
  To this end, consider the conditional expectation $Y_1 \coloneqq
  \Ex[A_k(\bR \cup U_1) - A_k(\bR)\mid \bR \setminus U_1]$ and note that $Y_1$
  is an increasing random variable on a product probability space with
  coordinates $\br{N}\setminus U_1$. So $\Ex[Y_1]\leq \Ex_{U_2}[Y_1]$ by
  Harris' inequality. To complete the proof, observe that $\Ex[Y_1] =
  \Ex_{U_1}[X] -\Ex[X] = \psi(U_1)$ and $\Ex_{U_2}[Y_1] = \Ex_{U_1\cup U_2}[X]
  -\Ex_{U_2}[X] = \psi(U_1\cup U_2) - \psi(U_2)$.
\end{proof}

\begin{definition}
  A \emph{cluster} is a set $\fC \subseteq \APk$ that is connected when viewed as a hypergraph.
  We say that a cluster is:
  \begin{itemize}
  \item
    \emph{small} if
    \[
      |\fC| \le 2\big(\log(1/p)\big)^3
      \quad
      \text{and}
      \quad
      |\fC| \le \frac{\eps \log(1/p)}{\log\log(1/p)} \cdot |\bigcup \fC|;
    \]
  \item
    \emph{$L$-bounded}, for a given positive real $L$, if
    \[
      \psi\left(\bigcup \fC\right) \le \frac{L \mu |\fC|}{Kt},
    \]
      where $K>0$ is a sufficiently large constant depending only on $k$.
  \item
    \emph{heavy} if it is neither small nor $1$-bounded.
  \end{itemize}
\end{definition}

For a sequence $\bA = (A_1, \dots, A_t) \in \ffAPk{t}$, we will denote by
$\fB(\bA)$ the collection of all heavy maximal\footnote{With respect to the subset relation.
Such maximal clusters correspond to connnected components of the hypergraph spanned
by $A_1,\dotsc,A_t$.} clusters
in $\{A_1, \dotsc, A_t\}$ and let $U(\bA) \coloneqq \bigcup
\bigcup \fB(\bA)$ be the union of all progressions that belong to some (heavy,
maximal) cluster in $\fB(\bA)$. The maximality ensures that the family $\{\bigcup
\fC\}_{\fC \in \fB(\bA)}$ is a partition of $U(\bA)$ for every $\bA \in
\ffAPk{t}$, so \cref{lemma:union-seeds} implies the following.

\begin{corollary}
  \label{corollary:psi-union-heavy-clusters}
  For every $\bA \in \ffAPk{t}$, we have
  \[
    \psi(U(\bA)) \ge \sum_{\fC \in \fB(\bA)} \psi(\bigcup \fC).
  \]
\end{corollary}

Our next lemma establishes a key property of heavy maximal clusters.

\begin{lemma}
  \label{lemma:HeavyClustersArentBig}
  For every $\bA \in \ffAPk{t}$, we have $U(\bA) \in \cSb(\psi(U(\bA)))$.
\end{lemma} 
\begin{proof}
  For brevity, let us write $U\coloneqq U(\bA)$ and $\fB \coloneqq \fB(\bA)$.
  We can assume without loss of generality that $U$ and $\fB$ are nonempty.
  Observe that $U$ is a $\psi(U)$-seed by definition, so it only remains to
  show that it also satisfies the size constraint from the definition of
  $\cSb(\psi(U))$, namely,
  \[
    \psi(U)\ge C|U| \cdot \max\{1, Np^{k-1}, |U|p^{k-2} \cdot N^{(k-2)(|U|/\psi(U))^{1/(k-1)}}\}.
  \]
  Due to our assumption $Np^{k-1} \ll 1$, which also implies the inequality
  $N^{(k-2)(|U|/\psi(U))^{1/(k-1)}} \le
  N^{(k-2)C^{-1/(k-1)}} \le p^{-1/4}$ when $|U|\leq \Psi(U)/C$
  for a large enough $C$, it is in fact enough to show that
  \[
    \psi(U) \ge C|U| \cdot \max{}\{1, |U| p^{k-9/4}\}.
  \]

  Since $U$ is the union of at most $t$ progressions, we have $|U| \le kt \le k
  \cdot \big( \mu \log(1/p) \big)^{2/3} + k \cdot \big(\log(1/p)\big)^3$.
  Further,
  \[
    \mu^{2/3} \le (N^2 p^k)^{2/3} = \big(Np^{k-1}\big)^{4/3} \cdot p^{(4-2k)/3} \le p^{(4-2k)/3} \le p^{7/3-k},
  \]
  where the penultimate inequality follows from our assumption that $Np^{k-1}
  \le 1$.  Since $p$ vanishes, we may conclude that $|U| p^{k-9/4} \le 1$ for
  all sufficiently large $N$, and therefore it only remains to show that $\psi(U)
  \ge C|U|$.

  To this end, note first that \cref{corollary:psi-union-heavy-clusters} implies that
  \[
    \frac{\psi(U)}{|U|} \ge \frac{\sum_{\fC \in \fB} \psi(\bigcup
    \fC)}{\sum_{\fC \in \fB}|\bigcup \fC|} \ge \min_{\fC \in \fB}
    \frac{\psi(\bigcup \fC)}{|\bigcup \fC|}.
  \]
  Let $\fC \in \fB$ be a cluster that realises this minimum.  Suppose first
  that $|\fC| \le 2\big(\log(1/p)\big)^3$.  Since $\fC$ is heavy, we must have
  $|\fC| / |\bigcup \fC| > \frac{\eps \log(1/p)}{\log \log (1/p)}$. Moreover,
  every progression in $\fC$ contributes $1-p^k$ to $\psi(\bigcup \fC)$, so we may
  conclude that
  \[
    \frac{\psi(\bigcup \fC)}{|\bigcup \fC|} \ge
    \frac{|\fC|\cdot(1-p^k)}{|\bigcup \fC|}\ge
    \frac{\eps \log(1/p)}{\log \log (1/p)} \cdot (1-p^k) \ge C
  \]
  for all large enough $N$. Suppose now that $|\fC| > 2\big(\log(1/p)\big)^3$.
  In this case, we have
  \[
    2\log(1/p)^3 < |\fC| \le t \le \big(\mu \log(1/p) \big)^{2/3} + \big(\log(1/p)\big)^3,
  \]
  which also implies that $\mu/t \ge \sqrt{t} / (3 \log(1/p)) \ge
  \sqrt{\log(1/p) / 3}$.  Finally, since $|\bigcup \fC| \le k |\fC|$ and
  $\fC$ is heavy, and thus not $1$-bounded, we have
  \[
    \frac{\psi(\bigcup \fC)}{|\bigcup \fC|} \ge \frac{\psi(\bigcup
    \fC)}{k|\fC|} \ge \frac{\mu}{Kkt} \ge \frac{\sqrt{\log(1/p)}}{\sqrt{3}Kk} \ge C
  \]
  for all $N$ large enough.
\end{proof}

In view of \cref{lemma:HeavyClustersArentBig}, we may conclude from \eqref{eq:Ex-ffXt-Z} that
\begin{equation}
  \label{eq:ffXt-Zr-psi-leq-u}
  \Ex[\ff{X}{t} \cdot Z_u] \le \sum_{\substack{\bA = (A_1, \dotsc, A_t) \in \ffAPk{t} \\ \psi(U(\bA)) < u}} \Pr\big(A_1 \cup \dotsb \cup A_t \subseteq \bR\big).
\end{equation}
Indeed, if $A_1\cup\dotsb \cup A_t$ does not contain any subset in $\cSb(u)$,
then, in particular, $U(\bA)\notin \cSb(u)$. Thus, either $U(\bA)$ is not a
$u$-seed at all (but as it is a $\psi(U(\bA))$-seed automatically, we then must
have $u> \psi(U(\bA))$) or it is a $u$-seed that is too large to be a member of
$\cSb(u)$: since $U(\bA)\in \cSb(\psi(U(\bA)))$ by the lemma, this implies
$u>\psi(U(\bA))$ as well.

In order to estimate the right-hand side of~\eqref{eq:ffXt-Zr-psi-leq-u}, we
will first partition the set of all sequences $(A_1, \dotsc, A_t) \in \ffAPk{t}$
according to simple statistics of the maximal clusters in $\{A_1, \dotsc, A_t\}$.
To do so, note first that it follows from \cref{lem:IntervalOptimal} that $|\bigcup \fC| \ge |\fC|^{1/2}$
for every cluster $\fC$. In particular, any cluster $\fC$ satisfying
\[
  |\fC| \le \left(\frac{\eps \log(1/p)}{\log\log(1/p)}\right)^2 \eqqcolon s_0
\]
is automatically small (and thus cannot be heavy).
Additionally, set $\xi \coloneqq 1 + \eps/15 $, and define the \emph{weight} of a
heavy cluster $\fC$ to be the smallest integer $\ell$ such that $\fC$ is
$\xi^\ell$-bounded; note that the weight of a heavy cluster is necessarily
positive, as heavy clusters are not $1$-bounded.

\begin{definition}
  Let $\cL$ denote the family of all triples of integer sequences $(c_1, \dotsc, c_t)$, $(b_1, \dotsc, b_t)$, and $(\ell_{s,j} : s \in \br{t}, j \in \br{b_s})$ satisfying:
  \begin{enumerate}[label=(\roman*)]
  \item
    \label{item:cbl-1}
    $0 \le b_s \le c_s$ and $b_s = 0$ when $s \le s_0$;
  \item
    \label{item:cbl-2}
    $\ell_{s,1} \ge \dotsb \ge \ell_{s,b_s} \ge 1$ for all $s$ and $1 \le j \le b_s$; and
  \item
    \label{item:cbl-3}
    $c_1 + 2c_2 + \dotsb + tc_t = t$.
  \end{enumerate}
  Given a triple $(\vc, \vb, \vell) \in \cL$, denote by $\cC(\vc, \vb, \vell)$
  the family of all sequences $(A_1, \dotsc, A_t) \in \ffAPk{t}$ such that
  $\{A_1, \dotsc, A_t\}$ has, for every $s \in \br{t}$, exactly $c_s$ maximal clusters
  containing $s$ progressions, out of which $b_s$ are heavy clusters with weights
  $\ell_{s,1}, \dotsc, \ell_{s,b_s}$.
\end{definition}

The following two key lemmas, whose proofs we postpone to the end of the
section, will be used to bound from above the contribution of sequences from
each collection $\cC(\vc, \vb, \vell)$ to the right-hand side
of~\eqref{eq:ffXt-Zr-psi-leq-u}.  For every $s \ge 1$ and real $L$, let $\cD_s$
denote the collection of all $s$-elements clusters that are small and let
$\cD_{s,L}$ denote the collection of $s$-element clusters that are not small, but
$L$-bounded. 

\begin{lemma}
  \label{lemma:small-clusters-Poisson}
  For every $s \ge 2$, we have
  \[
    D_s \coloneqq \sum_{\fC \in \cD_s} \Pr\left(\bigcup \fC \subseteq \bR\right)
    \le \frac{\mu^s}{t^s} \cdot \frac{\eps t \log(1+t/\mu)}{4^{s+1}}.
  \]
\end{lemma}

\begin{lemma}
  \label{lemma:bounded-clusters-Poisson}
  For every $s \ge 2$ and $L \ge 1$, we have
  \[
    D_{s,L} \coloneqq \sum_{\fC \in \cD_{s,L}} \Pr\left(\bigcup \fC \subseteq \bR\right)
    \le \frac{\mu^s}{t^s} \cdot L^s \cdot \frac{\eps t \log(1+t/\mu)}{4^{s+1}}.
  \]
\end{lemma}

\begin{corollary}
  \label{corollary:sum-cbell-upper}
  For every $(\vc, \vb, \vell) \in \cL$, we have
  \[
    \Sigma(\vc,\vb,\vell) \coloneqq \hspace{-1em}\sum_{(A_1, \dotsc, A_t) \in \cC(\vc, \vb,
    \vell)}\hspace{-1em} \Pr\big(A_1 \cup \dotsb \cup A_t \subseteq \bR\big) \le \mu^t \cdot
    \prod_{s \ge 2} \left(\frac{1}{c_s!} \left(\frac{\eps t \log
    (1+t/\mu)}{2^s}\right)^{c_s} \cdot \prod_{j=1}^{b_s}
  \left(\xi^{\ell_{s,j}}\right)^s\right).
  \]
\end{corollary}
\begin{proof}
  Fix an arbitrary sequence $(A_1, \dotsc, A_t) \in \cC(\vc, \vb, \vell)$ and
  let $\fC_1, \dotsc, \fC_j$ be the maximal clusters in
  $\{A_1, \dotsc, A_t\}$.  Since
  the sets $\bigcup \fC_1, \dotsc, \bigcup \fC_j$ partition $A_1 \cup \dotsb
  \cup A_t$, we have
  \[
    \Pr\big(A_1 \cup \dotsb \cup A_t \subseteq \bR\big) = \prod_{i=1}^j
    \Pr\left(\bigcup \fC_i \subseteq \bR\right).
  \]
  Further, for every collection $\{\fC_1, \dotsc, \fC_j\}$ of distinct clusters
  satisfying $|\fC_1| + \dotsb + |\fC_j|=t$, there are exactly $t!$ sequences
  $(A_1, \dotsc, A_t) \in \ffAPk{t}$ with these clusters.  Since non-heavy
  clusters of $s$ progressions belong to the set $\cD_s \cup \cD_{s,1}$ and
  heavy clusters with weight $\ell$ belong to the set $\cD_{s,\xi^\ell}$, we
  have
  \[
    \begin{split}
      \Sigma(\vc,\vb,\vell) & \le \frac{1}{c_1!} \left(\sum_{A \in \APk} \Pr(A
      \subseteq \bR)\right)^{c_1} \cdot \prod_{s \ge 2} \left(\frac{\big(D_s +
      D_{s,1}\big)^{c_s-b_s}}{(c_s-b_s)!} \cdot \prod_{j=1}^{b_s}
    D_{s,\xi^{\ell_{s,j}}} \right) \cdot t! \\
      & \le \mu^{c_1} \cdot \prod_{s \ge 2} \left(\frac{\big(D_s +
      D_{s,1}\big)^{c_s-b_s} \cdot c_s^{b_s}}{c_s!} \cdot \prod_{j=1}^{b_s}
    D_{s,\xi^{\ell_{s,j}}}\right) \cdot t^{t-c_1}.
    \end{split}
  \]
  By \cref{lemma:small-clusters-Poisson,lemma:bounded-clusters-Poisson}, we have
  \[
    D_s+D_{s,1}\leq \frac{\mu^s}{t^s}\cdot\frac{\eps t \log(1+t/\mu)}{4^s}
  \]
  and
  \[
    D_{s,\xi^{\ell_{s,j}}}\leq \frac{\mu^s}{t^s}\cdot 
    \xi^{s\ell_{s,j}}\cdot \frac{\eps t \log(1+t/\mu)}{4^s}.
  \]
  Using the fact that $t-c_1 = \sum_{s \ge 2} sc_s$, we may conclude that
  \[
    \Sigma(\vc,\vb,\vell)
    \le \mu^t \cdot \prod_{s \ge 2} \left(\frac{c_s^{b_s}}{c_s!}\cdot\left(\frac{\eps t
      \log(1+t/\mu)}{4^s}\right)^{c_s}\cdot \prod_{j=1}^{b_s}
    \left(\xi^{\ell_{s,j}}\right)^s\right).
  \]
  Finally, $b_s = 0$ unless $s \ge s_0 \gg \log t$.
  Since also $b_s\leq c_s\leq t$, it follows that $c_s^{b_s} \leq t^{c_s} \le 2^{sc_s}$ for all $s$,
  which gives the desired inequality.
\end{proof}

The following lemma supplies a lower bound on $\psi(U(\bA))$, valid for all
$\bA \in \cC(\vc,\vb,\vell)$, that features the (logarithm of the) expression
appearing in the upper bound on $\Sigma(\vc,\vb,\vell)$ given in
\cref{corollary:sum-cbell-upper}.

\begin{lemma}
  \label{lemma:psi-UbA-lower-bound}
  If $\bA \in \cC(\vc, \vb, \vell)$ for some $(\vc,
  \vb, \vell) \in \cL$, then
  \[
    \psi(U(\bA)) \ge \frac{1}{\xi^2 \log(1+K t/\mu)} \cdot \sum_{s = 2}^t
    \sum_{j=1}^{b_s} s \ell_{s,j}\log \xi.
  \]
\end{lemma}
\begin{proof}
  Fix some $(\vc, \vb, \vell) \in \cL$ and let $\bA$ be an arbitrary sequence
  in $\cC(\vc, \vb, \vell)$.  Since a heavy cluster with weight $\ell$ is not
  $\xi^{\ell-1}$-bounded, \cref{corollary:psi-union-heavy-clusters} gives
  (writing $\ell(\fC)$ for the weight of a cluster $\fC$)
  \[
    \psi(U(\bA)) \ge \sum_{\fC \in \fB(\bA)} \psi(\bigcup \fC) \ge \sum_{\fC
    \in \fB(\bA)} \frac{\xi^{\ell(\fC)-1}\mu|\fC|}{Kt} = \sum_{s=2}^t
    \sum_{j=1}^{b_s} \frac{\xi^{\ell_{s,j}-1} \mu s}{Kt}.
  \]
  Since the weight $\ell(\fC)$ of every heavy cluster $\fC$ is positive and satisfies
  \[
    |\fC| \cdot (1-p^k) \le \psi(\bigcup \fC) \le \frac{\xi^{\ell(\fC)}\mu |\fC|}{Kt},
  \]
  the asserted inequality will follow once we show that
  \[
    \rho \coloneqq \max\left\{\frac{Kt \ell \log \xi}{\xi^{\ell-1} \mu}
    : \ell \ge 1 \text{ and } \frac{\xi^{\ell}\mu}{Kt} \ge 1-p^k\right\} \le
    \xi^2\log(1+K t /\mu).
  \]
  Since the function $\ell \mapsto \frac{Kt \ell \log \xi}{\xi^{\ell-1}\mu}$ is
  decreasing for $\ell \geq \log{\xi}$, we have $\rho \leq Kt (\log
  \xi)^2/(\xi^{\log \xi -1}\mu)$. Since we can assume without loss of
  generality that $\eps$ is rather small, and $\xi\leq 1+\eps$, this implies,
  say, $\rho \leq Kt/(2\mu)$. In
  particular, if $Kt/\mu \le 2$,
  then $\rho \le \log(1+K t / \mu)$. On the other hand, if $Kt/\mu > 2$,
  then the smallest real solution $\ell_0$ to the constraint $\xi^{\ell}\mu/(Kt) \geq
  1-p^k$ satisfies $\ell_0 = \log(Kt(1-p^k)/\mu)/(\log \xi) \geq 1\geq\log \xi$ and therefore the maximum in the definition of $\rho$ is achieved at $\ell = \lceil \ell_0\rceil$.
  Thus,
  \[
    \rho \le 
    \frac{\xi\log(Kt(1-p^k)/\mu)}{1-p^k} 
    \le \xi^2 \log(1+K t/\mu).\qedhere
  \]
\end{proof}

The upshot of \cref{lemma:psi-UbA-lower-bound} is that each sequence $\bA$ that
appears in the right-hand side of~\eqref{eq:ffXt-Zr-psi-leq-u} belongs to the
family $\cC(\vc, \vb, \vell)$ for some triple $(\vc, \vb, \vell)$ from the set
\[
  \cL_u \coloneqq \left\{(\vc, \vb, \vell) : \sum_{s \ge 2} \sum_{j=1}^{b_s}
  s\ell_{s,j} \log \xi  \le \xi^2 u \log(1+Kt/\mu) \right\}.
\]
Therefore, by \cref{corollary:sum-cbell-upper},
\[
  \begin{split}
    \frac{\Ex[\ff{X}{t} \cdot Z_u]}{\mu^t} & \le \sum_{(\vc,\vb,\vell) \in \cL_u}
  \frac{\Sigma(\vc,\vb,\vell)}{\mu^t}\\
  &\le \underbrace{\sum_{(\vc,\vb,\vell) \in \cL_u} \prod_{s \ge 2}
  \left(\frac{1}{c_s!} \left(\frac{\eps t
  \log(1+t/\mu)}{2^s}\right)^{c_s}\right)}_{E} \cdot \exp\left(\xi^2 u
  \log(1+Kt/\mu)\right),
\end{split}
\]
where
\[
  \begin{split}
    E & \le \prod_{s \ge 2} \left(\sum_{c=0}^{\infty} \frac{1}{c!}
    \left(\frac{\eps t \log (1+t/\mu)}{2^s}\right)^c\right)  \cdot \max_{\vc}
    |\{(\vb,\vell) : (\vc, \vb, \vell) \in \cL_u\}| \\
    & = \exp\left(\frac{\eps t \log(1+t/\mu)}{2}\right) \cdot \max_{\vc}
    |\{(\vb,\vell) : (\vc, \vb, \vell) \in \cL_u\}|.
\end{split}
\]
Finally, we bound the number of pairs $(\vb, \vell)$ such that $(\vc, \vb,
\vell) \in \cL_u$ for a given sequence $\vc$.  To this end, let $\pi$ denote
the (number-theoretic) partition function, for which we will only need the
very crude bound $\pi(n)\leq e^n$.  Note first that, for a given $\vL =
(L_1, \dotsc, L_t)$, the number of pairs $(\vb, \vell)$ such that $\ell_{s,1} +
\dotsb + \ell_{s,b_s} = L_s$ for every $s$ is at most $\prod_s \pi(L_s) \le
\exp\left(\sum_s L_s\right)$.  Second, if $(\vc, \vb, \vell) \in \cL_u$ for
some such pair, then
\[
  \sum_{s \ge 1} L_s = \sum_{s \ge 1} \sum_{j=1}^{b_s} \ell_{s,j} = \sum_{s > s_0} \sum_{j=1}^{b_s} \ell_{s,j} \le \frac{1}{s_0} \sum_{s > s_0} \sum_{j=1}^{b_s} s\ell_{s,j} \le \frac{\xi^2 u \log(1+Kt/\mu)}{s_0 \log \xi} \ll \frac{t \log(1+Kt/\mu)}{\log t},
\]
where the last inequality follows from the assumption that $u \le t$ and the
fact that $s_0 \gg \log(1/p) \geq \Omega(\log N) \geq \Omega(\log t)$. In
particular, each $L_s$ is at most $t$, and it follows that the number $S$ of
admissible sequences $\vL$ satisfies
\[
  S \le t^{\frac{\eps t \log(1+Kt/\mu)}{8\log t}} = \exp\left(\frac{\eps t
  \log(1+Kt/\mu)}{8}\right).
\]
Consequently $E \le \exp\big((2\eps/3) \cdot t\log(1+Kt/\mu)\big)$.  Since $\xi^2 \le 1+\eps/6$, we may finally conclude that
\[
  \Ex[\ff{X}{t} \cdot Z_u] \le \mu^t \cdot \exp\big((u+\eps t)\log(1+Kt/\mu)\big).
\]

\subsubsection{Counting clusters}
\label{sec:counting-clusters}

Both \cref{lemma:small-clusters-Poisson,lemma:bounded-clusters-Poisson} will be
derived from a more general upper bound on the number of clusters $\fC$
expressed in terms of the number of $k$-term arithmetic progressions that
$\bigcup \fC$ is allowed to intersect.  More precisely, for a vector $\va =
(a_1, \dotsc, a_k)$, let
let $\cC_{m,s}(\va)$ denote the set of all clusters $\fC$ with $|\fC|=s$,
$|\bigcup \fC|=m$, and
\[
  A_{\geq i}\left(\bigcup \fC\right) \le a_i \text{ for each $i \in \br{k}$},
\]
where we use the notation $A_{\ge i}(U) \coloneqq A_i^{(k)}(U) +\dotsb+
A_k^{(k)}(U)$.

\begin{proposition}
  \label{proposition:Ex-number-of-clusters}
  The following holds for every $k\geq 3$,  $s \ge 2$,
  $p \in [0,1]$, $x\geq 0$, and vector~$\va \in \RR^k$.  Letting
  $M \coloneqq \max_{i \in \br{k}} \, a_i\cdot p^{k-i}$, we have
  \[
    \sum_{m \ge k+x} |\cC_{m,s}(\va)| \cdot p^m \le \mu \cdot
    \left(\frac{e^2k^2M}{s} \right)^{s-1} \cdot \left(p \cdot \max_{i \in \br{k-1}}
    \left(\frac{a_i}{M}\right)^{1/(k-i)}\right)^x \le \mu \cdot
    \left(\frac{e^2k^2M}{s}\right)^{s-1}.
  \]
\end{proposition}

We postpone the proof of this proposition, showing first how to derive
\cref{lemma:small-clusters-Poisson,lemma:bounded-clusters-Poisson}.

\begin{proof}[{Derivation of \cref{lemma:small-clusters-Poisson} from
  \cref{proposition:Ex-number-of-clusters}}]
  We can assume that $s \le 2(\log(1/p))^3$, as this is the maximum number of
  \kap{}s in a small cluster. We let $\va \in \Nats^k$ be the vector defined by
  \[
    a_1 \coloneqq k^2sN
    \qquad
    \text{and}
    \qquad
    a_i \coloneqq k^4s^2 \text{ for $i \ge 2$}.
  \]
 Since a union of $s$ progressions contains at most $ks$ elements and
  \[
    |A_{\ge1}(U)| \le |U| k N
    \qquad
    \text{and}
    \qquad
    |A_{\ge i}(U)| \le |U|^2k^2 \text{ for $i \ge 2$},
  \]
  every $s$-element cluster $\fC$ belongs to $\bigcup_m \cC_{m,s}(\va)$.
  Gearing up for an application of \cref{proposition:Ex-number-of-clusters}, we
  calculate
  \[
    M \coloneqq \max_{i \in \br{k}}\, a_i \cdot p^{k-i} = \max\left\{k^2sNp^{k-1}, k^4s^2\right\} = k^4s^2,
  \]
  where we used the assumption $Np^{k-1} \ll 1$, and
  \[
    \max_{i \in \br{k-1}} \left(\frac{a_i}{M}\right)^{1/(k-i)} = \max\left\{\left(\frac{N}{k^2s}\right)^{1/(k-1)}, 1\right\} = \left(\frac{N}{k^2s}\right)^{1/(k-1)},
  \]
  as $s \le 2\big(\log(1/p)\big)^3 \le N/k^2$.
  Let $m_s$ denote the smallest cardinality of $\bigcup \fC$ for an $s$-element
  cluster $\fC$ that is small and observe that
  \begin{equation}
    \label{eq:ms-lower}
    m_s \ge \max\left\{k+1, \frac{\log\log(1/p)}{\eps \log(1/p)} \cdot s\right\},
  \end{equation}
  by $s\geq 2$ and the definition of a small
  cluster.  We may now deduce from \cref{proposition:Ex-number-of-clusters} that
  \[
    \begin{split}
      D_s \le \sum_{m \ge m_s} |\cC_{m,s}(\va)| \cdot p^m
      & \le \mu \cdot (k^6e^2s)^{s-1} \cdot
      \left(\frac{Np^{k-1}}{k^2s}\right)^{\frac{m_s-k}{k-1}} \\
      & \le \left(\frac{\mu}{t}\right)^s \cdot \frac{t}{4^{s+1}} \cdot
      \underbrace{16\left(\frac{4k^6e^2ts}{\mu}\right)^{s-1} \cdot
      \left(Np^{k-1}\right)^{m_s/k^2}}_{T}.
  \end{split}
  \]
  Therefore, it suffices to show that $T \le \eps \log(1+t/\mu)$.

  Assume first that $s \le \mu / (4e^3k^6t)$.  Since the function
  $s \mapsto
  (4k^6e^2ts/\mu)^{s-1}$ is decreasing on the interval $\big[2, \mu /
  (4e^3k^6t)\big]$, and since $Np^{k-1} \ll 1$ and $m_s \ge k+1$, we have
  \[
    T \le \frac{64k^6e^2t}{\mu} \cdot \big(Np^{k-1}\big)^{m_s/k^2} \le \frac{t}{\mu} \cdot (Np^{k-1})^{1/k}.
  \]
  Now, if $t \le \mu$, then the claimed inequality follows from the inequality
  $t/\mu \le 2\log(1+t/\mu)$.
  Otherwise, if $t \ge \mu$, then the assumption that $t \le \big(\mu
  \log(1/p)\big)^{2/3} + \big(\log(1/p)\big)^3$ implies that $\mu \le
  \big(2\log(1/p)\big)^3$ and thus $Np^{k-1} = O\big(\sqrt{\mu p^{k-2}}\big) =
  O\big(p^{1/3}\big)$.  Since $t / \mu \le O\big((\log(1/p))^3\big)$ also follows
  from the same assumption on $t$ (as $\mu \geq \Omega(1)$), we may conclude that $T
  \ll 1$, whereas $\log(1+t/\mu) \ge \log 2$.

  Finally, assume that $\mu/(4e^3k^6t) < s \le 2\big(\log(1/p)\big)^3$. In this
  case, our assumption $t \le \big(\mu \log(1/p)\big)^{2/3} +
  \big(\log(1/p)\big)^3$ implies that $\mu =
  O\big(\big(\log(1/p)\big)^{11}\big)$, and thus $Np^{k-1} = O(\sqrt{\mu
  p^{k-2}}) \leq p^{1/3}$. Since also $t / \mu \le
  O\big((\log(1/p))^3\big)$, we obtain
  \[
    T \le \left(O\big((\log(1/p))^6\big)\right)^{s-1} \cdot  p^{m_s/(3k^2)}
    \le \exp\left(7s\log\log(1/p) - m_s/(3k^2) \cdot \log(1/p)\right).
  \]
Recalling from \eqref{eq:ms-lower} that
  $s \log\log(1/p)\leq m_s\big(\eps\log(1/p)\big)$,
  where we can assume without loss of generality that $\eps$ is somewhat small,
  we may conclude that
  \[
    T \le p^{m_s/(4k^2)} \le p^{1/(4k)} \ll \frac1{2\mu} \le \log(1+1/\mu) \le \log(1+t/\mu),
  \]
  as desired.
\end{proof}

\begin{proof}[{Derivation of \cref{lemma:bounded-clusters-Poisson} from
  \cref{proposition:Ex-number-of-clusters}}]
  Let $\va \in \Nats^k$ be the vector defined by
  \[
    a_i \coloneqq \frac{L\mu s}{Kp^{k-i}(1-p)t}.
  \]
  Since every $s$-element cluster $\fC$ that is $L$-bounded satisfies
  \[
    A_{\ge i}(\bigcup \fC) \cdot (p^{k-i} - p^k) \le \psi(\bigcup \fC) \le
    \frac{L\mu s}{Kt} \le a_i \cdot (p^{k-i} - p^k)
  \]
  for all $i \in \br{k}$,  it belongs to $\bigcup_m \cC_{m,s}(\va)$.
  We may thus deduce from \cref{proposition:Ex-number-of-clusters} that
  \[
    D_{s,L} \le \sum_{m} |\cC_{m,s}(\va)| \cdot p^m \le \mu \cdot
    \left(\frac{e^2k^2L\mu}{K(1-p)t}\right)^{s-1} = \frac{\mu^s}{t^s} \cdot
    \left(\frac{e^2k^2L}{K(1-p)}\right)^{s-1} \cdot t \le \frac{\mu^s}{t^s} \cdot
    L^s \cdot \frac{e^{-s} \cdot t}{4^{s+1}},
  \]
  when $K$ is chosen to be sufficiently large as a function of $k$, as we
  assume $L\geq 1$.

  Since $s \ge \big(\eps\log(1/p)/\log\log(1/p)\big)^2 \gg \log(1/p)$
  for every $s$-cluster that is not small, and as $Np^{k-1}\leq 1$
  implies $N^2p^{2k}\ll 1$, we have
  \[
    e^{-s} \le p^k \ll \frac{1}{2\mu} \le \log(1+1/\mu) \le \log(1+t/\mu),
  \]
  which implies the desired inequality.
\end{proof}

We finally move to proving \cref{proposition:Ex-number-of-clusters}.  To do
this, we will need to establish an upper bound on $\cC_{m,s}(\va)$ for all
integers $m$ and $s$ and all vectors $\va$.  Since there is nothing special
here about the family (hypergraph) $\APk$ of arithmetic progressions that
underlies the notion of a cluster, we will prove a more abstract statement that
provides an analogous bound for the number of connected subhypegraphs with a
given edge boundary, which could be of independent interest. The proof of this
theorem can be found in \cref{sec:Appendix}.

Let $\cH$ be a hypergraph. Given a subset $W \subseteq V(\cH)$ and a positive
integer $i$, we denote $\partial^{(i)}_\cH(W) \coloneqq \{ e \in \cH: |e \cap
W| = i \}$.  Further, for a vector $\va = (a_1, \dotsc, a_k) \in \RR^k$, we
define $\cC_{m,s}(\va; \cH)$ to be the set of connected subhypergraphs $\cH'
\subseteq \cH$ with $m$ vertices and $s$ edges that satisfy
$|\partial_\cH^{(i)} (V(\cH'))| + \dotsb + |\partial_{\cH}^{(k)}(V(\cH'))| \le
a_i$ for all $i \in \br{k}$, so that $\cC_{m,s}(\va) = \cC_{m,s}(\va; \APk)$.

\begin{restatable}{theorem}{enumeratingClusters}
  \label{theorem:enumerating-clusters}
  Suppose that $\cH$ is a $k$-uniform hypergaph.  For all $m,s\in \NN$ and $\va =(a_1,\dotsc,a_k)\in \RR^{k}$,
  \[
    |\cC_{m,s}(\va;\cH)| \le e(\cH) \cdot \sum_{\substack{s_1, \dotsc, s_k \ge 0\\ \sum_i
        (k-i)s_i = m-k\\ \sum_i s_i = s-1}} \prod_{i=1}^k \binom{a_i}{s_i}.
  \]
\end{restatable}

\begin{proof}[Proof of \cref{proposition:Ex-number-of-clusters}]
  We work in the hypergraph $\APk$ of \kap{}s in $\br N$.
  Since every cluster of arithmetic progressions naturally corresponds
  to a connected subhypergraph of $\APk$ and since $\partial_{\APk}^{(i)}(U) =
  A_i(U)$ for every $U\subseteq \br N$, one can see that
  \cref{theorem:enumerating-clusters} implies
  \[ |\cC_{m,s}(\va)| \le |\cC_{m,s}(\va; \APk)|
  \le |\APk| \cdot \sum_{\substack{s_1, \dotsc, s_k \ge 0\\ \sum_i
  (k-i)s_i = m-k\\ \sum_i s_i = s-1}} \prod_{i=1}^{k} \binom{a_i}{s_i}.\]

  Fix some $x\geq 0$. Letting
  \[
    \cS \coloneqq \{(s_1, \dotsc, s_k) \in \ZZ_{\ge 0}^k : s_1 + \dotsb + s_k = s-1 \text{ and } \sum_{i=1}^{k} (k-i)s_i \ge x\}, \]
  we then have
  \[
    \begin{split}
      \sum_{m \ge k+x} |\cC_{m,s}(\va)| \cdot p^m
      & \leq |\APk| \cdot p^k \cdot \sum_{(s_1, \dotsc, s_k) \in \cS} \prod_{i=1}^k \binom{a_i}{s_i} \cdot p^{(k-i)s_i}\\
      & \le \mu \cdot \sum_{(s_1, \dotsc, s_k) \in \cS} \prod_{i=1}^k \left(\frac{e a_i p^{k-i}}{s_i}\right)^{s_i}\\
    \end{split}
  \]
  where the convention $0^0=1$ is to be used when $s_i = 0$.  Note that, for
  every $(s_1, \dotsc, s_k) \in \cS$,
  \[
    \prod_{i=1}^k \left(\frac{1}{s_i}\right)^{s_i} = \left(\frac{k}{s-1}\right)^{s-1} \cdot \exp\left(- \sum_{i=1}^k s_i \log\left(\frac{s_i}{(s-1)/k}\right) \right) \le \left(\frac{k}{s-1}\right)^{s-1},
  \]
  since the sum in the expression above is $s-1$ times the (nonnegative)
  Kullback--Leibler divergence of the random variable taking the value $i \in
  \br{k}$ with probability $s_i/(s-1)$ from the uniform element of $\br{k}$.
  This yields the bound
  \[
    \begin{split}
      \sum_{m \ge k+x} |\cC_{m,s}(\va)| \cdot p^m  & \le \mu \cdot \left(\frac{ek}{s-1}\right)^{s-1} \cdot \sum_{(s_1, \dotsc, s_k) \in \cS}
      \prod_{i=1}^k \left(a_ip^{k-i}\right)^{s_i}.
    \end{split}
  \]
  By the definition of $M$, we have $a_ip^{k-i} \le M$ for every $i\in \br k$;
  so for every $(s_1, \dotsc, s_k) \in \cS$, we have
  \[
    \prod_{i=1}^k \left(a_ip^{k-i}\right)^{s_i} \le M^{s-1} \cdot
    \prod_{i=1}^{k-1} \left(\frac{a_ip^{k-i}}{M}\right)^{s_i}
    \le M^{s-1} \cdot \prod_{i=1}^{k-1} \left(
      \max_{j \in \br{k-1}} \left(\frac{a_jp^{k-j}}{M}\right)^{1/(k-j)}
    \right)^{(k-i)s_i}.
  \]
  Since the quantity in the iterated product is at most $1$, we have
  \[
    \sum_{m \ge k+x} |\cC_{m,s}(\va)| \cdot p^m  \le \mu \cdot \left(\frac{ekM}{s-1}\right)^{s-1} \cdot \left(p \cdot \max_{j \in \br{k-1}} \left(\frac{a_j}{M}\right)^{1/(k-j)}\right)^x \cdot |\cS|.
  \]
  The claimed bound now follows after we observe that $|\cS| \le
  \binom{s-1+k-1}{s-1} \le (ek)^{s-1}$.
\end{proof}

\bibliographystyle{abbrv}
\bibliography{refs}

\newpage
\appendix

\section{Proof of the hypergraph cluster lemma}
\label{sec:Appendix}


In this appendix, we prove \cref{theorem:enumerating-clusters}, the bound for
the number of connected subhypergraphs with a given edge boundary, which we
used to control factorial moments of the number of $k$-term arithmetic
progressions in the Poisson and the localised regimes.  We restate the theorem
here for reader's convenience.  Recall that given a set $W$ of vertices of a
hypergraph $\cH$ and a positive integer $i$, we denote $\partial^{(i)}_\cH(W)
\coloneqq \{ e \in \cH: |e \cap W| = i \}$.  Further, for a vector $\va = (a_1,
\dotsc, a_k) \in \RR^k$, we defined $\cC_{m,s}(\va; \cH)$ to be the set of
connected subhypergraphs $\cH' \subseteq \cH$ with $m$ vertices and $s$ edges
that satisfy $|\partial_\cH^{(i)} (V(\cH'))| + \dotsb +
|\partial_{\cH}^{(k)}(V(\cH'))| \le a_i$ for all $i \in \br{k}$.

\enumeratingClusters*

\begin{proof}
  We prove the claim by exhibiting an injective map from $\cC_{n,m}(\va;\cH)$ into a set of combinatorial objects
  that are easier to count. To this end, let $I(\va) \coloneqq \{ (i,j): 1\le i \le k\text{ and } 1\le j \le a_i \}$ and
  let $\cT_{m,s}(\va)$ be the collection of all
  subsets $T\subseteq I(\va)$ such that
  \begin{equation}
    \label{eq:code-condition}
    |T| = s-1 \qquad \text{and} \qquad \sum_{(i,j)\in T} (k-i) = m-k.
  \end{equation}
  It may help to think of $I(\va)$ as representing a grid or tableau consisting of $k$ rows of respective lengths $a_1,\dotsc,a_k$ and of an element of $\cT_{m,s}(\va)$
  as a certain way of marking some $s-1$ cells of the tableau.

  Further below, we will argue that there exists an injective map
  \[
    \text{emb} \colon \cC_{m,s}(\va;\cH) \to E(\cH) \times \cT_{m,s}(\va),
  \]
  which clearly implies that
  \begin{equation}
    \label{eq:code-counting}
    |\cC_{m,s}(\va;\cH)| \le e(\cH) \cdot |\cT_{m,s}(\va)|.
  \end{equation}
  From this, we may complete the proof by counting as follows.
  First, note that, for every $T\in \cT_{m,s}(\va)$, the `row counts' $s_1,\dotsc,s_k$ defined by $s_i \coloneqq |\{(i',j)\in T: i'=i\}|$ satisfy $\sum_i(k-i)s_i = m-k$ and $\sum_i s_i = s-1$. If $T\in \cT_{m,s}(\va)$,
  then $T$ is fully specified by giving $(s_1,\dotsc,s_k)$ and then choosing, for each row $1\le i \le k$, the $s_i$ values $j\in \br{a_i}$ such that $(i,j)\in T$. This gives the bound 
  \[
    |\cT_{m,s}(\va)| \le \sum_{\substack{s_1, \dotsc, s_k \ge 0 \\ \sum_i (k-i)s_i = m-k \\ \sum_i s_i = s-1}} \prod_{i=1}^{k} \binom{a_i}{s_i},
  \]
  which, together with \eqref{eq:code-counting}, gives the assertion of the theorem.
  
  \begin{algorithm}[t]
    \SetAlgoLined
    \caption{The algorithm defining emb\label{alg:encoding}}
    \KwData{A hypergraph $\cH'\in \cC_{m,s}(\va;\cH)$, where $\cH$
    is a $k$-uniform hypergraph}
    \KwResult{A pair $(e,T)$ consisting of an edge $e\in \cH$ and a subset $T\subseteq I(\va)$}

    fix a total ordering $\leH$ on the edges of $\cH$\;
    \For{$i\in \br{k}$}{
      $\sigma_{1}^{(i)}\gets $ the empty sequence (of edges of $\cH$)\;
    }
    $e_1 \gets$ the $\leH$-smallest edge of $\cH'$\;
    $T_1 \gets \emptyset \subseteq I(\va)$\;
    \For{$\ell = 2,\dotsc,s$}{
      \For{$i\in \br{k}$}{
        $\sigma_{\ell}^{(i)}\gets \sigma_{\ell-1}^{(i)}\mathbin\Vert (x_{1},\dotsc,x_r)$, where ${\mathbin\Vert}$ denotes concatenation of finite sequences and
        $x_1,\dotsc,x_r$ are the elements of $\partial_\cH^{(i)}(e_1 \cup \dotsb \cup e_{\ell-1}) \setminus \sigma_{\ell-1}^{(i)}$ ordered according to $\leH$\;
      }
      $(i,j) \gets $ the lexicographically largest $(i,j)\in I(\va)$ such that $(\sigma_\ell^{(i)})_j \in E(\cH')\setminus \{e_1,\dotsc,e_{\ell-1}\}$\;
      $e_\ell \gets (\sigma_\ell^{(i)})_j$\;
      $T_\ell \gets T_{\ell-1} \cup \{(i,j)\}$\;
    }
    \Return{$(e_1,T_m)$}
  \end{algorithm}

  Algorithm~\ref{alg:encoding} takes an element
  $\cH' \in \cC_{n,m}(\va; \cH)$ as input
  and returns a pair $(e,T)$ consisting of an edge $e\in E(\cH)$ and
  a subset $T\subseteq I(\va)$.
  We will argue that one can define emb 
  as the function computed by the algorithm.
  This requires showing the following:
  \begin{enumerate}[label=(\roman*)]
  \item
    each step in the algorithm is well-defined and can indeed be carried out as described;
  \item
    the pair $(e,T)$ computed by the algorithm belongs to $E(\cH)\times \cT_{m,s}(\va)$;
  \item
    the map $\cH' \mapsto (e,T)$ is injective.
  \end{enumerate}

  A first glance at the definition of the algorithm reveals that,
  in addition to computing $(e,T)$, the algorithm
  defines sequences $(e_1,\dotsc,e_s)$, $(T_1,\dotsc,T_s)$, and
  $(\sigma^{(i)}_1,\dotsc,\sigma^{(i)}_s)$
  for every $i\in \br{k}$. At this point, we can already convince
  ourselves, using straightforward induction, that each $e_\ell$ is a distinct edge of $\cH'$, each $T_\ell$ is a
  subset of $I(\va)$, and each $\sigma_\ell^{(i)}$ is a finite sequence
  of distinct edges of $\cH$ (however, one and the same edge might belong to two
  different sequences $\sigma_\ell^{(i)}$ and $\sigma_\ell^{(i')}$).
  
  As far as the well-definedness of the algorithm is concerned,
  the only critical step is on line 11 of the algorithm, where we need to show
  that at least one $(i,j)\in I(\va)$ with $(\sigma_\ell^{(i)})_j \in E(\cH')\setminus \{e_1,\dotsc,e_{\ell-1}\}$ exists.
  Since $\cH'$ is connected and $\ell\le s = e(\cH')$, there
  exists at least one edge in $E(\cH') \setminus \{e_1, \dotsc, e_{\ell-1}\}$ that
  intersects $e_1 \cup \dotsb \cup e_{\ell-1}$, in $i \in \br{k}$ vertices.
  Since $\sigma_\ell^{(i)}$
  contains all edges in $\partial^{(i)}_\cH(e_1 \cup \dotsb \cup e_{\ell-1})$
  by definition,
  there thus exists some $j$ such that $(\sigma_\ell^{(i)})_j \in
  \cH' \setminus \{e_1,\dotsc,e_{\ell-1}\}$. It now suffices to show that any such $(i,j)$ belongs to $I(\va)$. To this end, observe
  that every edge in $\sigma_{\ell}^{(i)}$ belongs to 
  $\partial^{(i)}_\cH(e_1 \cup \dotsb \cup e_{\ell-1}) \cup \dotsb \cup \partial^{(k)}_\cH(e_1 \cup \dotsb \cup e_{\ell-1})$, since at some
  (possibly earlier) iteration it must have intersected the union of a prefix of $(e_1,\dotsc,e_{\ell-1})$ in $i$ vertices. The assumption $\cH'\in
  \cC_{m,s}(\va;\cH)$ then implies that $\sigma_{\ell}^{(i)}$ has length
  at most $a_i$; so we have $(i,j)\in I(\va)$.
  Therefore, a pair $(i,j)$ as on line
  11 of Algorithm~\ref{alg:encoding} really exists, which
  shows that the algorithm is well-defined.

  We now show that the pair $(e,T)$ computed by Algorithm~\ref{alg:encoding}
  belongs to $E(\cH)\times \cT_{m,s}(\va)$. For this, it is sufficient to show
  that for every $1\le \ell \le m$, we have
  \[ |T_\ell| = \ell-1 \qquad\text{and}\qquad
  \sum_{(i,j)\in T_\ell} (k-i) = |e_1\cup \dotsb \cup e_\ell|-k.\]
  Indeed, both statements are easily seen to hold for $\ell=1$. Using induction
  on $\ell$, the first statement now follows from the definition of
  $T_\ell$ on line 13 and the fact that $(\sigma_\ell^{(i)})_j =e_\ell \notin \{e_1,\dotsc,e_{\ell-1}\}$, which implies that a different pair $(i,j)$ is added on each iteration (since the sequences $\sigma_1^{(i)}, \dotsc,\sigma_\ell^{(i)}$ are each a prefix of the next). For the second statement, we observe
  that the maximality of $i$ on line 11, together
  with the fact that $\sigma_\ell^{(i)}$ 
  contains all edges intersecting $e_1\cup \dotsb \cup e_{\ell-1}$ in
  exactly $i$ vertices, implies that the edge $e_\ell$
  intersects $e_1\cup \dotsb \cup e_{\ell-1}$ in \emph{exactly} $i$ vertices.
  Thus, adding $e_\ell$ to the list of edges adds precisely
  $k-i = \sum_{(i,j)\in T_\ell}(k-i) - \sum_{(i,j)\in T_{\ell-1}}(k-i)$ new vertices to their union.

  Finally, we show that the function computed by the algorithm is injective.
  To this end, suppose that Algorithm~\ref{alg:encoding} was run on two \emph{different} hypegraphs $\cH', \hat{\cH'} \in \cC_{m,s}(\va; \cH)$ and defined:
  \begin{itemize}
  \item
    sequences $(e_\ell)_\ell$, $(T_\ell)_\ell$, and $(\sigma_\ell^{(i)})_\ell$ for every $i \in \br{k}$ while run on $\cH'$;
  \item
    sequences $(\hat{e}_\ell)_\ell$, $(\hat{T}_\ell)_\ell$, and $(\hat{\sigma}_\ell^{(i)})_\ell$ for every $i \in \br{k}$ while run on $\hat{\cH'}$.
  \end{itemize}
  Since $e_1, \dotsc, e_s$ and $\hat{e}_1, \dotsc, \hat{e}_s$ are orderings of all edges of $\cH'$ and $\hat{\cH'}$, respectively, they must differ in at least one coordinate;  let $\ell$ be the smallest such coordinate.  If $\ell=1$, then $e_1 \neq \hat{e}_1$ and the two outputs are clearly different.  We will thus assume that $\ell > 1$.  By the minimality of $\ell$, we have $(e_1, \dotsc, e_{\ell-1}) = (\hat{e}_1, \dotsc, \hat{e}_{\ell-1})$.  Observe that this implies that $(T_1, \dotsc, T_{\ell-1}) = (\hat{T}_1, \dotsc, \hat{T}_{\ell-1})$ and that $(\sigma_1^{(i)}, \dotsc, \sigma_{\ell}^{(i)}) = (\hat{\sigma}_1^{(i)}, \dotsc, \hat{\sigma}_{\ell}^{(i)})$ for every $i \in \br{k}$.

  Let $(i,j)$ and $(\hat{i}, \hat{j})$ be the pairs chosen in line~11 of the $\ell$-th iteration of the main for loop during the two respective executions of the algorithm.  Since $(\sigma_\ell^{(i)})_j = e_\ell \neq \hat{e}_{\ell} = (\hat{\sigma}_\ell^{(\hat{i})})_{\hat{j}} = (\sigma_{\ell}^{(\hat{i})})_{\hat{j}}$, it must be that $(i,j) \neq (\hat{i}, \hat{j})$;  without loss of generality, we may assume that $(i, j)$ is lexicographically larger.  This means that $e_\ell =  (\sigma_\ell^{(i)})_j = (\hat{\sigma}_\ell^{(i)})_j \notin E(\hat{\cH}') \setminus \{\hat{e}_1, \dotsc, \hat{e}_{\ell-1}\}$, by maximality of $(\hat{i}, \hat{j})$.  Finally, since $(\hat{\sigma}^{(i)}_{\ell'})_j = (\hat{\sigma}^{(i)}_{\ell})_j$ for every $\ell' \ge \ell$, we may conclude that $(i,j)$ cannot be added to $\hat{T}_{\ell'}$ for any $\ell' \ge \ell$.  In particular, as $(i, j) \notin T_{\ell-1} = \hat{T}_{\ell-1}$, we conclude that $(i,j) \in T_s \setminus \hat{T}_s$.  This completes the proof.
\end{proof}

\end{document}